\documentclass[a4paper,oneside,11pt]{article}
\usepackage{amsmath}
\usepackage{amssymb}
\usepackage{amsthm}
\usepackage{MnSymbol}
\usepackage{enumerate}
\usepackage{bm}
\usepackage{graphicx}
\usepackage[margin = 3cm]{geometry}
\usepackage{color}
\usepackage{url}

\newcommand{\N}{\mathbb{N}}

\newcommand{\R}{\mathbb{R}}

\newcommand{\dHaus}{\, \mathrm{d} \mathcal{H}^{d-1} \,}
\newcommand{\dd}{\, \mathrm{d} \,}
\newcommand{\dz}{\, \mathrm{dz}\,}
\newcommand{\dx}{\, \mathrm{dx}\,}
\newcommand{\ds}{\, \mathrm{ds}\,}

\newcommand{\pd}{\partial}
\newcommand{\eps}{\varepsilon}
\newcommand{\der}{ \mathrm{D} }

\newcommand{\id}{\, \bm{\mathrm{I}}\,}
\newcommand{\abs}[1]{\left| #1 \right|}
\newcommand{\norm}[1]{\| #1 \|}

\newcommand{\inner}[2]{\langle #1 , #2 \rangle}

\newcommand{\tr}[1]{\mathrm{tr} ( #1 )}
\newcommand{\Laplace}{\Delta}
\newcommand{\surf}{\nabla_{\Gamma}}

\renewcommand{\div}{\, \mathrm{div}\,}

\newtheorem{thm}{Theorem}[section]
\newtheorem{lem}[thm]{Lemma}

\newtheorem{rem}{Remark}[section]

\newtheorem{assumption}{Assumption}[section]
\numberwithin{equation}{section}

\begin{document}

\title{Shape optimization for surface functionals in Navier--Stokes flow using a phase field approach}

\author{Harald Garcke \footnotemark[1] \and Claudia Hecht\footnotemark[1] \and Michael Hinze \footnotemark[2]\footnotemark[3]\and Christian Kahle \footnotemark[2] \and Kei Fong Lam\footnotemark[1]}

\date{\today}

\maketitle

\renewcommand{\thefootnote}{\fnsymbol{footnote}}
\footnotetext[1]{Fakult\"at f\"ur Mathematik, Universit\"at Regensburg, 93040 Regensburg, Germany
({\tt \{Harald.Garcke, Claudia.Hecht, Kei-Fong.Lam\}@mathematik.uni-regensburg.de}).}
\footnotetext[2]{Schwerpunkt Optimierung und Approximation, Universit\"at Hamburg, Bundesstrasse 55, 20146 Hamburg, Germany
({\tt \{Michael.Hinze, Christian.Kahle\}@uni-hamburg.de}).}
\footnotetext[3]{Corresponding author.}

\renewcommand{\thefootnote}{\arabic{footnote}}

\begin{abstract}
We consider shape and topology optimization for fluids which are governed by the Navier--Stokes equations. Shapes are modelled with the help
of a phase field approach and the solid body is relaxed to be a porous medium. The phase field method uses a Ginzburg--Landau functional in order to approximate a perimeter penalization. We focus
on surface functionals and carefully introduce a new modelling variant, show existence of minimizers and derive first order necessary conditions. These conditions are related to classical shape derivatives by identifying the sharp interface limit with the help of formally
matched asymptotic expansions. Finally, we present numerical computations based on a Cahn--Hilliard type gradient descent which demonstrate that the method can be used to solve shape optimization
problems for fluids with the help of the new approach.
\end{abstract}

\noindent \textbf{Key words. } Shape optimization, phase-field method, lift, drag, Navier--Stokes equations.\\

\noindent \textbf{AMS subject classification. } 49Q10, 49Q12, 35Q35, 35R35.

\section{Introduction}
Shape optimization problems are a very challenging field in mathematical analysis and has attracted more and more attention in the last decade.  One of the most discussed and oldest problems is certainly the task of finding the shape of a body inside a fluid having the least resistance.  This problem dates back at least to Newton, who proposed this topic in a rotationally symmetric setting.  Nowadays, there are a lot of important industrial applications leading to this kind of questions. Among others we mention in particular the problem of optimizing the shape of airplanes, cars and wind turbine blades in order to have least resistance or biomechanical applications like bypass constructions.  The wide fields of applications may be one of the reasons that shape optimization problems in fluids received growing attention recently.  Nevertheless, those problems turn out to be very challenging and so far no overall mathematical concept has been successful in a general sense.

One of the main difficulties certainly is that shape optimization problems are often not well-posed, i.e., no minimizer exists, compare for instance \cite{book:KawohlCellinaOrnelas00, article:Murat77, incoll:Tartar}.  There are some contributions leading to mathematically well-posed problem formulations, see for instance \cite{article:PlotnikovSokolowski10}, but the geometric restrictions are difficult to handle numerically.  The most common approaches used in practice parametrize the boundary of the unknown optimal shape by functions, see for instance \cite{incoll:BrandenburgLindemannUlbrichUlbrich09, article:Pironneau74}.  However, those formulations do not inherit a minimizer in general.  For numerical simulations typically shape sensitivity analysis is used. Here, one uses local boundary variations in order to find a gradient of the cost function with respect to the design variable, which is in this case the shape of the body. The necessary calculations are carried out without considering the existence or regularity of a minimizer. But in the end one obtains a mathematical structure that can be used for numerical implementations.

In \cite{GarckeHechtNS}, a phase field approach was introduced for minimizing general volume functionals in a Navier--Stokes flow.  For this purpose, the porous medium approach proposed by Borrvall and Petersson \cite{article:BorrvallPetersson03} and a Ginzburg--Landau regularization as in the work of Bourdin and Chambolle \cite{article:BourdinChambolle03} were combined.  The latter is a diffuse interface approximation of a perimeter regularization.  This leads to a model where existence of a minimizer can be guaranteed, and at the same time necessary optimality conditions can be derived and used for numerical simulations, see \cite{garckehinzeetal}.  In particular, this approach replaces the free boundary $\Gamma$ of the body $B$ by a diffuse interface. Hence, it is a priori not clear how to deal with objective functionals that are defined on the free boundary $\Gamma$. 

In this work, we study the following boundary objective functional:
\begin{align}\label{generalfunctional}
\int_{\Gamma} h(x, \nabla \bm{u}, p, \bm{\nu}) \dHaus,
\end{align}
where $h$ is a given function, $\bm{u}$ denotes the velocity field of the fluid, $p$ denotes the pressure, $\bm{\nu}$ is the \textit{inner} unit normal of the fluid region, i.e., pointing from the body $B$ into the complementary fluid region $E = B^{c}$.  The velocity $\bm{u}$ and pressure $p$ are assumed to obey the stationary Navier--Stokes equations inside the fluid region $E$, and the no-slip condition on $\Gamma$, namely,
\begin{subequations}\label{IntroNS}
\begin{alignat}{2}
-\div \bm{\sigma} + (\bm{u} \cdot \nabla) \bm{u} & = \bm{f} && \text{ in } E, \\
\div \bm{u} & = 0 && \text{ in } E, \\
\bm{u} & = \bm{0} && \text{ on } \Gamma,
\end{alignat}
\end{subequations}
where $\bm{\sigma} := \mu \left(\nabla \bm{u} + (\nabla \bm{u})^T \right)- p \id$ denotes the stress tensor of the velocity field $\bm{u}$, $\mu > 0$ denotes the viscosity of the fluid, $\bm{f}$ denotes an external body force, and $\id$ denotes the identity tensor.

An important example of $h$ is the hydrodynamic force component acting on $\Gamma$ with the force direction defined by the unit vector $\bm{a}$:
\begin{align}\label{HydroDynamForce}
h(x, \nabla \bm{u}, p, \bm{\nu}) = \bm{a} \cdot (\bm{\sigma}\bm{\nu}) = \bm{a} \cdot (\mu (\nabla \bm{u} + (\nabla \bm{u})^{T}) - p \id) \bm{\nu},
\end{align}
and so \eqref{generalfunctional} becomes
\begin{align}\label{HydroDynamFunctional}
\int_{\Gamma} \bm{a} \cdot (\bm{\sigma} \bm{\nu}) \dHaus = \bm{a} \cdot \left ( \int_{\Gamma} \bm{\sigma} \bm{\nu} \dHaus \right ).
\end{align}
If $\bm{a}$ is parallel to the direction of the flow, then \eqref{HydroDynamFunctional} represents the drag of the object $B$.  If $\bm{a}$ is perpendicular to the direction of the flow, then \eqref{HydroDynamFunctional} represents the lift of the object.  

In the work at hand we propose an approach on how to deal with boundary objective functionals in the phase field setting.  To be precise, we aim to minimize an appropriate phase field approximation of the functional \eqref{generalfunctional}, and also the functional \eqref{HydroDynamFunctional}, which can be considered as one of the most important objectives in shape optimization in fluids.  The fluid is assumed to be an incompressible, Newtonian fluid described by the stationary Navier--Stokes equations \eqref{IntroNS}.  

For this purpose, we first discuss how we model the integral over the free boundary $\Gamma$ if it is replaced by a diffuse interface and how the normal $\bm{\nu}$ can be defined in this setting, see Section \ref{sec:DerivationPhaseField}.  Afterwards, we analyze the phase field problem for both \eqref{generalfunctional} and \eqref{HydroDynamFunctional} and discuss the existence of a minimizer and optimality conditions, see Section \ref{sec:AnalysisPhaseField}.  In Section \ref{sec:SharpInterfaceAsymp}, we focus on the hydrodynamic force functional \eqref{HydroDynamFunctional} and the corresponding phase field problem is then related to the sharp interface free boundary problem with a perimeter regularization by the method of matched formal asymptotic expansions.  We find that the formal sharp interface limit of the optimality system gives the same results as can be found in the shape sensitivity literature.

We then solve the phase field problem  numerically, see Section
\ref{sec:Numerics}.
For this purpose, we derive a gradient flow equation for the reduced objective functional and
arrive in a Cahn--Hilliard type system.
After time discretization, this system is treated in every time step by a Newton method.
We numerically solve shape optimization problems involving drag and the lift-to-drag ratio.

\section{Notation and problem formulation}\label{sec:SharpInterfaceProblem}
Let us assume that $\Omega \subset \R^{d}$, $d \in \{2,3\}$, is a fixed domain with Lipschitz boundary.  Inside this fixed domain $\Omega$ we may have certain parts filled with fluid, denoted by $E$, and the complement $B:= \overline{\Omega} \setminus E$ is some non-permeable medium.  In the following we will denote by $\bm{\nu}$ the outer unit normal of $B$, i.e., the inner unit normal of the fluid region.  The aim is to minimize the functional, given by \eqref{generalfunctional}, where $\Gamma := \pd B \cap \Omega$, subject to the Navier--Stokes equations \eqref{IntroNS}.  We additionally impose a volume constraint on the amount of fluid. For this purpose we choose $\beta \in (-1,1)$ and only use fluid regions $E \subset \Omega$ fulfilling the constraint $\abs{E}=\frac{(\beta+1)}{2}\abs{\Omega}$.

We prescribe some inflow or outflow regions on the boundary of $\Omega$ and choose for this purpose $\bm{g} \in \bm{H}^{\frac{1}{2}}(\pd \Omega)$ such that $\int_{\pd \Omega} \bm{g} \cdot \bm{\nu}_{\pd \Omega} \; \dHaus = 0$. Additionally, we may have some body force $\bm{f} \in \bm{L}^{2}(\Omega)$ acting on the design domain.  Note that throughout this paper we denote $\R^{d}$-valued functions and spaces consisting of $\R^{d}$-valued functions in boldface.

As already mentioned in the introduction, problems like this are generally not well-posed in the sense that the existence of a minimizer can not be guaranteed.  Hence, we use an additional perimeter regularization.  For this purpose, we add a multiple of the perimeter of the obstacle to the cost functional \eqref{generalfunctional}.  In order to properly formulate the resulting problem we introduce a design function $\varphi:\Omega \to \{ \pm1 \}$, where $\{ \varphi = 1 \} = E$ describes the fluid region and $\{ \varphi = -1 \} = B$ is its complement.  The volume constraint reads in this setting as $\int_{\Omega} \varphi \dx = \beta \abs{\Omega}$.  

The design functions are chosen to be functions of bounded variation, such that the fluid region has finite perimeter, i.e., $\varphi \in BV(\Omega,\{\pm1\})$.  We shall write $P_{\Omega}(E)$ for the perimeter of some set of bounded variation $E \subseteq \Omega$ in $\Omega$. Besides, if $\varphi$ is a function of bounded variation, its distributional derivative $\der \varphi$ is a finite Radon measure and we can define the total variation by $\abs{\der\varphi}(\Omega)$. For $\varphi \in BV(\Omega,\{\pm1\})$, it holds that
\begin{align}\label{HausmeasureRelation}
\abs{\der \varphi}(\Omega)= 2P_{\Omega}(\{\varphi=1\}).
\end{align}  

For a more detailed introduction to the theory of sets of finite perimeter and functions of bounded variation we refer to \cite{book:EvansGariepy, book:Giusti}. We hence arrive in the following space of admissible design functions:
\begin{align}\label{defn:PhiAd0}
\Phi_{ad}^{0}:= \left \{ \varphi \in BV(\Omega,\{\pm 1\}) \mid \int_{\Omega} \varphi \dx = \beta \abs{\Omega} \right \}.
\end{align}


Let $\gamma > 0$ denote the weighting factor for the perimeter regularization.  Then, we arrive at the following shape optimization problem for the functional \eqref{generalfunctional} with additional perimeter regularization:
\begin{align}\label{IntroObjFclt}
\min_{(\varphi,\bm{u}, p)}  J_{0}(\varphi, \bm{u},p) :=  \int_{\Omega} \frac{1}{2} h(x, \nabla \bm{u}, p, \bm{\nu}_{\varphi}) \dd \abs{\der \varphi }+ \frac{\gamma}{2} \abs{ \der \varphi }(\Omega ),
\end{align}
subject to $\varphi \in \Phi_{ad}^{0}$
and $(\bm{u},p) \in \bm{H}^{1}(E) \times L^{2}(E)$ fulfilling 
\begin{subequations}\label{IntroStateEquSharp}
\begin{alignat}{2}
-\mu \Laplace \bm{u} + ( \bm{u} \cdot \nabla) \bm{u} + \nabla p & = \bm{f} &&\text{ in } E = \{ \varphi =1 \},\\
\div \bm{u} & = 0 &&\text{ in } E, \label{NSdiv} \\
\bm{u} & = \bm{g} && \text{ on } \pd \Omega \cap \pd E,\\
\bm{u} & = \bm{0} && \text{ on } \Gamma = \Omega \cap \pd E. \label{NSnoslip}
\end{alignat}
\end{subequations}

Here, we used the relation \eqref{HausmeasureRelation} to replace the perimeter of $E$ with $\frac{1}{2} \abs{\der \varphi}(\Omega)$.  Furthermore, by the polar decomposition 
\begin{align}\label{polardecomp}
\der \varphi = \bm{\nu}_{\varphi} \abs{\der \varphi} \text{ for } \varphi \in BV(\Omega, \{ \pm 1 \}),
\end{align} 
of the Radon measure $\der \varphi$ into a positive measure $\abs{\der \varphi}$ and a $S^{d-1}$-valued function $\bm{\nu}_{\varphi} \in L^{1}\left(\Omega,\abs{\der\varphi} \right)^{d}$, see for instance \cite[Corollary 1.29]{book:Ambrosio}, we replace the product of the normal and the Hausdorff measure in \eqref{HydroDynamFunctional} by $\frac{1}{2}\bm{\nu}_{\varphi} \dd \abs{\der \varphi}$.  In particular, $\bm{\nu}_{\varphi}$ can be considered as a generalised unit normal on $\pd E$. 

We remark that the shape optimization problem \eqref{IntroObjFclt} for the hydrodynamic force component \eqref{HydroDynamForce} have been studied extensively in the literature.  In the work of \cite{article:Bello97}, the boundary integral \eqref{HydroDynamFunctional} is transformed into a volume integral.  This is also done in \cite{incoll:BrandenburgLindemannUlbrichUlbrich_advancedNumMeth_DesignNSflow, article:PlotnikovSokolowski10}, but in the latter, the compressible Navier--Stokes equations are considered.  We also mention \cite{article:Kondoh12}, which utilises the approach of Borrvall and Petersson \cite{article:BorrvallPetersson03} and the volume integral formulation.  The shape derivatives for general volume and boundary objective functionals in Navier--Stokes flow have been derived in \cite{article:SchmidtSchulz10}.  Finally, we mention the work of \cite{incoll:Boisgerault}, which bears the most similarity to our set-up.  Under the assumption that the set $E = \{ \varphi = 1\}$ is $C^{2}$ and that there is a unique, sufficiently regular solution $\bm{u}$ to \eqref{IntroNS}, the analysis of \cite{incoll:Boisgerault} obtained, via the speed method, that the shape derivative of
\begin{align*}
J(E) = \int_{\Gamma} \bm{a} \cdot (\mu (\nabla \bm{u} + (\nabla \bm{u})^{T}) - p \bm{I}) \bm{\nu} \dHaus
\end{align*}
with respect to vector field $V$ is given by (see \cite[Theorem 4, Equation 39]{incoll:Boisgerault}) \footnote{We remark that in \cite{incoll:Boisgerault}, the normal $\bm{n}$ is pointing from the fluid domain to the obstacle, i.e., in comparison with our set-up, $\bm{n} = - \bm{\nu}$.}
\begin{align}\label{ShapeDerivHydroDynamForce}
\der J(E)[V] = \int_{\Gamma} \inner{V(0)}{\bm{\nu}}(\bm{f} \cdot \bm{a} + \mu \pd_{\bm{\nu}} \bm{q} \cdot \pd_{\bm{\nu}} \bm{u}) \dHaus,
\end{align}
where $\bm{q}$ is the solution to the adjoint system (see \cite[Equation 33.2]{incoll:Boisgerault}):
\begin{subequations}\label{IntroAjointEquSharp}
\begin{alignat}{2}
-\mu \Laplace \bm{q} + (\nabla \bm{u})^{T} \bm{q} -  ( \bm{u} \cdot \nabla) \bm{q} + \nabla \pi & = \bm{0} &&\text{ in } E ,\\
\div \bm{q} & = 0 &&\text{ in } E,\\
\bm{q} & = \bm{a} && \text{ on } \Gamma,\\
\bm{q} & = \bm{0} && \text{ on } \pd \Omega \cap \pd E.
\end{alignat}
\end{subequations}
Here, we denote the normal derivative of a scalar $\alpha$ and of a vector $\bm{\beta}$ as
\begin{align}\label{normalderivativevector}
\pd_{\bm{\nu}} \alpha := \nabla \bm{\alpha} \cdot \bm{\nu}, \quad \pd_{\bm{\nu}} \bm{\beta} := (\nabla \bm{\beta}) \bm{\nu}.
\end{align}

We note that as $\bm{u}$ satisfies the no-slip boundary condition \eqref{NSnoslip}, $\bm{u}$ has no tangential components on $\Omega \cap \pd E$.  Thus, we obtain 
\begin{align}\label{decomposition:nablau:surface}
\nabla \bm{u} = \pd_{\bm{\nu}} \bm{u} \otimes \bm{\nu} \text{ on } \Gamma = \Omega \cap \pd E.
\end{align}
Using the divergence free condition \eqref{NSdiv}, and the no-slip condition \eqref{NSnoslip}, we obtain on $\Gamma$:
\begin{align}\label{pdnuUdotnuZero}
0 = \div \bm{u} = \tr{\nabla \bm{u}} = \sum_{i=1}^{d} \pd_{\bm{\nu}} u_{i} \nu_{i} = \pd_{\bm{\nu}} \bm{u} \cdot \bm{\nu} \Longrightarrow (\nabla \bm{u})^{T}  \bm{\nu} = (\pd_{\bm{\nu}} \bm{u} \cdot \bm{\nu}) \bm{\nu} = \bm{0},
\end{align}
which in turn implies that 
\begin{align}\label{HydroDynamForceSim}
J(E) = \int_{\Gamma} \bm{a} \cdot (\bm{\sigma}\bm{\nu})\dHaus = \int_{\Gamma} \bm{a} \cdot (\mu \nabla \bm{u} - p \id)  \bm{\nu} \dHaus.
\end{align}
This is similar to the setting of \cite[Remark 12]{article:SchmidtSchulz10} and 
by following the computations in \cite{article:SchmidtSchulz10} one obtains \eqref{IntroAjointEquSharp} as the adjoint system and the shape derivative of \eqref{HydroDynamForceSim} for a $C^{2}$ domain in the direction of $V$ is \footnote{We remark that in \cite[Remark 12]{article:SchmidtSchulz10} the term $\div_{\Gamma} (\mu (\nabla \bm{u}) \bm{a})$ appears instead of $\div_{\Gamma}( \mu (\nabla \bm{u})^{T} \bm{a})$, which we believe is a typo.}
\begin{equation}\label{SchmidtShapeDerivHydroDynamForce}
\begin{aligned}
\der J(E)[V] & = \int_{\Gamma} \inner{V(0)}{\bm{\nu}} \left ( -\mu \pd_{\bm{\nu}}(\pd_{\bm{\nu}} \bm{u}) \cdot \bm{a} + \pd_{\bm{\nu}} p (\bm{a} \cdot \bm{\nu}) + \mu \pd_{\bm{\nu}} \bm{q} \cdot \pd_{\bm{\nu}} \bm{u} \right ) \dHaus \\
& - \int_{\Gamma} \inner{V(0)}{\bm{\nu}} \div_{\Gamma} \left ( \mu (\nabla \bm{u})^{T} \bm{a} - p \bm{a} \right ) \dHaus,
\end{aligned}
\end{equation}
where $\div_{\Gamma}$ denotes the surface divergence.  We introduce the surface gradient of $f$ on $\Gamma$ by $\nabla_{\Gamma}f$ with components $(\underline{D}_{k}f)_{1 \leq k \leq d}$, and with this definition we obtain $\div_{\Gamma} \bm{v} = \sum_{k=1}^{d} \underline{D}_{k} v_{k}$ for a vector field $\bm{v}$.  Moreover, in components, we have
\begin{align*}
\pd_{\bm{\nu}} (\pd_{\bm{\nu}} \bm{u}) \cdot \bm{a} = \sum_{i,j,k=1}^{d} \nu_{i} \pd_{i}(\nu_{j} \pd_{j} u_{k}) a_{k}.
\end{align*}

\begin{rem}\label{rem:SchmidtBoisgerault}
In \cite[Remark 12]{article:SchmidtSchulz10}, the term $\mu \pd_{\bm{\nu}} (\pd_{\bm{\nu}} \bm{u}) \cdot \bm{a}$ appearing on the right hand side of \eqref{SchmidtShapeDerivHydroDynamForce} is originally given as $\sum_{i,j,k=1}^{d} \nu_{i} \frac{\pd^{2} u_{k}}{\pd x_{i} \pd x_{j}} \nu_{j} a_{k}$.  This is related to $\pd_{\bm{\nu}}(\pd_{\bm{\nu}} \bm{u}) \cdot \bm{a}$ by the formula
\begin{align}\label{interchangederivativesnu}
\sum_{i,j,k=1}^{d} \nu_{i} \frac{\pd^{2} u_{k}}{\pd x_{i} \pd x_{j}} \nu_{j} a_{k} = \pd_{\bm{\nu}} (\pd_{\bm{\nu}} \bm{u}) \cdot \bm{a} - \sum_{i,j,k = 1}^{d} \nu_{i} \pd_{i} \tilde{\nu}_{j} \pd_{j} u_{k} a_{k},
\end{align}
where $\tilde{\bm{\nu}} = (\tilde{\nu}_{j})_{1 \leq j \leq d}$ denotes an extension of $\bm{\nu}$ off the boundary $\Gamma$ to a neighbourhood $U \supset \Gamma$ with $\abs{\tilde{\bm{\nu}}} = 1$ near $\Gamma$ and $\tilde{\bm{\nu}} \mid_{\Gamma} = \bm{\nu}$. 

By \eqref{decomposition:nablau:surface}, we see that $\pd_{j} u_{k} = \pd_{\bm{\nu}} u_{k} \nu_{j}$ on $\Gamma$, and so
\begin{align}
\sum_{i,j,k=1}^{d} \nu_{i} \pd_{i} \tilde{\nu}_{j} \pd_{j} u_{k} a_{k} = \sum_{i,j,k=1}^{d} \nu_{i} \pd_{i} \tilde{\nu}_{j} \nu_{j} \pd_{\bm{\nu}} u_{k} a_{k} = \sum_{i,j,k=1}^{d} \tfrac{1}{2} \nu_{i} \pd_{i} (\abs{\tilde{\nu}_{j}}^{2}) \pd_{\bm{\nu}} u_{k} a_{k} = 0.
\end{align}
Thus, the last term in \eqref{interchangederivativesnu} is zero and we have the relation
\begin{align}\label{relation:Schmidt}
\sum_{i,j,k=1}^{d} \nu_{i} \frac{\pd^{2} u_{k}}{\pd x_{i} \pd x_{j}} \nu_{j} a_{k} = \pd_{\bm{\nu}} (\pd_{\bm{\nu}} \bm{u}) \cdot \bm{a},
\end{align}
when $\bm{u} = \bm{0}$ on $\Gamma$.
\end{rem}
Based on Remark \ref{rem:SchmidtBoisgerault}, if $(\bm{u}, p)$ are sufficiently regular, then a short computation involving \eqref{relation:Schmidt} shows that on $\Gamma$,
\begin{align*}
& \; -\mu \div_{\Gamma} ((\nabla \bm{u})^{T} \bm{a}) - \mu \pd_{\bm{\nu}}(\pd_{\bm{\nu}} \bm{u}) \cdot \bm{a} + \pd_{\bm{\nu}}p (\bm{a} \cdot \bm{\nu}) + \div_{\Gamma} (p \bm{a})\\
= & \; -\mu \sum_{i=1}^{d} \underline{D}_{i} (\pd_{i} u_{j}) a_{j} - \mu \sum_{i,j,k=1}^{d} \nu_{i} \pd_{k} (\pd_{i} u_{j}) \nu_{k} a_{j} + \nabla p \cdot \bm{a} \\
 = & \; -\mu \Laplace \bm{u} \cdot \bm{a} + \nabla p \cdot \bm{a} = \bm{f} \cdot \bm{a} + (\bm{u} \cdot \nabla) \bm{u} \cdot \bm{a} = \bm{f} \cdot \bm{a},
\end{align*}
where we have used the no-slip condition \eqref{NSnoslip}, and hence \eqref{SchmidtShapeDerivHydroDynamForce} is equivalent to \eqref{ShapeDerivHydroDynamForce}.

\section{Derivation of the phase field formulation}\label{sec:DerivationPhaseField}
The problem derived in the previous section has several drawbacks.  First, it is not clear if this is well-posed, i.e., if for every $\varphi \in \Phi_{ad}^{0}$ there is a solution of the state equations \eqref{IntroStateEquSharp} and if there exists a minimizer $(\varphi, \bm{u}, p)$ of the overall problem \eqref{IntroObjFclt}-\eqref{IntroStateEquSharp}.  Second, optimizing in the space $BV(\Omega)$ is not very practical.  Deriving optimality conditions is not easy and it is not clear how to perform numerical simulations on this problem.  Hence, we now want to approximate the complex shape optimization problem \eqref{IntroObjFclt}-\eqref{IntroStateEquSharp} by a problem that can be treated by well-known approaches.  To this end we introduce a diffuse interface version of the free boundary problem by using a phase field approach.  

\subsection{The state equations in the phase field setting}\label{sec:PFState}
In this setting, the design variable $\varphi: \Omega \to \R$ is now allowed to have values in $\R$, instead of only the two discrete values $\pm1$, and inherits $H^{1}(\Omega)$ regularity.  In addition to the two phases $\{\varphi = 1 \}$ (fluid region E) and $\{\varphi = -1\}$ (solid region B), we also have an interfacial region $\{-1< \varphi < 1\}$ which is related to a small parameter $\eps >0$. By \cite{article:Modica87}, we know that the Ginzburg--Landau energy 
\begin{align}\label{defn:GinzburgLandau}
\mathcal{E}_{\eps} : H^{1}(\Omega) \to \R, \quad \mathcal{E}_{\eps}(\varphi):= \int_{\Omega} \frac{\eps}{2} \abs{\nabla \varphi}^{2} \dx+\frac{1}{\eps} \psi(\varphi)\dx
\end{align}
approximates $\varphi \mapsto c_{0} \abs{\der\varphi}(\Omega) = 2c_{0} P_{\Omega}(\{ \varphi = 1 \})$ in the sense of $\Gamma$-convergence.  Here, 
\begin{align}\label{defn:c0}
c_{0} := \frac{1}{2}\int_{-1}^1\sqrt{2\psi(s)} \ds
\end{align}
and $\psi: \R \to \R$ is a potential with two equal minima at $\pm 1$, and in this paper we focus on an arbitrary double-well potential satisfying the assumption below:

\begin{assumption}\label{assump:psi}
Let $\psi \in C^{1,1}(\R)$ be a non-negative function such that $\psi(s) = 0$ if and only if $s \in\{ \pm 1 \}$, and the following growth condition is fulfilled for some constants $c_{1}, c_{2} ,t_{0} > 0$ and $k \geq 2$:
\begin{align*}
c_{1} t^{k} \leq \psi(t) \leq c_{2} t^{k} \quad \forall t \geq t_{0}.
\end{align*}
\end{assumption}

Additionally, we use the so-called porous medium approach for the state equations, see also \cite{GarckeHechtNS, garckehinzeetal}.  This means that, we relax the non-permeability of the solid region $B$ outside the fluid by placing a porous medium of small permeability $(\overline{\alpha}_{\eps})^{-1} \ll 1$ outside the fluid region $E$.  In the interfacial region $\{-1< \varphi <1\}$ we interpolate between the equations describing the flow through the porous medium and the stationary Navier--Stokes equations by using an interpolation function $\alpha_{\eps}$ satisfying the following assumption:
\begin{assumption}\label{assump:alpha}
We assume that $\alpha_{\eps} \in C^{1,1}(\R)$ is non-negative, with $\alpha_{\eps}(1) = 0$, $\alpha_{\eps}(-1)= \overline{\alpha}_{\eps} > 0$, and there exist $s_{a}, s_{b} \in \R$ with $s_{a} \leq -1$ and $s_{b} \geq 1$ such that 
\begin{equation}\label{alphaepsSaSb}
\begin{aligned}
\alpha_{\eps}(s) & = \alpha_{\eps}(s_{a}) \text{ for } s \leq s_{a}, \\
\alpha_{\eps}(s) & = \alpha_{\eps}(s_{b}) \text{ for } s \geq s_{b}.
\end{aligned}
\end{equation}
Moreover, we assume that the inverse permeability vanishes as $\eps \searrow 0$, i.e., $\lim_{\eps \searrow 0} \overline{\alpha}_{\eps} = \infty$. 
\end{assumption}

In particular, we have that 
\begin{align*}
0 \leq \alpha_{\eps}(s) \leq \sup_{t \in [s_{a}, s_{b}]} \alpha_{\eps}(t) < \infty \quad \forall s \in \R,
\end{align*}
i.e., $\alpha_{\eps} \in L^\infty(\R)$.  The resulting state equations for the phase field problem are then given in the strong form by the following system:
\begin{subequations}\label{IntroStateEquPhase}
\begin{alignat}{2}
\label{state1}
\alpha_{\eps}(\varphi) \bm{u} - \mu \Laplace \bm{u} + (\bm{u} \cdot \nabla )\bm{u} + \nabla p &=\bm{f} && \text{ in } \Omega,\\
\label{state2}
\div \bm{u} & = 0 && \text{ in }\Omega,\\
\label{state3}
\bm{u} & = \bm{g} &&\text{ on } \pd \Omega.
\end{alignat}
\end{subequations}
Later we add $\int_{\Omega} \frac{1}{2} \alpha_{\eps}(\varphi)\abs{\bm{u}}^{2}\dx$ to the objective functional and this ensures that in the limit $\eps \searrow 0$, the velocity $\bm{u}$ vanishes outside the fluid region, and hence the medium can really be considered as non-permeable again.

In the following, we will use the following function spaces:
\begin{align*}
\bm{H}^{1}_{0,\sigma}(\Omega):=\left\{\bm{v} \in \bm{H}^{1}_{0}(\Omega) \mid \div \bm{ v} = 0 \right \},\quad \bm{H}^{1}_{\bm{g},\sigma}(\Omega) := \left\{\bm{v} \in \bm{ H}^{1}(\Omega) \mid \bm{v}|_{\pd \Omega} = \bm{g} ,\,\div \bm{v} = 0 \right \},
\end{align*}
and for the pressure we use the space $L^{2}_{0}(\Omega):= \left \{ p \in L^{2}(\Omega)\mid \int_{\Omega} p \dx = 0 \right\}$.  The function space of admissible design functions for the phase field optimization problem will be given correspondingly to \eqref{defn:PhiAd0} as
\begin{align*}
\Phi_{ad}:=\left \{ \varphi \in H^{1}(\Omega) \mid \int_{\Omega} \varphi \dx =\beta \abs{\Omega}\right\}.
\end{align*}

\subsection{The cost functional in the phase field setting}\label{sec:PFCostFunc}
We are now left to transfer the boundary integral in \eqref{IntroObjFclt} to the diffuse interface setting where the free boundary $\Gamma$ is replaced by an interfacial region.  To this end, we apply a result of \cite{article:Modica87} and approximate the perimeter regularization term with $\frac{1}{2c_{0}} \mathcal{E}_{\eps}(\varphi)$.  Meanwhile, keeping in mind the polar decomposition \eqref{polardecomp} and the relation \eqref{HausmeasureRelation}, we consider the vector-valued measure with density $\frac{1}{2} \nabla \varphi$ as an approximation to $\bm{\nu} \dHaus$.  Thus, we may approximate \eqref{IntroObjFclt} with
\begin{align*}
\int_{\Omega} \frac{1}{2} h(x, \nabla \bm{u}, p, \nabla \varphi) \dx + \frac{\gamma}{2 c_{0}} \mathcal{E}_{\eps}(\varphi).
\end{align*} 

Alternatively, we may appeal to the property of equipartition for the Ginzburg--Landau energy, i.e., it holds asymptotically that (see for instance, \eqref{equipartition} in Section \ref{sec:SharpInterfaceAsymp}, or \cite[Section 5.1]{article:Chen96}):
\begin{align*}
\int_{\Omega} \abs{\frac{1}{\eps} \psi(\varphi_{\eps})- \frac{\eps}{2} \abs{ \nabla \varphi_{\eps}}^{2}} \dx \sim 0 \text{ as } \eps \searrow 0.
\end{align*}
Hence, together with \eqref{HausmeasureRelation}, and the fact that $\Gamma$-limit of $\mathcal{E}_{\eps}(\varphi)$ is the functional $c_{0} \abs{ \der \varphi}(\Omega)$, defined for functions with values in $\{\pm 1\}$, and $+\infty$ otherwise, we have loosely speaking
\begin{align}\label{HausdroffPsiApprox}
2c_{0} \mathcal H^{d-1} \lefthalfcup \Gamma \sim c_{0} \abs{\der \varphi} \sim \frac{\eps}{2} \abs{\nabla \varphi}^{2} + \frac{1}{\eps} \psi(\varphi) \sim \frac{2}{\eps}\psi(\varphi),
\end{align}
where $\frac{\eps}{2} \abs{\nabla \varphi}^{2} + \frac{1}{\eps} \psi(\varphi)$ and $\frac{2}{\eps} \psi(\varphi)$ are interpreted as measures on $\Omega$, by using their values as densities.  Here, we have identified $\Gamma = \pd \{\varphi = 1\} \cap \Omega$ with its reduced boundary, then it holds that $\frac{1}{2} \abs{\der \varphi} = \abs{\der\chi_{\{\varphi=1\}}} =\mathcal{H}^{d-1} \lefthalfcup \Gamma$, see for instance \cite[Theorem 3.59]{book:Ambrosio}.

The generalised unit normal $\bm{\nu}$ can be approximated by $\frac{\nabla\varphi}{\abs{\nabla\varphi}}$.  To rewrite this into a more convenient form, which is in particular differentiable with respect to $\varphi$, we use equipartition of energy and replace $\abs{\nabla\varphi}$ by $\frac{1}{\eps}\sqrt{2\psi(\varphi)}$, and obtain the approximation
\begin{align}\label{normalHausdiffuseapprox}
c_{0} \bm{\nu} \mathrm d\mathcal H^{d-1} \sim \eps \frac{\nabla \varphi}{\sqrt{2\psi(\varphi)}} \frac{1}{\eps}\psi(\varphi) \dx = \sqrt{\frac{\psi(\varphi)}{2}} \nabla \varphi  \dx.
\end{align}

Hence, we may also approximate \eqref{IntroObjFclt} with
\begin{align}\label{ObjFctlDerivationPhaseField1}
\frac{1}{c_{0}} \int_{\Omega} \sqrt{\tfrac{\psi(\varphi)}{2}} h(x, \nabla \bm{u}, p, \nabla \varphi) \dx + \frac{\gamma}{2c_{0}} \mathcal{E}_{\eps}(\varphi),
\end{align}
when we extend $h(x, \nabla \bm{u}, p, \cdot)$ from unit vectors to all of $\R^{n}$ such that $h$ is positively one homogeneous with respect to its last variable.  This allows us to extract the factor $\sqrt{\frac{\psi(\varphi)}{2}}$.

We note that in the bulk regions $\{\varphi=\pm1\}$, we have $\psi(\varphi)=0$ and hence the functional \eqref{ObjFctlDerivationPhaseField1} is not differentiable with respect to $\varphi$.  Hence, we add a small constant $\delta_{\eps}$ to $\psi$ in order to have $\psi(s) + \delta_{\eps} > 0$ for all $s \in \R$.  However, we neglect the addition of this constant for the Ginzburg--Landau regularization $\mathcal{E}_{\eps}(\varphi)$ in the objective functional because adding a constant to the cost functional will not change the optimization problem.  

In fact, for the analysis of the phase field problem, it is only important that $\delta_{\eps} > 0$.  In Section \ref{sec:SharpInterfaceAsymp} where we perform a formal asymptotic analysis, we will require $\lim_{\eps \searrow 0} \delta_{\eps} = 0$ at a superlinear rate (see Remark \ref{rem:deltaepslinearscaling}).

\subsection{Optimization problem in the phase field setting}
Combining the above ideas, we arrive in the following phase field approximation:

\begin{equation}\label{IntroObjFcltPhase}
\begin{aligned}
\min_{(\varphi,\bm{u}, p)} J_{\eps}^{h} \left(\varphi,\bm{u}, p \right) := & \int_{\Omega} \frac{1}{2} \alpha_{\eps}(\varphi) \abs{\bm{u}}^{2} +  \frac{1}{2c_{0}} \left (\frac{\eps}{2} \abs{\nabla\varphi}^{2} + \frac{1}{\eps}\psi(\varphi) \right )\dx \\
+ & \int_{\Omega} \mathcal{M}(\varphi)  h(x, \nabla \bm{u}, p, \nabla \varphi) \dx,
\end{aligned}
\end{equation}
subject to $\varphi \in \Phi_{ad}$ and $(\bm{u},p) \in \bm{H}^{1}_{\bm{g},\sigma}(\Omega)\times L^{2}_{0}(\Omega)$ fulfilling
\begin{equation}\label{IntroStateEquPhaseWeak}
\int_{\Omega} \alpha_{\eps}(\varphi)\bm{u}\cdot \bm{v} + \mu \nabla \bm{u} \cdot \nabla \bm{v} + \left(\bm{u} \cdot \nabla\right) \bm{u} \cdot\bm{v} - p \div \bm{v} \dx = \int_{\Omega} \bm{f} \cdot \bm{v} \dx\quad \forall \bm{v} \in \bm{H}^{1}_{0}(\Omega).
\end{equation}

Notice, that \eqref{IntroStateEquPhaseWeak} is a weak formulation of the state equations \eqref{IntroStateEquPhase}.  Moreover, based on the discussions in Section \ref{sec:PFCostFunc}, the function $\mathcal{M}(\varphi)$ can be chosen to be
\begin{align}\label{defn:mathcalL}
\mathcal{M}(\varphi) = \frac{1}{2} \text{ or } \mathcal{M}(\varphi) = \frac{1}{c_{0}} \sqrt{\tfrac{\psi(\varphi) + \delta_{\eps}}{2}}.
\end{align}

The phase field approximation for the shape optimization problem with the hydrodynamic force \eqref{HydroDynamForce} is obtained from \eqref{IntroObjFcltPhase} by substituting
\begin{align*}
h(x, \nabla \bm{u}, p, \nabla \varphi) = \nabla \varphi \cdot (\mu (\nabla \bm{u} + (\nabla \bm{u})^{T}) - p \id) \bm{a}.
\end{align*}
I.e., 
\begin{equation}\label{ObjFunctHydroPhase}
\begin{aligned}
\min_{(\varphi,\bm{u}, p)} J_{\eps} \left(\varphi,\bm{u}, p \right) := & \int_{\Omega} \frac{1}{2} \alpha_{\eps}(\varphi) \abs{\bm{u}}^{2} +  \frac{1}{2c_{0}} \left (\frac{\eps}{2} \abs{\nabla\varphi}^{2} + \frac{1}{\eps}\psi(\varphi) \right )\dx \\
+ &	 \int_{\Omega} \mathcal{M}(\varphi) \nabla \varphi \cdot (\mu (\nabla \bm{u} + (\nabla \bm{u})^{T}) - p \id) \bm{a} \dx,
\end{aligned}
\end{equation}
subject to $\varphi \in \Phi_{ad}$ and $(\bm{u}, p) \in \bm{H}^{1}_{\bm{g}, \sigma}(\Omega) \times L^{2}_{0}(\Omega)$ fulfilling (\ref{IntroStateEquPhaseWeak}).  

\subsection{Possible modifications}

\subsubsection{Double obstacle potential}
We could also use a double obstacle potential $\psi : \R \to \R \cup \{+\infty\}$ instead of the double-well potential in Assumption \ref{assump:psi}, i.e.,
\begin{align}\label{doubleobstacle}
\psi(\varphi)= \begin{cases} \frac{1}{2}(1-\varphi^{2}) & \text{ if } \varphi \in [-1,1], \\
+\infty & \text{ if } \abs{\varphi} > 1.
\end{cases}
\end{align}
Then, one has to treat the constraint $\abs{\varphi} \leq 1$ a.e. in the necessary optimality system either by writing the gradient equation in form of a variational inequality or by including additional Lagrange parameters. Numerical simulations could be implemented by a Moreau-Yosida relaxation as in \cite{garckehinzeetal}. A Moreau-Yosida relaxation also leads to a differentiable double well potential, and here we restrict ourselves to a differentiable potential where both settings can then be included in the above mentioned way.

\subsubsection{Inequality constraint for fluid volume}
Another possible modification of the problem setting would be to replace the equality constraint $\int_{\Omega} \varphi \dx= \beta \abs{\Omega}$ by an inequality constraint $\int_{\Omega} \varphi \dx \leq \beta \abs{\Omega}$. This would make sense in certain settings, if a maximal amount of fluid that can be used during the optimization process is prescribed and not the exact volume fraction. This would not change anything in the analysis, only that the Lagrange multiplier for this constraint would have a sign and an additional complementarity constraint appears in the optimality system.

\subsubsection{Objective functionals with no dependency on the unit normal}
We may also consider objective functionals with no dependence on the normal, i.e., the boundary objective functional \eqref{IntroObjFclt} takes the form
\begin{align}\label{ObjFucNoDep}
\int_{\Gamma} k(x, \nabla \bm{u}, p) \dHaus.
\end{align}
An example of \eqref{ObjFucNoDep} is the best approximation to a target surface pressure distribution in the sense of least squares:
\begin{align*}
k(x, \nabla \bm{u}, p) = \frac{1}{2}\abs{p - p_{d}}^{2},
\end{align*}
where $p_{d}$ denotes the target surface pressure distribution.  Then, using \eqref{HausdroffPsiApprox}, we deduce that the phase field approximation of \eqref{ObjFucNoDep} is given by
\begin{align*}
\frac{1}{c_{0}} \int_{\Omega} \frac{1}{\eps} \psi(\varphi) k(x, \nabla \bm{u}, p) \dx.
\end{align*}
If $k(\cdot, \cdot, \cdot)$ satisfies similar assumptions to Assumptions \ref{assump:generalh} and \ref{assump:regularityh} (see below), one can adapt the proofs of Theorems \ref{t:PhaseFieldExistMin} and \ref{t:GeneralisationOptimality} to obtain existence of a minimiser and the corresponding first order necessary optimality conditions.

\section{Analysis of the phase field problem}\label{sec:AnalysisPhaseField}
In this section we want to analyze the phase field problem \eqref{IntroObjFcltPhase}-\eqref{IntroStateEquPhaseWeak} derived in the previous section as a diffuse interface approximation of the shape optimization problem of minimizing \eqref{generalfunctional} for a Navier--Stokes flow. For this purpose, we introduce some notation for the nonlinearity in the stationary Navier--Stokes equations. We define the trilinear form 
\begin{align*}
b & : \bm{H}^{1}(\Omega) \times \bm{H}^{1}(\Omega) \times \bm{H}^{1}(\Omega) \to \R,  \\
b(\bm{u}, \bm{v}, \bm{w})& := \int_{\Omega}\left(\bm{u} \cdot \nabla \right)\bm{v} \cdot \bm{w} \dx = \sum_{i,j=1}^{d} \int_{\Omega} u_{i} \pd_{i} v_{j} w_{j} \dx.\end{align*}
We directly obtain the following properties:

\begin{lem}\label{l:PropertiesTrilinearForm}
The form $b$ is well-defined and continuous in the space
$\bm{H}^{1}(\Omega) \times \bm{H}^{1}(\Omega) \times \bm{H}^{1}_{0} (\Omega)$.  Moreover we have:
\begin{align}\label{e:ContinuityEstimateTrilinearForm}
\abs{b(\bm{u}, \bm{v}, \bm{w})} \leq K_{\Omega} \norm{\nabla\bm{u}}_{\bm{L}^{2}(\Omega)} \norm{\nabla \bm{v}}_{\bm{L}^{2}(\Omega)} \norm{\nabla \bm{w}}_{\bm{L}^{2}(\Omega)} \quad \forall \bm{u}, \bm{w} \in \bm{H}^{1}_{0}(\Omega), \bm{v} \in \bm{H}^{1}(\Omega),
 \end{align}
with 
\begin{align}\label{KOmega}
K_{\Omega} = \begin{cases}
\frac{1}{2}\abs{\Omega}^{1/2} & \text{ if } d = 2, \\
\frac{2\sqrt{2}}{3}\abs{\Omega}^{1/6} & \text{ if } d = 3.
\end{cases}
\end{align}
Additionally, the following properties are satisfied:
\begin{align}
\label{e:TrilinearformLastTwoEqualZero}
b\left(\bm{u}, \bm{v}, \bm{v}\right) =0 \quad & \forall \bm{u}\in \bm{H}^{1}(\Omega), \div \bm{u} = 0, \quad \bm{v} \in \bm{H}^{1}_{0}(\Omega), \\
\label{e:TrilinearformLastTwoSwitch}
b\left(\bm{u}, \bm{v}, \bm{w}\right) = -b\left(\bm{u}, \bm{w}, \bm{v}\right) \quad & \forall \bm{u} \in \bm{H}^{1}(\Omega), \div \bm{u} = 0 , \quad \bm{v}, \bm{w} \in \bm{H}^{1}_{0}(\Omega).
\end{align}
\end{lem}
\begin{proof}
The stated continuity and estimate \eqref{e:ContinuityEstimateTrilinearForm} can be found in \cite[Lemma IX.1.1]{book:Galdi} and \eqref{e:TrilinearformLastTwoEqualZero}-\eqref{e:TrilinearformLastTwoSwitch} are considered in \cite[Lemma IX.2.1]{book:Galdi}.
\end{proof}

Besides, we have the following important continuity property:
\begin{lem}\label{l:TrilinearFormStrongCont}
Let $(\bm{u}_{n})_{n\in\N}, (\bm{v}_{n})_{n\in\N} , (\bm{w}_{n})_{n \in \N} \subset \bm{H}^{1}(\Omega)$, $ \bm{u},\bm{v}, \bm{w} \in \bm{H}^{1}(\Omega)$ be such that $\bm{u}_{n} \rightharpoonup \bm{u}$, $\bm{v}_{n} \rightharpoonup \bm{v}$ and $\bm{w}_{n} \rightharpoonup \bm{w}$ in $\bm{H}^{1}(\Omega)$ where $\bm{v}_{n}|_{\pd \Omega} = \bm{v}|_{\pd \Omega}$ for all $n \in \N$.  Then 
\begin{align}\label{e:bbilinearformTwoconvergence}
\lim_{n\to\infty} b (\bm{u}_{n}, \bm{v}_{n}, \tilde{\bm{w}}) = b(\bm{u}, \bm{v}, \tilde{\bm{w}}) \quad \forall \tilde{\bm{w}} \in \bm{H}^{1}(\Omega). 
\end{align}
Moreover, one can show that
\begin{align}\label{e:StatNSStrongCong}
\bm{H}^{1}(\Omega) \times \bm{H}^{1}(\Omega) \ni (\bm{u}, \bm{v})\mapsto b(\bm{u}, \cdot, \bm{v}) \in \bm{H}^{-1}(\Omega)
\end{align}
is strongly continuous, and thus 
\begin{align}\label{e:bbilinearformThreeconvergence}
\lim_{n \to \infty} b(\bm{u}_{n}, \bm{v}_{n}, \bm{w}_{n}) = b(\bm{u}, \bm{v}, \bm{w}).
\end{align}
\end{lem}

\begin{proof}
We apply the idea of \cite[Lemma 72.5]{book:Zeidler4} and make in particular use of the compact embedding $\bm{H}^{1}(\Omega)\hookrightarrow \bm{L}^{3}(\Omega)$ and the continuous embedding $\bm{H}^{1}(\Omega) \hookrightarrow \bm{L}^{6}(\Omega)$.  The strong continuity of \eqref{e:StatNSStrongCong} follows from \cite[Lemma 72.5]{book:Zeidler4}.  In addition, from the boundedness of the sequences $(\bm{u}_{n})_{n \in \N}, (\bm{v}_{n})_{n \in \N}, (\bm{w}_{n})_{n \in \N}$, and \eqref{e:StatNSStrongCong}, we have
\begin{align*}
& \; \abs{b(\bm{u}_{n}, \bm{v}_{n}, \bm{w}_{n}) - b(\bm{u}, \bm{v}, \bm{w})} \\
= & \; \abs{b(\bm{u}_{n} - \bm{u}, \bm{v}_{n}, \bm{w}_{n})} + \abs{b(\bm{u}, \bm{v}_{n}, \bm{w}_{n} - \bm{w})} + \abs{b(\bm{u}, \bm{v}_{n} - \bm{v}, \bm{w})} \\ 
\leq & \; \underbrace{\norm{\bm{u}_{n} - \bm{u}}_{\bm{L}^{3}(\Omega)}}_{\xrightarrow{ n \to \infty}0} \underbrace{\norm{\nabla \bm{v}_{n}}_{\bm{L}^{2}(\Omega)} \norm{\bm{w}_{n}}_{\bm{L}^{6}(\Omega)}}_{\leq C} + \underbrace{\norm{\bm{u}}_{\bm{L}^{6}(\Omega)} \norm{\nabla \bm{v}_{n}}_{\bm{L}^{2}(\Omega)}}_{\leq C} \underbrace{\norm{\bm{w}_{n} - \bm{w}}_{\bm{L}^{3}(\Omega)}}_{\xrightarrow{ n \to \infty}0} \\
+ & \; \underbrace{\abs{b(\bm{u}, \bm{v}_{n} - \bm{v}, \bm{w})}}_{\xrightarrow{ n \to \infty} 0 \text{ by } \eqref{e:StatNSStrongCong}}.
\end{align*}
\end{proof}

\subsection{Existence results}
In this section, we want to analyze the solvability of the state equations \eqref{IntroStateEquPhaseWeak}. Afterwards, we will show existence of a minimizer for the overall optimization problem \eqref{IntroObjFcltPhase}-\eqref{IntroStateEquPhaseWeak}.

\begin{lem}\label{l:StateEquSolvable}
Let Assumption \ref{assump:alpha} hold.  Then, for every $\varphi \in L^{1}(\Omega)$ there exists at least one pair $(\bm{u}, p) \in \bm{H}^{1}_{\bm{g},\sigma}(\Omega) \times L^{2}_{0}(\Omega)$ such that the state equations \eqref{IntroStateEquPhase} are fulfilled in the sense of \eqref{IntroStateEquPhaseWeak}.  This solution $(\bm{u},p)$ fulfils the estimate
\begin{align}\label{e:AprioriSolOpE}
\norm{\bm{u}}_{\bm{H}^{1}(\Omega)} + \norm{p}_{L^2(\Omega)} \leq C(\mu,\alpha_{\eps},\bm{f},\bm{g},\Omega),
\end{align}
with a constant $C = C(\mu,\alpha_{\eps}, \bm{f}, \bm{g},\Omega)$ independent of $\varphi$.
\end{lem}

\begin{proof}
We refer to \cite[Lemma 4]{GarckeHechtNS}, where the existence and uniqueness statements for the velocity field $\bm{u}$ are discussed. We point out, that the restriction to functions $\varphi \in L^{1}(\Omega)$ with $\abs{\varphi} \leq 1$ a.e. in $\Omega$ used in \cite{GarckeHechtNS} is only necessary because the function $\alpha_{\eps}$ in \cite{GarckeHechtNS} is only defined on the interval $[-1,1]$. But of course, the same arguments apply to our case where $\alpha_{\eps}$ is bounded and $\varphi \in L^{1}(\Omega)$.

Now for every $\varphi \in L^{1}(\Omega)$ and $\bm{u} \in \bm{H}^{1}_{\bm{g},\sigma}(\Omega)$ fulfilling
\begin{align*}
\int_{\Omega} \alpha_{\eps}(\varphi) \bm{u} \cdot \bm{v} + \mu \nabla \bm{u} \cdot \nabla \bm{v} + (\bm{u} \cdot \nabla) \bm{u} \cdot \bm{v} \dx =\int_{\Omega} \bm{f} \cdot \bm{v} \dx \quad \forall \bm{v} \in \bm{H}^{1}_{0,\sigma}(\Omega),
\end{align*}
we find by \cite[Lemma II.2.1.1]{book:Sohr} a unique $p\in L^{2}_{0}(\Omega)$ such that \eqref {IntroStateEquPhaseWeak} together with
\begin{align*}
\norm{p}_{L^{2}(\Omega)} \leq C(\Omega) \norm{\alpha_{\eps}(\varphi) \bm{u} - \mu \Laplace \bm{u} + (\bm{u} \cdot \nabla) \bm{u} - \bm{f}}_{\bm{H}^{-1}(\Omega)}
\end{align*}
is fulfilled.  Combining this with the previous statements we can conclude the lemma.
\end{proof}

This motivates the definition of a set-valued solution operator 
\begin{align}\label{defn:soluoperator}
\bm{S}_{\eps}(\varphi) : =\{(\bm{u}, p) \in \bm{H}^{1}_{\bm{g},\sigma}(\Omega) \times  L^{2}_{0}(\Omega) \mid (\bm{u}, p) \text{ fulfil }\eqref{IntroStateEquPhaseWeak} \} \text{ for }\varphi\in L^{1}(\Omega).
\end{align}

\begin{rem}\label{rem:uniquenessSolnOperator}
If there is some $\bm{u} \in \bm{S}_{\eps}(\varphi)$ with $\norm{\nabla \bm{u}}_{\bm{L}^{2}(\Omega)} < \frac{\mu}{K_{\Omega}}$, where $K_{\Omega}$ is defined in \eqref{KOmega}.  Then $\bm{S}_{\eps}(\varphi) = \{(\bm{u}, p)\}$.  I.e., there is exactly one solution of \eqref{IntroStateEquPhaseWeak} corresponding to $\varphi$ (see for instance \cite[Lemma 11.5]{thesis:Hecht} or \cite[Lemma 5]{GarckeHechtNS}).
\end{rem}

Moreover, we show a certain continuity property of the solution operator:

\begin{lem}\label{l:SolOpCont}
Under Assumption \ref{assump:alpha}, assume $(\varphi_{k})_{k \in \N}\subset L^{1}(\Omega)$ converges strongly to $\varphi \in L^{1}(\Omega)$ in the $L^{1}$-norm and $(\bm{u}_{k} ,p_{k})_{k \in \N} \subset \bm{H}^{1}(\Omega)\times L^{2}(\Omega)$ are given such that $(\bm{u}_{k}, p_{k}) \in \bm{S}_{\eps}(\varphi_{k})$ for all $k \in \N$.  Then there is a subsequence, which will be denoted by the same, such that $(\bm{u}_{k}, p_{k})_{k \in \N}$	converges strongly in $\bm{H}^{1}(\Omega) \times  L^{2}(\Omega)$ to some element $(\bm{u},p) \in \bm{S}_{\eps}(\varphi)$.
\end{lem}

\begin{proof}
Let $(\varphi_{k})_{k \in \N}$ and $(\bm{u}_{k}, p_{k})_{k \in \N}$ be chosen as in the statement.  By passing to another subsequence, denoted the same, we can without loss of generality assume that $\varphi_{k} \to \varphi$ almost everywhere.  Invoking \eqref{e:AprioriSolOpE}, we obtain a uniform bound on $(\bm{u}_{k}, p_{k})$ in $\bm{H}^{1}(\Omega) \times  L^{2}(\Omega)$ because $(\bm{u}_{k}, p_{k}) \in \bm{S}_{\eps}(\varphi_{k})$.  And so there is a subsequence, which will be denoted by the same, such that $\bm{u}_{k}$ converges weakly in $\bm{H}^{1}(\Omega)$ and strongly in $\bm{L}^{2}(\Omega)$ to some limit element $\bm{u} \in \bm{ H}^{1}_{\bm{g}, \sigma}(\Omega)$ and $p_{k}$ converges weakly in $L^{2}(\Omega)$ to some limit element $p \in L^{2}_{0}(\Omega)$.

We now aim to show that
\begin{align*}
F_{k} & : \bm{H}^{1}_{\bm{g}, \sigma}(\Omega) \to \R, \\
F_{k}(\bm{v}) & := \int_{\Omega} \frac{1}{2}\alpha_{\eps}(\varphi_{k}) \abs{\bm{v}}^{2} +\frac{\mu}{2}\abs{\nabla \bm{v}}^{2} + (\bm{u}_{k} \cdot \nabla) \bm{u}_{k} \cdot \bm{v} - \bm{f} \cdot \bm{v} \dx,
\end{align*}
$\Gamma$-converges in $\bm{H}^{1}_{\bm{g},\sigma}(\Omega)$ equipped with the weak topology to 
\begin{align*}
F_{\infty} &: \bm{H}^{1}_{\bm{g},\sigma}(\Omega) \to \R, \\
F_{\infty}(\bm{v}) & := \int_{\Omega}\frac{1}{2} \alpha_{\eps} (\varphi) \abs{\bm{v}}^{2} + \frac{\mu}{2} \abs{\nabla \bm{v}}^{2} + (\bm{u} \cdot \nabla ) \bm{u} \cdot \bm{v} - \bm{f} \cdot \bm{v} \dx,
\end{align*}
as $k \to \infty$.  To see this we first notice that for any sequence $(\bm{v}_{k})_{k \in \N}\subseteq \bm{H}^{1}_{\bm{g},\sigma}(\Omega)$ converging weakly in $\bm{H}^{1}(\Omega)$ to $\bm{v} \in \bm{H}^{1}_{\bm{g},\sigma}(\Omega)$, by Fatou's lemma it holds that
\begin{align*}
\int_{\Omega} \alpha_{\eps}(\varphi)\abs{\bm{v}}^{2}\dx \leq \liminf_{k \to \infty} \int_{\Omega} \alpha_{\eps}(\varphi_{k})\abs{\bm{v}_{k}}^{2} \dx. 
\end{align*}
Applying the boundedness and continuity properties of the trilinear form $b(\cdot, \cdot, \cdot)$, see Lemma \ref{l:PropertiesTrilinearForm} and \ref{l:TrilinearFormStrongCont}, we can deduce that $\lim_{k \to \infty} b(\bm{u}_{k}, \bm{u}_{k}, \bm{v}_{k}) = b(\bm{u}, \bm{u}, \bm{v})$.  As the remaining terms of $F_k$ are weakly lower semicontinuous in $\bm{H}^1(\Omega)$ and independent of $\varphi_{k}$, we directly obtain 
\begin{align*}
F_{\infty}(\bm{v}) \leq \liminf_{k\to\infty} F_{k}(\bm{v}_{k}).
\end{align*}

Let $\bm{v} \in \bm{H}^{1}_{\bm{g},\sigma}(\Omega)$ be chosen.  We will show, that the constant sequence $(\bm{v})_{k \in \N}$ defines a recovery sequence.  For this purpose, we notice that due to the boundedness and continuity of $\alpha_{\eps}$, we have from Lebesgue's dominated convergence theorem
\begin{align}\label{convergenceAlphaEpsterm}
\lim_{k \to \infty} \int_{\Omega} \alpha_{\eps}(\varphi_{k})\abs{\bm{v}}^{2} \dx = \int_{\Omega} \alpha_{\eps}(\varphi) \abs{\bm{v}}^{2} \dx. 
\end{align}
Invoking \eqref{e:bbilinearformTwoconvergence} in Lemma \ref{l:TrilinearFormStrongCont}, we deduce that
\begin{align*}
\lim_{k \to \infty} b(\bm{u}_{k}, \bm{u}_{k}, \bm{v}) = b(\bm{u}, \bm{u}, \bm{v}),
\end{align*}
and thus, we obtain that $\lim_{k \to \infty} F_{k}(\bm{v}) = F_{\infty}(\bm{v})$.  This shows that the $\Gamma$-limit of $(F_{k})_{k \in \N}$ in $\bm{H}^{1}_{\bm{g}, \sigma}(\Omega)$ with respect to the weak topology equals $F_{\infty}$.

Now we notice, that $\bm{u}_{k}$ is exactly the unique minimizer of $F_{k}$ in $\bm{H}^{1}_{\bm{g},\sigma}(\Omega)$, as it fulfils per definition the necessary and sufficient first order optimality conditions for the convex optimization problem $\min_{\bm{u} \in \bm{H}^{1}_{\bm{g},\sigma}(\Omega)}F_{k}(\bm{u})$.  Hence, the weak $\bm{H}^{1}(\Omega)$ limit of $(\bm{u}_{k})_{k \in \N}$, which is $\bm{u} \in \bm{H}^{1}_{\bm{g},\sigma}(\Omega)$, is the unique solution of $\min_{\bm{u} \in \bm{H}^{1}_{\bm{g},\sigma}(\Omega)} F_{\infty}(\bm{u})$, thus it holds that
\begin{align}\label{e:SEpsContinuousStateEquDivForm}
\int_{\Omega} \alpha_{\eps}(\varphi) \bm{u} \cdot \bm{v} + \mu \nabla \bm{u} \cdot \nabla \bm{v} + (\bm{u} \cdot \nabla) \bm{u} \cdot \bm{v} \dx = \int_{\Omega} \bm{f} \cdot \bm{v} \dx \quad \forall \bm{v} \in \bm{H}^{1}_{0,\sigma}(\Omega).
\end{align}
By \cite[Lemma II.2.1.1]{book:Sohr} we can associate to \eqref{e:SEpsContinuousStateEquDivForm} a unique $\tilde{p} \in L^{2}_{0}(\Omega)$ such that \eqref{IntroStateEquPhaseWeak} is fulfilled, and hence $\tilde{p} = p$.  Altogether we have shown $(\bm{u}, p)\in \bm{S}_{\eps}(\varphi)$.

To show the strong convergence in $\bm{H}^{1}(\Omega) \times  L^{2}(\Omega)$, we note that from the $\Gamma$-convergence of $(F_{k})_{k \in \N}$ to $F_{\infty}$ we obtain additionally that $\lim_{k \to \infty} F_{k}(\bm{u}_{k}) = F_{\infty}(\bm{u})$.  Invoking Lemma \ref{l:ConvPropertiesAlphaEpsTerm} below we find
\begin{align*}
\lim_{k \to \infty} \int_{\Omega} \alpha_{\eps}(\varphi_{k}) \abs{\bm{u}_{k}}^{2} \dx = \int_{\Omega} \alpha_{\eps}(\varphi) \abs{\bm{u}}^{2} \dx.
\end{align*} 
In addition, by means of \eqref{e:bbilinearformThreeconvergence} from Lemma \ref{l:TrilinearFormStrongCont} we have 
\begin{align*}
\lim_{k \to \infty} b(\bm{u}_{k}, \bm{u}_{k}, \bm{u}_{k}) = b(\bm{u}, \bm{u}, \bm{u}). 
\end{align*}
These two results allow us to deduce from the convergence of the minimal functional values of $(F_{k})_{k \in \N}$ that $\lim_{k \to \infty} \int_{\Omega} \abs{\nabla \bm{u}_{k}}^{2} \dx = \int_{\Omega} \abs{\nabla \bm{u}}^{2} \dx$. Then, together with $\bm{u}_{k} \rightharpoonup \bm{u}$ in $\bm{H}^{1}(\Omega)$ this yields that $\lim_{k\to\infty}\norm{\bm{u}_{k} - \bm{u}}_{\bm{H}^{1}(\Omega)}=0$.

Subtracting the state equations \eqref {IntroStateEquPhaseWeak} written for $\varphi$ from the state equations \eqref {IntroStateEquPhaseWeak} written for $\varphi_{k}$, we find from Lemma \ref{l:ConvPropertiesAlphaEpsTerm} below that
\begin{align*}
\int_{\Omega} (p_{k} - p) \div \bm{v} \dx &=\int_{\Omega} (\alpha_{\eps}(\varphi_{k}) \bm{u}_{k} - \alpha_{\eps}(\varphi) \bm{u}) \cdot \bm{v} + \mu \nabla( \bm{u}_{k} - \bm{u}) \cdot \nabla \bm{v} \dx \\
& + b (\bm{u}_{k}, \bm{u}_{k}, \bm{v}) - b(\bm{u}, \bm{u}, \bm{v}) \\
&\leq \underbrace{\norm{\alpha_{\eps}(\varphi_{k}) \bm{u}_{k} - \alpha_{\eps}(\varphi) \bm{u}}_{\bm{L}^{2}(\Omega)}}_{\xrightarrow{k \to \infty}0} \norm{\bm{v}}_{\bm{L}^{2}(\Omega)} + \mu \underbrace{\norm{\bm{u}_{k} - \bm{u}}_{\bm{H}^{1}(\Omega)}}_{\xrightarrow{k \to \infty} 0} \norm{\bm{v}}_{\bm{H}^{1}(\Omega)} \\
&+ \underbrace{\norm{b(\bm{u}_{k}, \bm{u}_{k}, \cdot) - b(\bm{u}, \bm{u}, \cdot)}_{\bm{H}^{-1}(\Omega)}}_{\xrightarrow{k \to \infty}0} \norm{\bm{v}}_{\bm{H}^{1}(\Omega)}.
\end{align*}
Thus $\lim_{k \to \infty} \norm{\nabla (p_{k} - p)}_{\bm{H}^{-1}(\Omega)} = 0$. Using now the pressure estimate, see for instance \cite[Lemma II.1.5.4]{book:Sohr}, we find
\begin{align*}
\norm{p_{k} - p}_{L^{2}(\Omega)} \leq c \norm{\nabla (p_{k} - p)}_{\bm{H}^{-1}(\Omega)}\xrightarrow{k\to\infty}0.
\end{align*}
Therefore, we deduce that $(p_{k})_{k \in \N}$ converges strongly in $L^{2}(\Omega)$ to $p$.
\end{proof}

In the previous proof we made use of the following lemma:

\begin{lem}\label{l:ConvPropertiesAlphaEpsTerm}
Under Assumption \ref{assump:alpha}, assume that for  $(\varphi_{k})_{k \in \N}\subset L^{1}(\Omega)$, $(\bm{u}_{k})_{k \in \N}\subset \bm{L}^{2}(\Omega)$ and $\varphi \in L^{1}(\Omega)$, $\bm{u} \in \bm{L}^{2}(\Omega)$,
\begin{align*}
\lim_{k \to \infty} \norm{\varphi_{k} -\varphi}_{L^{1}(\Omega)} = 0, \quad \varphi_{k} \to \varphi \text{ a.e. and } \lim_{k\to\infty} \norm{\bm{u}_{k} - \bm{u}}_{\bm{L}^{2}(\Omega)}=0.
\end{align*}
Then it holds that
\begin{align*}
\lim_{k \to \infty} \int_{\Omega} \alpha_{\eps}(\varphi_{k})\abs{\bm{u}_{k}}^{2} \dx = \int_{\Omega} \alpha_{\eps}(\varphi) \abs{\bm{u}}^{2} \dx \text{ and } \lim_{k \to \infty} \norm{\alpha_{\eps}(\varphi_{k}) \bm{u}_{k} - \alpha_{\eps}(\varphi) \bm{u}}_{\bm{L}^{2}(\Omega)}=0.
\end{align*}
\end{lem}
\begin{proof}
Using the ideas of \cite[Theorem 5.1]{thesis:Hecht} and \cite[Theorem 1]{GarckeHechtNS} we find that
\begin{align*}
\abs{\int_{\Omega} \alpha_{\eps} (\varphi_{k})\abs{\bm{u}_{k}}^{2} - \alpha_{\eps}(\varphi) \abs{\bm{u}}^{2} \dx} & = \int_{\Omega} \alpha_{\eps}(\varphi_{k}) \left( \abs{\bm{u}_{k}}^{2} - \abs{\bm{u}}^{2} \right) \dx \\
& + \int_{\Omega} (\alpha_{\eps}(\varphi_{k}) - \alpha_{\eps}(\varphi))\abs{\bm{u}}^{2}\dx,
\end{align*}
and from $\alpha_{\eps} \in L^{\infty}(\R)$ we obtain
\begin{align*}
\int_{\Omega} \alpha_{\eps}(\varphi_{k}) \left(\abs{\bm{u}_{k}}^{2} - \abs{\bm{u}}^{2} \right) \dx \leq \norm{\alpha_{\eps}}_{L^{\infty}(\R)} \norm{\bm{u}_{k} + \bm{u}}_{\bm{L}^{2}(\Omega)} \norm{\bm{u}_{k} - \bm{u}}_{\bm{L}^{2}(\Omega)} \xrightarrow{k\to\infty}0.
\end{align*}
Moreover, the uniform bound on $\alpha_{\eps}$ yields by Lebesgue's dominated convergence theorem
\begin{align*}
\lim_{k \to \infty} \int_{\Omega} (\alpha_{\eps} (\varphi_{k}) - \alpha_{\eps}(\varphi)) \abs{\bm{u}}^{2}\dx = 0,
\end{align*}
which combined with the previous step yields 
the first assertion.

Using a similar idea we find
\begin{align*}
\norm{\alpha_{\eps}(\varphi_{k}) \bm{u}_{k} - \alpha_{\eps}(\varphi) \bm{u}}_{\bm{L}^{2}(\Omega)} & \leq \norm{\alpha_{\eps}(\varphi_{k})(\bm{u}_{k} - \bm{u})}_{\bm{L}^{2}(\Omega)} + \norm{(\alpha_{\eps}(\varphi_{k}) - \alpha_{\eps}(\varphi)) \bm{u}}_{\bm{L}^{2}(\Omega)} \\
& \leq \norm{\alpha_{\eps}}_{L^{\infty}(\R)}\norm{\bm{u}_{k} - \bm{u}}_{\bm{L}^{2}(\Omega)} + \norm{(\alpha_{\eps}(\varphi_{k}) - \alpha_{\eps}(\varphi))\bm{u}}_{\bm{L}^{2}(\Omega)}\xrightarrow{k\to\infty}0,
\end{align*}
where we applied Lebesgue's dominated convergence theorem in order to deduce from $\alpha_{\eps} \in L^{\infty}(\R)$ that $\lim_{k\to\infty}\norm{(\alpha_{\eps}(\varphi_{k}) - \alpha_{\eps}(\varphi))\bm{u}}_{\bm{L}^{2}(\Omega)}=0$.
\end{proof}

We make the following assumption regarding $h$:
\begin{assumption}\label{assump:generalh} Let
$h : \Omega \times \R^{d \times d} \times \R \times \R^{d} \to \R$ be a Carath\'{e}odory function, which fulfils
\begin{enumerate}
\item $h(\cdot, \bm{A}, s, \bm{w}) : \Omega \to \R$ is measurable for each $\bm{w} \in \R^{d}, s \in \R, \bm{A} \in \R^{d \times d}$, and
\item $h(x, \cdot, \cdot, \cdot) : \R^{d \times d} \times \R \times \R^{d} \to \R$ is continuous for almost every $x \in \Omega$.
\end{enumerate}
Moreover, there exist non-negative functions $a \in L^{1}(\Omega)$, $b_{1}, b_{2}, b_{3} \in L^{\infty}(\Omega)$ such that for almost every $x \in \Omega$ it holds
\begin{align*}
\abs{h(x, \bm{A}, s, \bm{w})} \leq a(x) + b_{1}(x) \abs{\bm{A}}^{2} + b_{2}(x) \abs{s}^{2} + b_{3}(x) \abs{\bm{w}}^{2},
\end{align*}
for all $\bm{w} \in \R^{d}, s \in \R, \bm{A} \in \R^{d \times d}$.  

Furthermore, the functional $\mathcal{H} : \bm{H}^{1}(\Omega) \times L^{2}(\Omega) \times H^{1}(\Omega) \to \R$ defined as
\begin{align*}
\mathcal{H}(\bm{u}, p, \varphi) & := \int_{\Omega} \mathcal{M}(\varphi) h(x, \nabla \bm{u}, p, \nabla \varphi) \dx, 
\end{align*}
satisfy the following properties
\begin{enumerate}
\item[(i)] $\mathcal{H} \mid_{\bm{H}^{1}_{\bm{g}, \sigma}(\Omega) \times L^{2}_{0}(\Omega) \times \Phi_{ad}}$ is bounded from below, and
\item[(ii)] for all $\varphi_{n} \rightharpoonup \varphi$ in $H^{1}(\Omega)$, $\bm{u}_{n} \to \bm{u}$ in $\bm{H}^{1}(\Omega)$, $p_{n} \to p$ in $L^{2}(\Omega)$, it holds that
\begin{align*}
\mathcal{H}(\bm{u}, p, \varphi) \leq \liminf_{n \to \infty} \mathcal{H}(\bm{u}_{n}, p_{n}, \varphi_{n}).
\end{align*}
\end{enumerate}
\end{assumption}

We then obtain the following existence result for \eqref{IntroObjFcltPhase}-\eqref{IntroStateEquPhaseWeak}:

\begin{thm}\label{t:PhaseFieldExistMin}
Under Assumptions \ref{assump:psi}, \ref{assump:alpha} and \ref{assump:generalh}, there exists at least one minimizer of the optimal control problem \eqref{IntroObjFcltPhase}-\eqref{IntroStateEquPhaseWeak}.
\end{thm}

\begin{proof}
We may restrict ourselves to considering $\varphi \in \Phi_{ad}$ with $\varphi \in [s_{a}, s_{b}]$ a.e. in $\Omega$. In fact, we define as in \cite[Proof of Proposition 1]{article:Modica87} for arbitrary $\varphi \in \Phi_{ad}$ the truncated functions $\tilde{\varphi} := \max\{s_{a},\min\{\varphi,s_{b}\}\}$ and find $\mathcal{E}_{\eps}(\tilde{\varphi}) \leq \mathcal{E}_{\eps}(\varphi)$, where $\mathcal{E}_{\eps}$ is defined in (\ref{defn:GinzburgLandau}).  Moreover, by (\ref{alphaepsSaSb}), we have $\alpha_{\eps}(\varphi) = \alpha_{\eps}(\tilde{\varphi})$ and hence also $\bm{S}_{\eps}(\varphi) = \bm{S}_{\eps}(\tilde{\varphi})$.  Therefore we obtain 
\begin{align*}
J_{\eps}^{h}(\tilde{\varphi}, \bm{u}, p)\leq J_{\eps}^{h}(\varphi, \bm{u}, p) \text{ for all } (\bm{u}, p) \in \bm{S}_{\eps}(\varphi) = \bm{S}_{\eps}(\tilde{\varphi}).
\end{align*}

By Assumption \ref{assump:generalh}, $\mathcal{H} \mid_{\bm{H}^{1}_{\bm{g}, \sigma}(\Omega) \times L^{2}_{0}(\Omega) \times \Phi_{ad}}$ is bounded below by a constant $C_{0}$, and so $J_{\eps}^{h} : \Phi_{ad} \times \bm{H}^{1}_{\bm{g}, \sigma}(\Omega) \times L^{2}_{0}(\Omega)$ is bounded from below by a constant $C_{1}$.  Thus, we can choose  a minimizing sequence $(\varphi_{n} , \bm{u}_{n} , p_{n})_{n \in \N} \subset \Phi_{ad} \times \bm{H}^{1}_{\bm{g},\sigma}(\Omega) \times L^{2}_{0}(\Omega)$ with $(\bm{u}_{n}, p_{n}) \in \bm{S}_{\eps}(\varphi_{n})$ for all $n$ and
\begin{align*}
\lim_{n \to \infty } J_{\eps}^{h}(\varphi_{n}, \bm{u}_{n} , p_{n}) = \inf_{\varphi \in \Phi_{ad}, (\bm{u}, p) \in \bm{S}_{\eps}(\varphi)} J_{\eps}^{h}(\varphi, \bm{u}, p) > -\infty. 
\end{align*}
In particular, from the non-negativity of $\psi$ and $\alpha_{\eps}$, we see that for $\rho > 0$, there exists an $N$ such that $n > N$ implies
\begin{align*}
C_{0} + \frac{\gamma \eps}{2c_{0}}\norm{\nabla \varphi_{n}}_{\bm{L}^{2}(\Omega)} \leq J_{\eps}^{h}(\varphi_{n}, \bm{u}_{n}, p_{n}) \leq \inf_{\varphi \in \Phi_{ad}, (\bm{u}, p) \in \bm{S}_{\eps}(\varphi)} J_{\eps}^{h}(\varphi, \bm{u}, p) + \rho.
\end{align*}
Thus, $\{\nabla \varphi_{n}\}_{n \in \N}$ is bounded uniformly in $\bm{L}^{2}(\Omega)$.  Moreover, without loss of generality, we may assume that $\varphi_{n}(x) \in [s_{a}, s_{b}]$ for a.e. $x \in \Omega$ and every $n \in \N$.  And so, we deduce that $\{\varphi_{n}\}_{n \in \N}$ is bounded uniformly in $H^{1}(\Omega) \cap L^{\infty}(\Omega)$, and we may choose a subsequence $(\varphi_{n_{k}})_{k \in \N}$ that converges strongly in $L^{2}(\Omega)$ and pointwise almost everywhere in $\Omega$ to some limit element $\varphi \in \Phi_{ad}$.  

Using Lemma \ref{l:SolOpCont} we can deduce that there is a subsequence of $(\bm{u}_{n_{k}}, p_{n_{k}})_{k \in \N}$, denoted by the same index, such that 
\begin{align}\label{e:ExistMinProofStrongConvState}
\lim_{k \to \infty} \norm{\bm{u}_{n_{k}}- \bm{u}}_{\bm{H}^{1}(\Omega)} = 0, \quad\lim_{k \to \infty} \norm{p_{n_{k}} - p}_{L^{2}(\Omega)} = 0,
\end{align}
and $(\bm{u}, p) \in \bm{S}_{\eps}(\varphi)$.

From Lemma \ref{l:ConvPropertiesAlphaEpsTerm} we deduce additionally that
\begin{align}\label{e:ExistMinProofTerm2}
\lim_{k \to \infty} \int_{\Omega} \alpha_{\eps}(\varphi_{n_{k}}) \abs{\bm{u}_{n_{k}}}^{2} \dx = \int_{\Omega} \alpha_{\eps}(\varphi) \abs{\bm{u}}^{2}\dx.
\end{align}
As $\sup_{k \in \N} \norm{\psi(\varphi_{n_{k}})}_{L^{\infty}(\Omega)} < \infty$ we can use Lebesgue's dominated convergence theorem to deduce $\lim_{k \to\infty} \int_{\Omega} \psi (\varphi_{n_{k}} ) \dx = \int_{\Omega} \psi ( \varphi ) \dx$.  Finally, the weak lower semicontinuity of $H^{1}(\Omega) \ni \varphi \mapsto \int_{\Omega} \abs{\nabla \varphi}^{2}\dx$ yields
\begin{align}\label{e:ExistMinProofTerm3}
\int_{\Omega} \frac{\eps}{2} \abs{\nabla \varphi}^{2} +\frac{1}{\eps} \psi(\varphi) \dx \leq \liminf_{k \to \infty} \int_{\Omega} \frac{\eps}{2} \abs{\nabla \varphi_{n_{k}}}^{2} + \frac{1}{\eps}\psi(\varphi_{n_{k}}) \dx.
\end{align}

Together with the lower semicontinuity assumption on $\mathcal{H}$ from Assumption \ref{assump:generalh}, we deduce that
\begin{align*}
J_{\eps}^{h}(\varphi, \bm{u}, p) \leq \liminf_{k \to \infty} J_{\eps}^{h}(\varphi_{n_{k}}, \bm{u}_{n_{k}}, p_{n_{k}}) = \inf_{\varphi \in \Phi_{ad},(\bm{u}, p) \in \bm{S}_{\eps}(\varphi)} J_{\eps}^{h}(\varphi,\bm{u}, p),
\end{align*}
and so $(\varphi, \bm{u}, p)$ is a minimizer of \eqref{IntroObjFcltPhase}-\eqref{IntroStateEquPhaseWeak}.
\end{proof}

By the same arguments, one can show an analogous existence result for the optimal control problem $\{\eqref{IntroStateEquPhaseWeak}, \eqref{ObjFunctHydroPhase}\}$ involving the hydrodynamic force \eqref{HydroDynamForce}:

\begin{thm}\label{t:HydroDyanExistMin}
Under Assumptions \ref{assump:psi} and \ref{assump:alpha}, there exists at least one minimizer of the optimization problem  $\{\eqref{IntroStateEquPhaseWeak}, \eqref{ObjFunctHydroPhase}\}$ involving the hydrodynamic force \eqref{HydroDynamForce}.
\end{thm}

\begin{proof}
We will prove the assertion for the choice $\mathcal{M}(\varphi) = \sqrt{\tfrac{\psi(\varphi) + \delta_{\eps}}{2}}$, and the analogous assertion for the choice $\mathcal{M}(\varphi) = \frac{1}{2}$ follows along the same lines.  

We first show that $\{ J_{\eps}(\varphi, \bm{u}, p) \mid \varphi \in \Phi_{ad}, (\bm{u}, p) \in \bm{S}_{\eps}(\varphi)\}$ is bounded from below.  We may restrict ourselves to considering $\varphi \in \Phi_{ad}$ with $\varphi \in [s_{a}, s_{b}]$ a.e. in $\Omega$ as in the proof of Theorem \ref{t:PhaseFieldExistMin}. 

Now let $\varphi \in \Phi_{ad}$ be arbitrarily chosen with $\varphi \in [s_{a},s_{b}]$ for a.e. $x \in \Omega$ and choose $(\bm{u}, p) \in \bm{S}_{\eps}(\varphi)$.  From \eqref{e:AprioriSolOpE}, we find a constant $C_{2} > 0$ independent of $\varphi$ such that 
\begin{align*}
\norm{\bm{u}}_{\bm{H}^{1}(\Omega)} + \norm{p}_{L^{2}(\Omega)} < C_{2}.
\end{align*}
By construction, we have
\begin{align*}
\varphi \in [s_{a}, s_{b}] \Longrightarrow \norm{\psi(\varphi)}_{L^{\infty}(\Omega)} < C_{3},
\end{align*} 
for some constant $C_{3} > 0$ independent of $\varphi$.  Then, using Cauchy--Schwarz's inequality, and Young's inequality we have
\begin{align*}
& \;  \frac{1}{c_{0}} \int_{\Omega} \sqrt{\tfrac{\psi(\varphi)+ \delta_{\eps}}{2}} \nabla \varphi \cdot \left(\mu \left ( \nabla \bm{u} + (\nabla \bm{u})^{T} \right ) - p \id \right) \bm{a} \dx  \\
\geq & \; -\frac{1}{c_{0} \sqrt{2}} \norm{\nabla\varphi \sqrt{\psi(\varphi)+ \delta_{\eps}} }_{L^{2}(\Omega)} \norm{\mu \left( \nabla \bm{u} + (\nabla \bm{u})^{T} \right )\bm{a} - p \bm{a} }_{\bm{L}^{2}(\Omega)} \\
\geq & \; - \frac{1}{c_{0}} \sqrt{\tfrac{C_{3} + \delta_{\eps}}{2}} \norm{ \nabla \varphi }_{L^{2}(\Omega)} \left( 2\mu C_{2} + C_{2} \right) \geq -  \frac{\gamma \eps}{8 c_{0}}\norm{\nabla \varphi}_{L^{2}(\Omega)}^{2} - C_{4},
\end{align*}
with some constant $C_{4} > 0$ independent of $\varphi$.  The non-negativity of $\alpha_{\eps}$ and $\psi$ yield that
\begin{equation}
\label{JepsLowerbdd}
\begin{aligned}
J_{\eps}(\varphi, \bm{u} , p) & \geq \int_{\Omega} \frac{1}{c_{0}}  \sqrt{\tfrac{\psi(\varphi) + \delta_{\eps}}{2}} \nabla \varphi \cdot \left(\mu\left(\nabla \bm{u} + (\nabla \bm{u})^{T} \right) - p\id \right) \bm{a} + \frac{\gamma}{2c_{0}} \frac{\eps}{2} \abs{\nabla \varphi}^{2} \dx \\
& \geq - \frac{\gamma \eps}{8 c_{0}} \norm{\nabla \varphi}_{L^{2}(\Omega)}^{2} - C_{4} +  \frac{\gamma \eps}{4c_{0}} \norm{\nabla \varphi}_{L^{2}(\Omega)}^{2}  = \frac{\gamma \eps}{8c_{0}} \norm{\nabla \varphi}_{L^{2}(\Omega)}^{2} - C_{4} \geq -C_{4}.
\end{aligned}
\end{equation}
This shows that $\{J_{\eps}(\varphi,\bm{u} , p) \mid \varphi \in \Phi_{ad}, (\bm{u} , p) \in \bm{S}_{\eps}(\varphi)\}$ is bounded from below.  Hence we may choose a minimizing sequence $(\varphi_{n} , \bm{u}_{n} , p_{n})_{n \in \N} \subset \Phi_{ad} \times \bm{H}^{1}_{\bm{g},\sigma}(\Omega) \times L^{2}_{0}(\Omega)$ with 
\begin{align*}
\lim_{n \to \infty } J_{\eps}(\varphi_{n}, \bm{u}_{n} , p_{n}) = \inf_{\varphi \in \Phi_{ad}, (\bm{u}, p) \in \bm{S}_{\eps}(\varphi)} J_{\eps}(\varphi, \bm{u}, p)  > -\infty. 
\end{align*}

As before, we deduce that $\{\varphi_{n}\}_{n \in \N}$ is bounded uniformly in $H^{1}(\Omega) \cap L^{\infty}(\Omega)$, together with Lemma \ref{l:SolOpCont}, we have subsequences $(\varphi_{n_{k}}, \bm{u}_{n_{k}}, p_{n_{k}})_{k \in \N}$, that satisfy
\begin{align*}
\lim_{k \to \infty} \norm{\varphi_{n_{k}} - \varphi}_{L^{2}(\Omega)} = 0, \quad \lim_{k \to \infty} \norm{\bm{u}_{n_{k}}- \bm{u}}_{\bm{H}^{1}(\Omega)} = 0, \quad \lim_{k \to \infty} \norm{p_{n_{k}} - p}_{L^{2}(\Omega)} = 0,
\end{align*}
and $(\bm{u}, p) \in \bm{S}_{\eps}(\varphi)$.

To deduce that $(\varphi, \bm{u}, p)$ is a minimizer of $\{\eqref{IntroStateEquPhaseWeak}, \eqref{ObjFunctHydroPhase}\}$, we only need to show that 
\begin{equation}\label{HydroDynamTermWLSC}\begin{aligned}
& \; \liminf_{k \to \infty} \int_{\Omega} \sqrt{\psi(\varphi_{n_{k}}) + \delta_{\eps}} \nabla \varphi_{n_{k}} \cdot \left( \mu \left( \nabla \bm{u}_{n_{k}} + (\nabla \bm{u}_{n_{k}})^{T} \right ) - p_{n_{k}} \id \right ) \bm{a} \dx \\
\geq & \; \int_{\Omega} \sqrt{\psi(\varphi) + \delta_{\eps}} \nabla \varphi \cdot \left( \mu \left( \nabla \bm{u} +(\nabla \bm{u})^T \right) - p \id \right) \bm{a} \dx,
\end{aligned}
\end{equation}
as the other integrals in \eqref{ObjFunctHydroPhase} are shown to be weakly lower semicontinuous in the proof of Theorem \ref{t:PhaseFieldExistMin}.  We apply now an idea of \cite{article:Modica87} and define
\begin{align*}
\phi(t) := \int_{s_{a}}^{t} \sqrt{\psi(s)+\delta_{\eps}} \, \mathrm ds, \quad w_{n_{k}}(x) := \phi(\varphi_{n_{k}}(x)).
\end{align*}
Then we see that 
\begin{align*}
\der w_{n_{k}}(x) = \phi'(\varphi_{n_{k}}(x)) \der \varphi_{n_{k}}(x) = (\sqrt{\psi(\varphi_{n_{k}}(x)) + \delta_{\eps}}) \der \varphi_{n_{k}}(x).
\end{align*}
By the uniform boundedness of $(\varphi_{n_{k}})_{k \in \N}$ in $H^{1}(\Omega) \cap L^{\infty}(\Omega)$, we find that $(\psi(\varphi_{n_{k}}))_{k \in \N}$ is uniformly bounded in $L^{\infty}(\Omega)$, and so by the Cauchy--Schwarz inequality,
\begin{align*}
\norm{w_{n_{k}}}_{L^{2}(\Omega)}^{2} & \leq \int_{\Omega} (\varphi_{n_{k}} - s_{a})  \left ( \int_{s_{a}}^{\varphi_{n_{k}}} (\psi(s) + \delta_{\eps}) \ds \right ) \dx \\
& \leq \sup_{s \in [s_{a}, s_{b}]} (\psi(s) + \delta_{\eps}) \int_{\Omega} \abs{\varphi_{n_{k}} - s_{a}}^{2} \dx, \\
\norm{\der w_{n_{k}}}_{L^{2}(\Omega)}^{2} & \leq \sup_{k \in \N} \, (\psi(\varphi_{n_{k}}) + \delta_{\eps}) \norm{\der \varphi_{n_{k}}}_{L^{2}(\Omega)}^{2}.
\end{align*}
Thus, we deduce that $(w_{n_{k}})_{k \in \N}$ is bounded uniformly in $H^{1}(\Omega)$, and hence there is a subsequence, denoted by the same index, that converges weakly in $H^{1}(\Omega)$ and pointwise almost everywhere in $\Omega$ to some limit element $w \in H^{1}(\Omega)$.  Since $\phi$ is continuous and $\lim_{k \to \infty} \varphi_{n_k}(x) = \varphi(x)$ for almost every $x \in \Omega$, we know that $w = \phi(\varphi)$. In particular, the weak convergence of $\der w_{n_{k}}$ to $\der w$ implies that 
\begin{align}
\label{e:ExistMinProofWeakConvTerms}
\sqrt{\psi(\varphi_{n_{k}}) + \delta_{\eps}} \nabla \varphi_{n_{k}} \rightharpoonup \sqrt{\psi(\varphi) + \delta_{\eps}} \nabla \varphi \quad\text{ in } \bm{L}^{2}(\Omega).
\end{align}

Combining \eqref{e:ExistMinProofStrongConvState} and \eqref{e:ExistMinProofWeakConvTerms} we obtain from the product of weak-strong convergence:
\begin{equation}\label{e:ExistMinProofTerm1}\begin{aligned}
& \; \lim_{k \to \infty} \int_{\Omega} \sqrt{\psi(\varphi_{n_{k}}) + \delta_{\eps}} \nabla\varphi_{n_{k}} \cdot \left( \mu \left( \nabla \bm{u}_{n_{k}} + (\nabla \bm{u}_{n_{k}})^{T} \right) - p_{n_{k}} \id \right) \bm{a} \dx \\
= & \; \int_{\Omega} \sqrt{\psi(\varphi) + \delta_{\eps}} \nabla \varphi \cdot \left( \mu \left( \nabla \bm{u} +(\nabla \bm{u})^T \right) - p \id \right) \bm{a} \dx.
\end{aligned}
\end{equation}

Using \eqref{e:ExistMinProofTerm1}, \eqref{e:ExistMinProofTerm2} and \eqref{e:ExistMinProofTerm3}, we deduce that
\begin{align*}
J_{\eps}(\varphi, \bm{u}, p) \leq \liminf_{k \to \infty} J_{\eps}(\varphi_{n_{k}}, \bm{u}_{n_{k}}, p_{n_{k}}) = \inf_{\varphi \in \Phi_{ad},(\bm{u}, p) \in \bm{S}_{\eps}(\varphi)} J_{\eps}(\varphi,\bm{u}, p),
\end{align*}
and so $(\varphi, \bm{u}, p)$ is a minimizer of $\{\eqref{IntroStateEquPhaseWeak}, \eqref{ObjFunctHydroPhase}\}$.
\end{proof}

\begin{rem}
Note that, for the choice $\mathcal{M}(\varphi) = \frac{1}{2}$, the proof of Theorem \ref{t:HydroDyanExistMin} is completed once we showed that $J_{\eps}$ is bounded from below, which can be shown similarly as in \eqref{JepsLowerbdd}, and (ii) in Assumption \ref{assump:regularityh} has been verified.  This follows the product of weak-strong convergence:
\begin{equation}
\begin{aligned}
& \; \lim_{k \to \infty} \int_{\Omega} \nabla\varphi_{n_{k}} \cdot \left( \mu \left( \nabla \bm{u}_{n_{k}} + (\nabla \bm{u}_{n_{k}})^{T} \right) - p_{n_{k}} \id \right) \bm{a} \dx \\
= & \; \int_{\Omega} \nabla \varphi \cdot \left( \mu \left( \nabla \bm{u} +(\nabla \bm{u})^T \right) - p \id \right) \bm{a} \dx.
\end{aligned}
\end{equation}
\end{rem}

\subsection{Optimality conditions}
This section is devoted to the derivation of a first order necessary optimality system for the optimal control problem \eqref{IntroObjFcltPhase}-\eqref{IntroStateEquPhaseWeak}.  For this purpose, we first show Fr\'echet differentiability of the solution operator.  We will only be able to show differentiability at certain points where the solution to the state equations is unique.  Otherwise we cannot apply the implicit function theorem in order to deduce the statement.  To be precise, we obtain the following result:

\begin{lem}\label{l:SolOpDiffable}
Under Assumption \ref{assump:alpha}, let $\varphi_{\eps} \in H^{1}(\Omega)\cap L^{\infty}(\Omega)$ be given such that there is $(\bm{u}_{\eps} , p_{\eps}) \in \bm{S}_{\eps}(\varphi_{\eps})$ with $\norm{\nabla \bm{u}_{\eps}}_{\bm{L}^{2}(\Omega)} <\frac{\mu}{K_{\Omega}}$.  Then there is a neighborhood $N$ of $\varphi_{\eps}$ in $H^{1}(\Omega) \cap L^{\infty}(\Omega)$ such that for every $\varphi \in N$ the solution operator consists of exactly one pair, and hence we may write $\bm{S}_{\eps} : N \subset H^{1}(\Omega) \cap L^{\infty}(\Omega) \to \bm{H}^{1}(\Omega) \times L^{2}(\Omega)$. This mapping is then differentiable at $\varphi_{\eps}$ with $\der \bm{S}_{\eps}(\varphi_{\eps})(\varphi) =: (\bm{u}, p) \in \bm{H}^{1}_{0}(\Omega) \times L^{2}_{0}(\Omega)$ being the unique solution of the linearized state system
\begin{subequations}\label{e:PhaseLinearizedState}
\begin{alignat}{2}
\alpha'_{\eps}(\varphi_{\eps}) \varphi \bm{u}_{\eps} + \alpha_{\eps}(\varphi_{\eps}) \bm{u} - \mu \Laplace \bm{u} + (\bm{u} \cdot \nabla )\bm{u}_{\eps} + (\bm{u}_{\eps} \cdot \nabla)\bm{u} + \nabla p  & = \bm{0} && \text{ in } \Omega, \\
\div \bm{u} & = 0 && \text{ in }  \Omega, \\
\bm{u} & = \bm{0} && \text{ on }  \pd \Omega.
\end{alignat}
\end{subequations}
\end{lem}

\begin{proof}
As already mentioned, we want to apply the implicit function theorem to get the statements of the lemma.  For this purpose, we first note that, by \cite[Lemma IX.4.2]{book:Galdi}, there exists a $\bm{G} \in \bm{H}^{1}_{\bm{g},\sigma}(\Omega)$, i.e., $\bm{G}$ satisfies
\begin{align*}
\div \bm{G} = 0 \text{ in } \Omega, \quad \bm{G} \mid_{\pd \Omega} = \bm{g}.
\end{align*}

We define
\begin{align*}
F : (H^{1}(\Omega) \cap L^{\infty}(\Omega)) \times \bm{H}^{1}_{0}(\Omega) \times L^{2}_{0}(\Omega) \to \bm{H}^{-1}(\Omega) \times L^{2}_{0}(\Omega), \quad F = (F_{1}, F_{2}), 
\end{align*}
by
\begin{align*}
F_{1} (\varphi, \bm{u}, p ) \bm{v} &:= \int_{\Omega} \alpha_{\eps}(\varphi) \bm{u} \cdot \bm{v} + \mu \nabla \bm{u} \cdot \nabla \bm{v} + (\bm{u} \cdot \nabla )\bm{u} \cdot \bm{v} - p \div \bm{v} - \bm{f} \cdot \bm{v} \dx \\
& + \int_{\Omega} (\bm{u} \cdot \nabla) \bm{G} \cdot \bm{v} + (\bm{G} \cdot \nabla)\bm{u} \cdot \bm{v} + \alpha_{\eps}(\varphi) \bm{G} \cdot \bm{v} + \mu \nabla \bm{G} \cdot \nabla \bm{v} + (\bm{G} \cdot \nabla) \bm{G} \cdot \bm{v} \dx, \\
F_{2}(\varphi, \bm{u}, p) &:= \div \bm{u},
\end{align*}
for all $\bm{v} \in \bm{H}^{1}_{0}(\Omega)$.

Hence, $F(\varphi,\bm{u} - \bm{G}, p)=0$ if and only if $(\bm{u}, p) \in \bm{S}_{\eps}(\varphi)$. Thus in particular we have $F(\varphi_{\eps}, \bm{u}_{\eps} - \bm{G}, p_{\eps})=0$.  Besides, we directly see that the Fr\'{e}chet differential $\der_{(\bm{u}, p)}F$ exists and is given at $(\varphi_{\eps}, \bm{u}_{\eps} - \bm{G}, p_{\eps})$ as
\begin{align*}
\der_{(\bm{u}, p)} F_{1} (\varphi_{\eps}, \bm{u}_{\eps} - \bm{G}, p_{\eps}) (\bm{u}, p) \bm{v} &= \int_{\Omega} \alpha_{\eps}(\varphi_{\eps}) \bm{u} \cdot \bm{v} + \mu\nabla \bm{u} \cdot \nabla \bm{v} + (\bm{u} \cdot \nabla )\bm{u}_{\eps} \cdot \bm{v} \dx \\
& + \int_{\Omega} (\bm{u}_{\eps} \cdot \nabla ) \bm{u} \cdot \bm{v} - p \div \bm{v} \dx,\\
\der_{(\bm{u}, p)} F_{2} (\varphi_{\eps}, \bm{u}_{\eps} - \bm{G}, p_{\eps})(\bm{u}, p) & = \div \bm{u}.
\end{align*}

The assumption $\norm{\nabla \bm{u}_{\eps}}_{\bm{L}^{2}(\Omega)} < \frac{\mu}{K_{\Omega}}$, equations \eqref{e:ContinuityEstimateTrilinearForm} and \eqref{e:TrilinearformLastTwoEqualZero} ensure that 
\begin{align*}
\bm{H}^{1}_{0,\sigma}(\Omega)\times \bm{H}^{1}_{0,\sigma}(\Omega) \ni (\bm{u}, \bm{v}) \mapsto \int_{\Omega} \alpha_{\eps}(\varphi_{\eps}) \bm{u} \cdot \bm{v} + \mu \nabla \bm{u} \cdot \nabla \bm{v} + (\bm{u} \cdot \nabla )\bm{u}_{\eps} \cdot \bm{v} + (\bm{u}_{\eps} \cdot \nabla) \bm{u} \cdot \bm{v} \dx
\end{align*}
defines a coercive, continuous bilinear form.   Hence, we may use the Lax--Milgram theorem and standard results for the solvability of the divergence operator, see for instance \cite[Lemma II.2.1.1]{book:Sohr}, in order to obtain that $\der_{(\bm{u},p)}F(\varphi_{\eps}, \bm{u}_{\eps} - \bm{G}, p_{\eps})$ is an isomorphism.

Next, we want to consider the differentiability of $F$ with respect to its first argument.  For this purpose, we have to consider $\alpha_{\eps}: L^{6}(\Omega) \to L^{\frac{3}{2}}(\Omega)$ as a Nemytskii operator, making in particular use of the embedding $H^{1}(\Omega) \hookrightarrow L^{6}(\Omega)$.  The results in \cite[Section 4.3.3]{book:Troeltzsch} ensure that $\alpha_{\eps} : L^{6}(\Omega)\to L^{\frac{3}{2}}(\Omega)$ defines a Fr\'{e}chet-differentiable Nemytskii operator, which follows from the assumption $\alpha_{\eps} \in L^{\infty}(\R)\cap C^{1,1}(\R)$.  We can then conclude directly that $F$ is Fr\'{e}chet differentiable with respect to its first argument with 
\begin{align*}
\der_{\varphi} F_{1} (\varphi, \bm{u} - \bm{G}, p)(\tilde{\varphi})\bm{v} = \int_{\Omega} \alpha'_{\eps}(\varphi) \tilde{\varphi} \bm{u} \cdot \bm{v} \dx, \quad \der_{\varphi} F_{2}(\varphi, \bm{u} - \bm{G}, p) = 0.
\end{align*}

Additionally, we need that $F$ is Fr\'{e}chet differentiable in a neighborhood of $(\varphi_{\eps}, \bm{u}_{\eps}, p_{\eps})$.  To show this, we will use \cite[Proposition 4.14]{book:Zeidler1}, i.e., we show that the partial derivatives are continuous in order to conclude that $F$ is Fr\'{e}chet differentiable.  Thus let $(\varphi_{k}, \bm{u}_{k}, p_{k})_{k \in \N} \subset (H^{1}(\Omega) \cap L^{\infty}(\Omega)) \times \bm{H}^{1}_{0}(\Omega) \times L^{2}_{0}(\Omega)$ be sequences with 
\begin{align*}
\lim_{k \to \infty} \norm{\bm{u}_{k} - \bm{u}}_{\bm{H}^{1}(\Omega)} = 0, \quad \lim_{k \to \infty} \norm{p_{k} - p}_{L^{2}(\Omega)} = 0, \quad \lim_{k \to \infty} \norm{\varphi_{k} - \varphi}_{H^{1}(\Omega) \cap L^{\infty}(\Omega)} = 0. 
\end{align*}
As $\alpha_{\eps}:L^6(\Omega)\to L^{\frac{3}{2}}(\Omega)$ defines a continuous Nemytskii-operator, making additionally use of the continuity properties of the trilinear form as stated in Lemma \ref{l:TrilinearFormStrongCont}, we can deduce that
\begin{align*}
\lim_{k \to \infty} \norm{\der_{(\bm{u}, p)} F(\varphi_{k}, \bm{u}_{k}, p_{k}) - \der_{(\bm{u}, p)} F(\varphi, \bm{u}, p)}_{\mathcal{L}(\bm{H}^{1}_{0}(\Omega) \times L^{2}_{0}(\Omega), \bm{H}^{-1}(\Omega) \times L^{2}_{0}(\Omega))} = 0.
\end{align*} 

Moreover, from $\alpha'_{\eps} \in C^{0,1}$ and standard results for Nemytskii operators we find that $L^{6}(\Omega) \ni \varphi \mapsto \alpha'_{\eps}(\varphi) \in L^{6}(\Omega)$ is continuous.  And thus we also find by direct calculations that $\lim_{k \to \infty} \norm{\der_{\varphi}F (\varphi_{k}, \bm{u}_{k}, p_{k}) - \der_{\varphi}F(\varphi, \bm{u}, p)}_{\mathcal{L}(H^{1}(\Omega), \bm{H}^{-1}(\Omega) \times L^{2}_{0}(\Omega))}=0$.  Therefore, we obtain that $F$ is Fr\'{e}chet differentiable.

Finally, applying the implicit function theorem, we obtain for $\norm{\varphi-\varphi_{\eps}}_{H^{1}(\Omega) \cap L^{\infty}(\Omega)} \ll 1$ the existence and uniqueness of a pair $(\bm{u}, p)$ such that $F(\varphi, \bm{u} - \bm{G}, p) = 0$, i.e., $(\bm{u}, p) \in \bm{S}_{\eps}(\varphi)$.  This implies the first part of the statement. The second part of the lemma is a consequence of the differentiability statement of the implicit function theorem:
\begin{align*}
\der \bm{S}_{\eps}(\varphi_{\eps}) = -\left(\der_{(\bm{u}, p)} F (\varphi_{\eps}, \bm{u}_{\eps} -\bm{G}, p_{\eps}) \right)^{-1}\circ \der_{\varphi} F( \varphi_{\eps}, \bm{u}_{\eps} - \bm{G}, p_{\eps}),
\end{align*}
which reads in our setting as $\div \bm{u} = 0$ and
\begin{equation}\label{e:PhaseLinearizedStateWeakform}
\begin{aligned}
& \; \int_{\Omega} \alpha'_{\eps}(\varphi_{\eps})\varphi \bm{u}_{\eps} \cdot \bm{v} + \alpha_{\eps}(\varphi_{\eps}) \bm{u} \cdot \bm{v} + \mu \nabla \bm{u} \cdot \nabla \bm{v} \dx \\
+ & \; \int_{\Omega} (\bm{u} \cdot \nabla )\bm{u}_{\eps}\cdot \bm{v} + (\bm{u}_{\eps} \cdot \nabla)\bm{u} \cdot \bm{v} - p \div \bm{v} \dx = 0 \quad \forall \bm{v} \in \bm{H}^{1}_{0}(\Omega).
\end{aligned}
\end{equation}
\end{proof}

We denote by $\der_{i}h(x, \bm{A}, s, \bm{w})$ for $i \in \{1, 2, 3, 4\}$ as the differential of
\begin{align*}
\Omega \times \R^{d \times d} \times \R \times \R^{d} \ni (x, \bm{A}, s, \bm{w}) \mapsto h(x, \bm{A}, s, \bm{w})
\end{align*}
with respect to the $i$-th variable, respectively.

\begin{assumption}\label{assump:regularityh}
In addition to Assumption \ref{assump:generalh}, assume further that $x \mapsto h(x, \bm{A}, s, \bm{w})$ is in $W^{1,1}(\Omega)$ for all $(\bm{A}, s, \bm{w}) \in  \R^{d \times d} \times \R \times \R^{d}$ and the partial derivatives 
\begin{align*}
\der_{2}h(x, \cdot, s, \bm{w}), \; \der_{3}h(x, \bm{A}, \cdot, \bm{w}), \; \der_{4}h(x, \bm{A}, s, \cdot )
\end{align*} 
exist for all $\bm{w} \in \R^{d}$, $s \in \R$, $\bm{A} \in \R^{d \times d}$, and almost all $x \in \Omega$.  Moreover, we assume that 
\begin{equation}\label{equ:partialderivativesh}
\abs{\der_{i} h(x, \bm{A}, s, \bm{w})} \leq \tilde{a}(x) + \tilde{b}_{1}(x) \abs{\bm{A}} + \tilde{b}_{2}(x) \abs{s} + \tilde{b}_{3}(x) \abs{\bm{w}}, \text{ for } i \in \{2, 3, 4 \},
\end{equation}
for some non-negative $\tilde{a} \in L^{1}(\Omega)$, $\tilde{b}_{1}, \tilde{b}_{2}, \tilde{b}_{3} \in L^{\infty}(\Omega)$.
\end{assumption}

From Assumption \ref{assump:regularityh} we see that
\begin{align*}
(L^{2}(\Omega))^{d \times d} \ni \bm{A} & \mapsto \der_{2}h( \cdot, \bm{A}, s, \bm{w}) \in L^{2}(\Omega), \\
L^{2}(\Omega) \ni s & \mapsto \der_{3}h( \cdot, \bm{A}, s, \bm{w}) \in L^{2}(\Omega), \\
(L^{2}(\Omega))^{d} \ni \bm{w} & \mapsto \der_{4}h( \cdot, \bm{A}, s, \bm{w}) \in L^{2}(\Omega),
\end{align*}
are well-defined Nemytskii operators for $\bm{A} \in (L^{2}(\Omega))^{d \times d}$, $s \in L^{2}(\Omega)$, and $\bm{w} \in (L^{2}(\Omega))^{d}$ if and only if \eqref{equ:partialderivativesh} is fulfilled.  Moreover, the operator
\begin{align*}
(L^{2}(\Omega))^{d \times d} \times L^{2}(\Omega) \times (L^{2}(\Omega))^{d} \ni (\bm{A}, s, \bm{w}) \mapsto h( \cdot, \bm{A}, s, \bm{w}) \in L^{1}(\Omega)
\end{align*}
is continuously Fr\'{e}chet differentiable.

Next, by Assumption \ref{assump:psi}, $\psi \in C^{1,1}(\R)$, we have that $\der_{y}( \sqrt{\psi(y) + \delta_{\eps}})$ is locally Lipschitz and thus the Nemytskii operator
\begin{align*}
L^{\infty}(\Omega) \ni \varphi \mapsto \sqrt{\psi(\varphi) + \delta_{\eps}} \in L^{\infty}(\Omega)
\end{align*}
is continuously Fr\'{e}chet differentiable.  Hence, we find that
\begin{align*}
\mathcal{H} : \bm{H}^{1}(\Omega) \times L^{2}(\Omega) \times H^{1}(\Omega) \cap L^{\infty}(\Omega) \ni (\bm{u}, p, \varphi) \mapsto \int_{\Omega} \mathcal{M}(\varphi) h(x, \nabla \bm{u}, p, \nabla \varphi) \dx
\end{align*}
is continuously Fr\'{e}chet differentiable and its distributional derivative is given as
\begin{equation}\label{FrechDerivativemathcalH}
\begin{aligned}
\der \mathcal{H}(\bm{u}, p, \varphi)(\bm{v}, s, \eta) & = \int_{\Omega} \mathcal{M}(\varphi) (\der_{2}h, \der_{3}h, \der_{4}h)\mid_{(x, \nabla \bm{u}, p, \nabla \varphi)} \cdot (\nabla \bm{v}, s, \nabla \eta) \dx \\
& + \int_{\Omega} h(x, \nabla \bm{u}, p, \nabla \varphi) \mathcal{M}'(\varphi) \eta \dx.
\end{aligned}
\end{equation}

We note that for the choice $\mathcal{M}(\varphi) = \frac{1}{2}$, the second integral on the right hand side of \eqref{FrechDerivativemathcalH} vanishes as the Fr\'{e}chet derivative of $\frac{1}{2}$ is the zero functional.  On the other hand, for the choice $\mathcal{M}(\varphi) = \frac{1}{c_{0}} \sqrt{\frac{\psi(\varphi) + \delta_{\eps}}{2}}$, the Fr\'{e}chet derivative is given as
\begin{align}
\mathcal{M}'(\varphi) = \frac{1}{c_{0}} \frac{\psi'(\varphi)}{2 \sqrt{2 (\psi(\varphi) + \delta_{\eps})}}.
\end{align}

Before formulating the optimality system we want to discuss the adjoint system.  The pair of adjoint variables $(\bm{q}_{\eps}, \pi_{\eps}) \in \bm{H}^{1}_{0}(\Omega) \times L^{2}(\Omega)$ is the weak solution of the adjoint system, which is given as follows: find $(\bm{q}_{\eps}, \pi_{\eps}) \in \bm{H}^{1}_{0}(\Omega) \times L^{2}(\Omega)$ such that
\begin{subequations}\label{generalh:adjointsystem}
\begin{align}
\notag \alpha_{\eps} (\varphi_{\eps})(\bm{q}_{\eps} & - \bm{u}_{\eps}) - \mu \div (\nabla \bm{q}_{\eps} + (\nabla \bm{q}_{\eps})^{T}) + (\nabla \bm{u}_{\eps})^{T} \bm{q}_{\eps} - (\bm{u}_{\eps} \cdot \nabla) \bm{q}_{\eps} + \nabla \pi_{\eps} \\
& =  - \div \left ( \mathcal{M}(\varphi) \der_{2}h \right )  && \text{ in } \Omega, \\
\div \bm{q}_{\eps} & = - \mathcal{M}(\varphi) \der_{3}h + \vartheta_{\eps} && \text{ in } \Omega, \\
\bm{q}_{\eps}& = \bm{0} && \text{ on } \pd \Omega,
\end{align}
\end{subequations}
where $\der_{2}h, \der_{3}h$ are evaluated at $(x, \nabla \bm{u}_{\eps}, p_{\eps}, \nabla \varphi_{\eps})$ and
\begin{align}\label{defn:vartheta}
\vartheta_{\eps} := \strokedint_{\Omega} \mathcal{M}(\varphi) \der_{3}h(x, \nabla \bm{u}_{\eps}, p_{\eps}, \nabla \varphi_{\eps}) \dx.
\end{align}

\begin{rem}\label{r:VarThetaAsLagrangeMult}
The parameter $\vartheta_{\eps} \in \R$ can be interpreted as a Lagrange multiplier for the constraint $\int_{\Omega} p\dx = 0$.  By carrying out the formal Lagrange method as described for instance in \cite{book:Hinzeetal,book:Troeltzsch} and appending the mean value condition on the pressure $p$ with some Lagrange multiplier $\vartheta_{\eps}$ to the Lagrangian, one obtains that $\vartheta_{\eps}$ appears in the adjoint system as in \eqref{generalh:adjointsystem}.
\end{rem}

The next lemma shows that the system \eqref{generalh:adjointsystem} is uniquely solvable:

\begin{lem}\label{l:AdjointSysWellDef}
Let Assumptions \ref{assump:psi}, \ref{assump:alpha}, and \ref{assump:regularityh} hold, and let $\varphi_{\eps} \in H^{1}(\Omega) \cap L^{\infty}(\Omega)$ and $\bm{u}_{\eps} \in \bm{H}^{1}_{\bm{g}, \sigma}(\Omega)$ such that $\norm{\nabla \bm{u}_{\eps}}_{\bm{L}^{2}(\Omega)} < \frac{\mu}{K_{\Omega}}$ be given.  Then there exists a unique solution pair $(\bm{q}_{\eps}, \pi_{\eps}) \in \bm{H}^{1}_{0}(\Omega) \times L^{2}(\Omega)$ of the adjoint system \eqref{generalh:adjointsystem}.
\end{lem}
\begin{proof}
First, we notice that by definition of $\vartheta_{\eps}$ \eqref{defn:vartheta}, it holds that
\begin{align*}
\int_{\Omega} \mathcal{M}(\varphi) \der_{3}h(x, \nabla \bm{u}_{\eps}, p_{\eps}, \nabla \varphi_{\eps})  - \vartheta_{\eps} \dx = 0.
\end{align*}
As $\varphi_{\eps} \in L^{\infty}(\Omega)$, we have $\mathcal{M}(\varphi) \in L^{\infty}(\Omega)$ for either choices.  Thus, by Assumption \ref{assump:regularityh}, we obtain that $\mathcal{M}(\varphi) \der_{3}h \in L^2(\Omega)$.  So, from standard results, see for instance \cite[Lemma II.2.1.1]{book:Sohr}, we deduce the existence of some $\bm{w} \in \bm{H}^{1}_{0}(\Omega)$ such that 
\begin{align*}
\div \bm{w} = - \mathcal{M}(\varphi) \der_{3}h + \vartheta_{\eps}.
\end{align*}
Note that, by the density of $\bm{C}^{\infty}_{0,\sigma}(\Omega) := \{ \bm{v} \in (C^{\infty}_{0}(\Omega))^{d} \, | \div \bm{v}  = 0 \}$ in $\bm{H}^{1}_{0,\sigma}(\Omega)$ (see \cite[Lemma II.2.2.3]{book:Sohr}), for any $\bm{v} \in \bm{H}^{1}_{0,\sigma}(\Omega)$, there exists a sequence $\{\bm{v}^{n}\}_{n \in \N} \subset \bm{C}^{\infty}_{0,\sigma}(\Omega)$ such that 
\begin{align*}
\norm{\bm{v}^{n} - \bm{v}}_{\bm{H}^{1}(\Omega)} \to 0 \text{ as } n \to \infty.
\end{align*}
Thus, for any $\bm{y} \in \bm{H}^{1}_{0}(\Omega), \bm{v} \in \bm{H}^{1}_{0,\sigma}(\Omega)$, we find that by the commutativity of second derivatives,
\begin{equation}\label{nablaunablavTranspose}
\begin{aligned}
& \; \int_{\Omega} \nabla \bm{y} \cdot (\nabla \bm{v})^{T} \dx = \lim_{n \to \infty} \int_{\Omega} \nabla \bm{y} \cdot (\nabla \bm{v}^{n})^{T} \dx \\
= & \; \lim_{n \to \infty} \sum_{i,j=1}^{d} \int_{\Omega} \pd_{i} y_{j} \pd_{j} v^{n}_{i} \dx = \lim_{n \to \infty} \sum_{i,j=1}^{d} \left ( \int_{\pd \Omega} y_{j} \pd_{j} v^{n}_{i} \nu_{\pd \Omega,i} \dHaus - \int_{\Omega} y_{j} \pd_{j} \pd_{i} v^{n}_{i} \dx \right )\\
= & \; \lim_{n \to \infty} \int_{\pd \Omega} (\bm{y} \cdot \nabla) \bm{v}^{n} \cdot \bm{\nu}_{\pd \Omega} \dHaus - \int_{\Omega} \bm{y} \cdot \nabla (\div \bm{v}^{n}) \dx = 0.
\end{aligned}
\end{equation}

We define the bilinear form $a: \bm{H}^{1}_{0,\sigma}(\Omega) \times \bm{H}^{1}_{0,\sigma}(\Omega) \to \left ( \bm{H}^{1}_{0,\sigma}(\Omega) \right )'$ by 
\begin{equation}\label{adjoint:bilinearform}
\begin{aligned}
a(\bm{u}, \bm{v}) & := \int_{\Omega} \alpha_{\eps}(\varphi_{\eps}) \bm{u} \cdot \bm{v} + \mu \nabla \bm{u} \cdot (\nabla \bm{v} + (\nabla \bm{v})^{T}) + (\nabla \bm{u}_{\eps})^{T} \bm{u} \cdot \bm{v} - (\bm{u}_{\eps} \cdot \nabla) \bm{u} \cdot \bm{v} \dx \\
& = \int_{\Omega} \alpha_{\eps}(\varphi_{\eps}) \bm{u} \cdot \bm{v} + \mu \nabla \bm{u} \cdot \nabla \bm{v} + (\nabla \bm{u}_{\eps})^{T} \bm{u} \cdot \bm{v} - (\bm{u}_{\eps} \cdot \nabla) \bm{u} \cdot \bm{v} \dx,
\end{aligned}
\end{equation}
where we have used \eqref{nablaunablavTranspose} for $\bm{u}, \bm{v} \in \bm{H}^{1}_{0,\sigma}(\Omega)$.  Making use of $\norm{\nabla \bm{u}_{\eps}}_{\bm{L}^{2}(\Omega)} < \frac{\mu}{K_{\Omega}}$, \eqref{e:ContinuityEstimateTrilinearForm}, \eqref{e:TrilinearformLastTwoEqualZero}, and the Poincar\'{e} inequality, we can establish that $a(\cdot, \cdot)$ is a coercive bilinear form, i.e., there exists a constant $c(\mu, \abs{\Omega}) > 0$ such that,
\begin{align*}
a(\bm{u}, \bm{u}) & = \int_{\Omega} \underbrace{\alpha_{\eps}(\varphi_{\eps})}_{\geq 0} \abs{\bm{u}}^{2} + \mu \abs{\nabla \bm{u}}^{2} \dx + b(\bm{u}, \bm{u}_{\eps}, \bm{u}) - \underbrace{b(\bm{u}_{\eps}, \bm{u}, \bm{u})}_{=0 \text{ by }\eqref{e:TrilinearformLastTwoEqualZero}} \\
& \geq \mu \norm{\nabla \bm{u}}_{\bm{L}^{2}(\Omega)}^{2} - K_{\Omega} \norm{\nabla \bm{u}}_{\bm{L}^{2}(\Omega)}^{2} \norm{\nabla \bm{u}_{\eps}}_{\bm{L}^{2}(\Omega)} \geq c(\mu, \abs{\Omega}) \norm{\bm{u}}_{\bm{H}^{1}_{0}(\Omega)}^{2}.
\end{align*} 

Meanwhile, the boundedness of the bilinear form $a(\cdot, \cdot)$ in $\bm{H}^{1}_{0,\sigma}(\Omega) \times \bm{H}^{1}_{0,\sigma}(\Omega)$ can be shown using \eqref{e:ContinuityEstimateTrilinearForm}, the boundedness of $\alpha_{\eps}$, H\"{o}lder's inequality and the assumption $\norm{\nabla \bm{u}_{\eps}}_{\bm{L}^{2}(\Omega)} < \frac{\mu}{K_{\Omega}}$.  Thus, by the Lax--Milgram theorem, we obtain a unique $\hat{\bm{q}} \in \bm{H}^{1}_{0,\sigma}(\Omega)$ such that
\begin{align}\label{adjoint:weakform}
a(\hat{\bm{q}},\bm{v}) = \int_{\Omega}\alpha_{\eps}(\varphi_{\eps}) \bm{u}_{\eps} \cdot \bm{v} + \mathcal{M}(\varphi) ( \der_{2}h \cdot \nabla \bm{v} ) \dx - a(\bm{w}, \bm{v}) \quad \forall  \bm{v} \in \bm{H}^{1}_{0,\sigma}(\Omega).
\end{align}
We note that the integral terms are well-defined due to Assumption \ref{assump:regularityh} and the boundedness of $\alpha_{\eps}$.  We set $\bm{q}_{\eps} := \hat{\bm{q}} + \bm{w}$.  The existence of $\pi_{\eps} \in L^{2}(\Omega)$ follows from standard results, see for instance \cite[Lemma II.2.2.1]{book:Sohr}.  Thus, $(\bm{q}_{\eps}, \pi_{\eps})$ is the unique weak solution of the adjoint system \eqref{generalh:adjointsystem}.
\end{proof}

Now we can formulate necessary optimality conditions for our optimal control problem:

\begin{thm}\label{t:GeneralisationOptimality}
Let $(\varphi_{\eps}, \bm{u}_{\eps}, p_{\eps}) \in (\Phi_{ad} \cap L^{\infty}(\Omega)) \times \bm{H}^{1}_{\bm{g}, \sigma}(\Omega) \times L^{2}_{0}(\Omega)$ be a minimizer of $J_{\eps}^{h}$ such that $\norm{\nabla \bm{u}_{\eps}}_{\bm{L}^{2}(\Omega)} < \frac{\mu}{K_{\Omega}}$.  Then the following optimality system is fulfilled:  There exists a Lagrange multiplier $\lambda_{\eps} \in \R$ for the integral constraint such that
\begin{equation}\label{generalh:PhaseFieldOptSys}
\begin{aligned}
&\; \left(\alpha'_{\eps}(\varphi_{\eps}) \left(\frac{1}{2} \abs{\bm{u}_{\eps}}^{2} - \bm{u}_{\eps} \cdot \bm{q}_{\eps} \right) +\frac{\gamma}{2c_{0} \eps} \psi'(\varphi_{\eps}) +\lambda_{\eps} + \mathcal{M}'(\varphi_{\eps}) h(x, \nabla \bm{u}_{\eps}, p_{\eps}, \nabla \varphi_{\eps}) , \zeta \right)_{L^{2}(\Omega)} \\
+ \; & \left( \mathcal{M}(\varphi_{\eps}) \der_{4}h(x, \nabla \bm{u}_{\eps}, p_{\eps}, \nabla \varphi_{\eps}) + \frac{\gamma \eps}{2c_{0}} \nabla \varphi_{\eps} , \nabla \zeta \right)_{\bm{L}^{2}(\Omega)} = 0 \quad \forall \zeta \in H^{1}(\Omega) \cap L^{\infty}(\Omega).
\end{aligned}
\end{equation}
Here, $(\bm{q}_{\eps}, \pi_{\eps}) \in \bm{H}^{1}_{0}(\Omega) \times L^{2}(\Omega)$ is the unique weak solution of the adjoint system \eqref{generalh:adjointsystem}.
\end{thm}

\begin{proof}
We rewrite the problem \eqref{IntroObjFcltPhase}-\eqref{IntroStateEquPhaseWeak} as a minimizing problem for a reduced objective functional defined on an open set in $H^{1}(\Omega) \cap L^{\infty}(\Omega)$ by making use of Lemma \ref{l:SolOpDiffable}.  In particular, at least in a neighborhood $N \subset H^{1}(\Omega) \cap L^{\infty}(\Omega)$ of $\varphi_{\eps}$, the solution operator $\bm{S}_{\eps}$ is not set-valued, but for every $\varphi \in N$ we have $\bm{S}_{\eps}(\varphi) = \{ (\bm{u}, p) \}$.  Thus we may define the reduced functional $j_{\eps}^{h}: N \to \R$ by 
\begin{align*}
j_{\eps}^{h}(\varphi) := J_{\eps}^{h}(\varphi, \bm{S}_{\eps}(\varphi)).
\end{align*}
Then, $\varphi_{\eps}$ is also a local minimizer of $j_{\eps}^{h}$.  Hence, the gradient equation
\begin{align}\label{generalh:ProofOptSysEpsVarIn}
\der j_{\eps}^{h}(\varphi_{\eps})(\varphi) = 0, \quad \forall \varphi \in H^1(\Omega), \, \int_{\Omega} \varphi \dx=0,
\end{align}
would be fulfilled if $j_{\eps}^{h}$ would be differentiable.  

We will show in the next step that $j_{\eps}^{h}$ is differentiable at $\varphi_{\eps}$ as a mapping from $H^{1}(\Omega) \cap L^{\infty}(\Omega)$ to $\R$.  Lemma \ref{l:SolOpDiffable} already ensures that the solution operator $\bm{S}_{\eps}$ is differentiable from $H^{1}(\Omega) \cap L^{\infty}(\Omega)$ to $\bm{H}^{1}(\Omega)\times L^{2}(\Omega)$.  Thus we now look at dependence of $J_{\eps}^{h}$ on the first variable.

For this purpose we find first as in the proof of Lemma \ref{l:SolOpDiffable} that $\alpha_{\eps}: L^{6}(\Omega) \to L^{\frac{3}{2}}(\Omega)$ is a Fr\'{e}chet differentiable Nemytskii operator, and hence 
\begin{align*}
H^{1}(\Omega) \ni \varphi \mapsto \int_{\Omega} \alpha_{\eps}(\varphi)\abs{\bm{u}}^{2}\dx
\end{align*}
is Fr\'{e}chet differentiable for any $\bm{u}\in \bm{H}^{1}(\Omega)$.  With similar results, i.e. by making use of \cite[Section 4.3.3]{book:Troeltzsch}, we also find that
\begin{alignat*}{3}
L^{\infty}(\Omega)\ni \varphi & \mapsto \psi(\varphi) \in L^{\infty}(\Omega), && \quad L^{\infty}(\Omega)\ni \varphi &&\mapsto \int_{\Omega} \psi(\varphi)\dx, \\
H^{1}(\Omega) \ni \varphi &\mapsto \nabla \varphi \in \bm{L}^{2}(\Omega), && \quad H^{1}(\Omega) \ni \varphi &&\mapsto \int_{\Omega} \abs{\nabla \varphi}^{2} \dx
\end{alignat*}
are differentiable.  Combining these results and the Fr\'{e}chet differentiability of $\mathcal{H}$, we find that $j_{\eps}^{h} : N \to \R$ is differentiable.  Hence we may conclude by the minimizing property of $\varphi_{\eps}$ that the gradient equation \eqref{generalh:ProofOptSysEpsVarIn} is fulfilled.  We then find from \eqref{generalh:ProofOptSysEpsVarIn} that
\begin{align}\label{generalh:ProofOptSysEpsVarEq}
0 = \der j_{\eps}^{h}(\varphi_{\eps})\left(\varphi - \strokedint_{\Omega} \varphi \dx \right) = \der j_{\eps}^{h} (\varphi_{\eps}) (\varphi) + \lambda_{\eps} \int_{\Omega} \varphi \dx \quad \forall \varphi \in H^{1}(\Omega),
\end{align}
where we defined 
\begin{align}\label{defn:lambda}
\lambda_{\eps} := -\abs{\Omega}^{-1} \der j_{\eps} (\varphi_{\eps}) \in \R. 
\end{align}
In particular, we interpret $\lambda_{\eps} \in \R$ as a Lagrange multiplier for the integral constraint $\int_{\Omega} \varphi \dx = \beta \abs{\Omega}$.  

We now want to rewrite \eqref{generalh:ProofOptSysEpsVarEq} into a more convenient form by using the adjoint variable $\bm{q}_{\eps}$, which is defined as the solution of \eqref{generalh:adjointsystem}.  For this purpose we start calculating the derivative of $j_{\eps}^{h}$.  We find for every $\varphi \in H^{1}(\Omega)$ the following formula:
\begin{equation}\label{generalh:DerJEpsiloNFirstForm}
\begin{aligned}
\der j_{\eps}^{h}(\varphi_{\eps}) \varphi &= \int_{\Omega} \frac{1}{2} \alpha'_{\eps}(\varphi_{\eps})\varphi \abs{\bm{u}_{\eps}}^{2} + \alpha_{\eps}(\varphi_{\eps}) \bm{u}_{\eps} \cdot \bm{u} \dx \\
& + \frac{\gamma}{2c_{0}} \int_{\Omega} \eps \nabla \varphi_{\eps} \cdot \nabla \varphi + \frac{1}{\eps} \psi'(\varphi_{\eps}) \varphi \dx \\
& + \int_{\Omega} \mathcal{M}(\varphi_{\eps})(\der_{2}h, \der_{3}h, \der_{4}h)\mid_{(x, \nabla \bm{u}_{\eps}, p_{\eps}, \nabla \varphi_{\eps})} \cdot (\nabla \bm{u}, p, \nabla \varphi) \dx \\
& + \int_{\Omega} h(x, \nabla \bm{u}_{\eps}, p_{\eps}, \nabla \varphi_{\eps}) \mathcal{M}'(\varphi_{\eps}) \varphi \dx. \end{aligned}
\end{equation}
where $\bm{S}_{\eps}(\varphi_{\eps}) = \{ (\bm{u}_{\eps}, p_{\eps})\}$ and $(\bm{u}, p) := \der \bm{S}_{\eps}(\varphi_{\eps}) \varphi$ is the solution of the linearized state equation \eqref{e:PhaseLinearizedState}.  Now we use the adjoint state $\bm{q}_{\eps}$ as a test function in the linearized state equation \eqref{e:PhaseLinearizedState} and find that
\begin{equation}\label{generalh:TestLinearizedWithAdjoint}\begin{aligned}
& \; \int_{\Omega} \alpha'_{\eps}(\varphi_{\eps})\varphi \bm{u}_{\eps} \cdot \bm{q}_{\eps} + \alpha_{\eps}(\varphi_{\eps}) \bm{u} \cdot \bm{q}_{\eps} + \mu \nabla \bm{u} \cdot \nabla \bm{q}_{\eps} \dx \\
+ & \; \int_{\Omega} (\bm{u} \cdot \nabla ) \bm{u}_{\eps} \cdot \bm{q}_{\eps} + (\bm{u}_{\eps} \cdot \nabla) \bm{u} \cdot \bm{q}_{\eps} + p \left ( \mathcal{M}(\varphi_{\eps}) \der_{3}h - \vartheta_{\eps} \right )\dx = 0,
\end{aligned}
\end{equation}
where $\der_{3}h$ is evaluated at $(x, \nabla \bm{u}_{\eps}, p_{\eps}, \nabla \varphi_{\eps})$.  

Then we use the linearized state $\bm{u} \in \bm{H}^{1}_{0,\sigma}(\Omega)$ as a test function in \eqref{adjoint:weakform} and obtain
\begin{equation}\label{generalh:TestAdjointWithLinearized}
\begin{aligned}
&\; \int_{\Omega} \alpha_{\eps}(\varphi_{\eps}) \bm{q}_{\eps} \cdot \bm{u} + \mu \nabla \bm{q}_{\eps} \cdot \nabla \bm{u} + (\nabla \bm{u}_{\eps})^{T} \bm{q}_{\eps} \cdot \bm{u} - (\bm{u}_{\eps} \cdot \nabla) \bm{q}_{\eps} \cdot \bm{u} \dx \\
= & \; \int_{\Omega} \alpha_{\eps}(\varphi_{\eps}) \bm{u}_{\eps} \cdot \bm{u} + \mathcal{M}(\varphi_{\eps}) \left ( \der_{2}h \cdot \nabla \bm{u} \right ) \dx,
\end{aligned}
\end{equation}
where $\der_{2}h$ is evaluated at $(x, \nabla \bm{u}_{\eps}, p_{\eps}, \nabla \varphi_{\eps})$.

Comparing \eqref{generalh:TestLinearizedWithAdjoint} and \eqref{generalh:TestAdjointWithLinearized} yields the following identity
\begin{equation}\label{generalh:CompareAdjointAndLinearized}
\begin{aligned}
\int_{\Omega} \alpha_{\eps}'(\varphi_{\eps}) \varphi \bm{u}_{\eps} \cdot \bm{q}_{\eps} + \alpha_{\eps}(\varphi_{\eps}) \bm{u}_{\eps} \cdot \bm{u} + \mathcal{M}(\varphi_{\eps}) \left ( \der_{2}h \cdot \nabla \bm{u} + p \der_{3}h \right ) \dx = 0,
\end{aligned}
\end{equation}
where we have used that $p \in L^{2}_{0}(\Omega)$, $\div \bm{u}_{\eps} = 0$ in $\Omega$, $\bm{u} = \bm{q}_{\eps} = \bm{0}$ on $\pd \Omega$, and thus
\begin{align*}
\int_{\Omega} p \vartheta_{\eps} \dx = \vartheta_{\eps} \int_{\Omega} p \dx & = 0, \\
\int_{\Omega} (\bm{u}_{\eps} \cdot \nabla) \bm{q}_{\eps} \cdot \bm{u} + (\bm{u}_{\eps} \cdot \nabla) \bm{u} \cdot \bm{q}_{\eps} \dx = \int_{\Omega} \bm{u}_{\eps} \cdot \nabla (\bm{q}_{\eps} \cdot \bm{u}) \dx & = 0.
\end{align*}
Hence, by using \eqref{generalh:CompareAdjointAndLinearized}, we can rewrite \eqref{generalh:DerJEpsiloNFirstForm} as follows:
\begin{equation}\label{generalh:DerJEpsiloNSecondForm}
\begin{aligned}
\der j_{\eps}(\varphi_{\eps}) \varphi &= \int_{\Omega}  \alpha'_{\eps}(\varphi_{\eps})\varphi \left ( \frac{1}{2}\abs{\bm{u}_{\eps}}^{2} - \bm{u}_{\eps} \cdot \bm{q}_{\eps} \right ) + \frac{\gamma \eps}{2c_{0}} \nabla \varphi_{\eps} \cdot \nabla \varphi + \frac{\gamma}{2c_{0} \eps} \psi'(\varphi_{\eps}) \varphi \dx \\
& + \int_{\Omega} \mathcal{M}(\varphi_{\eps}) \der_{4}h(x, \nabla \bm{u}_{\eps}, p_{\eps}, \nabla \varphi_{\eps}) \cdot \nabla \varphi \dx\\
& + \int_{\Omega} h(x, \nabla \bm{u}_{\eps}, p_{\eps}, \nabla \varphi_{\eps}) \mathcal{M}'(\varphi_{\eps}) \varphi  \dx.
\end{aligned}
\end{equation}
Together with \eqref{generalh:ProofOptSysEpsVarEq}, this yields the statement of the theorem.
\end{proof}

The analogous optimality condition for the  optimization problem $\{\eqref{IntroStateEquPhaseWeak}, \eqref{ObjFunctHydroPhase}\}$ involving the hydrodynamic force \eqref{HydroDynamForce} is given as follows:
\begin{thm}\label{t:HydroDynamOptSys}
Let $(\varphi_{\eps}, \bm{u}_{\eps}, p_{\eps}) \in (\Phi_{ad} \cap L^{\infty}(\Omega)) \times \bm{H}^{1}_{\bm{g},\sigma}(\Omega) \times L^{2}_{0}(\Omega)$ be a minimizer of optimization problem $\{\eqref{IntroStateEquPhaseWeak}, \eqref{ObjFunctHydroPhase}\}$ involving the hydrodynamic force \eqref{HydroDynamForce}  with $\norm{\nabla \bm{u}_{\eps}}_{\bm{L}^{2}(\Omega)} < \frac{\mu}{K_{\Omega}}$, thus in particular, $\bm{S}_{\eps}(\varphi_{\eps}) =  \{(\bm{u}_{\eps}, p_{\eps})\}$.  Then the following optimality system is fulfilled:  There exists a Lagrange multiplier $\lambda_{\eps} \in \R$ for the integral constraint such that 
\begin{equation}\label{e:PhaseFieldOptSys}\begin{aligned}
&\; \left(\alpha'_{\eps}(\varphi_{\eps}) \left(\frac{1}{2} \abs{\bm{u}_{\eps}}^{2} - \bm{u}_{\eps} \cdot \bm{q}_{\eps} \right) + \frac{\gamma}{2c_{0} \eps} \psi'(\varphi_{\eps}) +\lambda_{\eps} + \mathcal{M}'(\varphi_{\eps}) \nabla \varphi_{\eps} \cdot \left( \bm{\sigma}_{\eps} \bm{a} \right), \zeta \right)_{L^{2}(\Omega)} \\
+ \; & \left( \mathcal{M}(\varphi_{\eps}) \bm{\sigma}_{\eps} \bm{a}  + \frac{\gamma \eps}{2c_{0}} \nabla \varphi_{\eps} , \nabla \zeta \right)_{\bm{L}^{2}(\Omega)} = 0 \quad \forall \zeta \in H^{1}(\Omega) \cap L^{\infty}(\Omega),
\end{aligned}
\end{equation}
where $\bm{\sigma}_{\eps} := \mu \left(\nabla \bm{u}_{\eps} + (\nabla \bm{u}_{\eps})^{T})  \right) - p_{\eps} \id$, and $(\bm{q}_{\eps}, \pi_{\eps}) \in \bm{H}^{1}_{0}(\Omega) \times L^{2}(\Omega)$ is the unique weak solution of the adjoint system 
\begin{subequations}\label{e:AdjointStrong}\begin{align}
\label{adjoint1}
\notag \alpha_{\eps} (\varphi_{\eps}) (\bm{q}_{\eps}& - \bm{u}_{\eps}) - \mu \nabla \cdot (\nabla \bm{q}_{\eps} + (\nabla \bm{q}_{\eps})^{T}) + (\nabla \bm{u}_{\eps})^{T} \bm{q}_{\eps} - (\bm{u}_{\eps} \cdot \nabla) \bm{q}_{\eps} + \nabla \pi_{\eps} \\
 & = - \mu \left ( \div \left ( \mathcal{M}(\varphi_{\eps}) \nabla \varphi_{\eps} \right ) \bm{a} - \nabla \left ( \mathcal{M}(\varphi_{\eps}) \nabla \varphi_{\eps} \right ) \bm{a} \right ) &&\text{in }\Omega,\\
\label{adjoint2}
\div \bm{q}_{\eps} & = \mathcal{M}(\varphi_{\eps}) \nabla \varphi_{\eps} \cdot \bm{a} - \strokedint_{\Omega} \mathcal{M}(\varphi_{\eps}) \nabla \varphi_{\eps} \cdot \bm{a} \dx && \text{in }\Omega, \\
\label{adjoint3}
 \bm{q}_{\eps} & = \bm{0} && \text{on }\pd \Omega.
\end{align}\end{subequations}
\end{thm}
\begin{proof}
Note that for the hydrodynamic force \eqref{HydroDynamForce}:
\begin{align*}
h(x,  \nabla \bm{u}_{\eps}, p_{\eps}, \nabla \varphi_{\eps}) = \nabla \varphi_{\eps} \cdot (\mu(\nabla \bm{u}_{\eps} + (\nabla \bm{u}_{\eps})^{T}) - p_{\eps} \id) \cdot \bm{a},
\end{align*}
and so we compute that
\begin{align*}
 \der_{2}h & = \mu \left ( \nabla \varphi_{\eps} \otimes \bm{a} + \bm{a} \otimes \nabla \varphi_{\eps} \right ), \quad \der_{3}h = -\bm{a} \cdot \nabla \varphi_{\eps}, \\ 
\der_{4}h & = (\mu (\nabla \bm{u}_{\eps} + (\nabla \bm{u}_{\eps})^{T}) - p_{\eps} \id) \bm{a}.
\end{align*}
As $\bm{a}$ is a constant vector, \eqref{equ:partialderivativesh} in Assumption \ref{assump:regularityh} is satisfied and the statements follow from the application of Theorem \ref{t:GeneralisationOptimality}.
\end{proof}

\begin{rem}\label{r:phaseStrong}
After using integration by parts, we find that we can rewrite the gradient equation \eqref{e:PhaseFieldOptSys} for the hydrodynamic force formally in the strong form as 	
\begin{align}\label{e:PhaseFieldGradientEquStrong}
-\frac{\gamma}{2c_{0}} \left ( \eps \Laplace \varphi_{\eps} - \frac{1}{\eps} \psi'(\varphi_{\eps}) \right) + \lambda_{\eps} + \alpha'_{\eps}(\varphi_{\eps}) \left ( \frac{1}{2} \abs{\bm{u}_{\eps}}^{2} - \bm{u}_{\eps} \cdot \bm{q}_{\eps} \right )  - \mathcal{M}(\varphi_{\eps}) \div (\bm{\sigma}_{\eps} \bm{a}) =0 \text{ in }\Omega,
\end{align}
with the boundary condition
\begin{align}\label{e:PhaseFieldGradientEquStrong:Bdy}
\frac{\gamma}{2c_{0}} \eps \nabla \varphi_{\eps} \cdot \bm{\nu}_{\pd \Omega} + \mathcal{M}(\varphi_{\eps}) \bm{\nu}_{\pd \Omega} \cdot (\bm{\sigma}_{\eps} \bm{a}) = 0 \text{ on }\pd \Omega.
\end{align}

Moreover, with sufficiently smooth solutions, we can make use of the state equation \eqref{state1} to rewrite \eqref{e:PhaseFieldGradientEquStrong} as:
\begin{equation}\label{e:PhaseFieldGradientEquStrongRewrite}
\begin{aligned}
& \; -\frac{\gamma}{2c_{0}} \left ( \eps \Laplace \varphi_{\eps} - \frac{1}{\eps} \psi'(\varphi_{\eps}) \right) + \lambda_{\eps} + \alpha'_{\eps}(\varphi_{\eps}) \left ( \frac{1}{2} \abs{\bm{u}_{\eps}}^{2} - \bm{u}_{\eps} \cdot \bm{q}_{\eps} \right ) \\
+ & \; \mathcal{M}(\varphi_{\eps}) \left( \bm{f} - \alpha_{\eps}(\varphi_{\eps}) \bm{u}_{\eps} - (\bm{u}_{\eps} \cdot \nabla ) \bm{u}_{\eps} \right ) \cdot \bm{a} = 0.
\end{aligned}
\end{equation}
\end{rem}

\begin{rem}\label{r:Dirichletbdy}
We note that the above analysis of \eqref{IntroObjFcltPhase}-\eqref{IntroStateEquPhaseWeak} can be modified to include a Dirichlet condition for the design function $\varphi_{\eps}$ on $\pd \Omega$, for instance $\varphi_{\eps} = 1$ on $\pd \Omega$.  This amounts to changing the space of admissible design functions to
\begin{align*}
\Phi_{ad} = \left \{ \varphi \in H^{1}(\Omega) \mid \int_{\Omega} \varphi \dx = \beta \abs{\Omega} \text{ and } \varphi = 1 \text{ on } \pd \Omega \right \}.
\end{align*}
Then, in the optimality conditions \eqref{generalh:PhaseFieldOptSys} and \eqref{e:PhaseFieldOptSys}, and also in \eqref{generalh:ProofOptSysEpsVarIn} and \eqref{generalh:ProofOptSysEpsVarEq}, we use test functions $\zeta \in H^{1}_{0}(\Omega) \cap L^{\infty}(\Omega)$, and $\varphi \in H^{1}_{0}(\Omega)$.   Moreover, from Remark \ref{r:phaseStrong}, the strong form of the resulting gradient equation \eqref{e:PhaseFieldOptSys} remains as \eqref{e:PhaseFieldGradientEquStrong} $($or \eqref{e:PhaseFieldGradientEquStrongRewrite}$)$, but now with the boundary condition 
\begin{align*}
\varphi_{\eps} = 1 \text{ on } \pd \Omega.
\end{align*}
\end{rem}

\section{Sharp interface asymptotics for the hydrodynamic force}\label{sec:SharpInterfaceAsymp}
In Section~\ref{sec:DerivationPhaseField}, we introduced the diffuse interface problem \eqref{IntroObjFcltPhase}-\eqref{IntroStateEquPhaseWeak} as an approximation of the shape optimization problem \eqref{IntroObjFclt}-\eqref{IntroStateEquSharp} for a general functional $h$.  In Section~\ref{sec:AnalysisPhaseField}, the existence of a minimizer $(\varphi_{\eps}, \bm{u}_{\eps}, p_{\eps})$ to \eqref{IntroObjFcltPhase}-\eqref{IntroStateEquPhaseWeak} for every fixed $\eps > 0$ is guaranteed by Theorem \ref{t:PhaseFieldExistMin}, and the first order necessary optimality condition is given in Theorem \ref{t:GeneralisationOptimality}.  The analogous results for the hydrodynamic force problem $\{\eqref{IntroStateEquPhaseWeak}, \eqref{ObjFunctHydroPhase}\}$ are also presented in Theorem \ref{t:HydroDyanExistMin} and Theorem \ref{t:HydroDynamOptSys}.

In this section, we focus only the hydrodynamic force problem $\{\eqref{IntroStateEquPhaseWeak}, \eqref{ObjFunctHydroPhase}\}$ and carry out a sharp interface limit of the system $\{\eqref{IntroStateEquPhase}, \eqref{e:AdjointStrong}, \eqref{e:PhaseFieldGradientEquStrongRewrite} \}$ by the method of formally matched asymptotic expansions.  We hereby recover the optimality conditions expected by classical shape sensitivity analysis presented in Section \ref{sec:SharpInterfaceProblem} in the limit $\eps \searrow 0$.  For an introduction and more detailed discussion of the techniques and basic assumptions used in the method of formally matched asymptotic analysis we refer for instance to \cite{article:FifePenrose95, article:GarckeStinner06}.  

In the asymptotic analysis, we assume there are sufficient smooth solutions to the system $\{\eqref{IntroStateEquPhase}, \eqref{e:AdjointStrong}, \eqref{e:PhaseFieldGradientEquStrong}\}$, and hence we consider \eqref{e:PhaseFieldGradientEquStrongRewrite} instead of \eqref{e:PhaseFieldGradientEquStrong} in the sequel as the analysis is comparatively easier.

\begin{assumption}
We assume that for small $\eps$, the domain $\Omega$ can be divided into two open subdomains $\Omega^{\pm}(\eps)$, separated by an interface $\Gamma(\eps)$.  Furthermore, we assume that there is a family $(\varphi_{\eps}, \bm{u}_{\eps}, p_{\eps}, \bm{q}_{\eps}, \pi_{\eps}, \lambda_{\eps}, \vartheta_{\eps})_{\eps > 0}$ of solutions to $\{\eqref{IntroStateEquPhase}, \eqref{e:AdjointStrong}, \eqref{e:PhaseFieldGradientEquStrongRewrite}\}$, which are sufficiently smooth and have an asymptotic expansion in $\eps$ in the bulk regions away from $\Gamma(\eps)$ (the outer expansion, see Section \ref{sec:OuterExp}), and another expansion in the interfacial region  (inner expansions, see Section \ref{sec:InnerExp}), see also \cite{article:FifePenrose95,article:GarckeStinner06} for a detailed formulation.
\end{assumption}

For the remainder of this section, we will make use of the following assumptions extensively:
\begin{assumption}
The correction constant $\delta_{\eps}$ and the interpolation function $\alpha_{\eps}$ fulfill
\begin{align*}
\delta_{\eps} = \eps^{k}, \, k > 1, \quad \alpha_{\eps}(t) = \frac{1}{\eps} \hat{\alpha}(t),
\end{align*}
where $\hat\alpha\in C^{1,1}(\R)\cap L^\infty(\R)$ satisfies the following properties:
\begin{align}\label{assump:asymptotics:alpha}
\hat{\alpha}(-1)>0,\quad \hat{\alpha}(1) = \hat{\alpha}'(1) = 0, \quad \hat{\alpha}(t) \neq 0 \text{ for } t \neq 1.
\end{align}
Moreover, we assume that the potential $\psi \in C^{2}(\R)$ satisfies:
\begin{align}\label{assump:asymptotics:psi}
\psi(\pm 1) = \psi'(\pm 1) = 0.
\end{align}
\end{assumption}

For the terms involving the square root, we make use of the following expansion for $a = a_{0} + \eps a_{1} + \eps^{2} a_{2} + \ldots$, which holds due to Taylor's theorem:
\begin{equation}\label{taylorexpansion}
\begin{aligned}
\sqrt{a + \delta_{\eps}} &= \sqrt{a_{0} +\eps a_{1} + \ldots + \eps^{k}(a_{k} + 1 ) + \ldots} \\
& = \sqrt{a_{0}} + \frac{1}{2 \sqrt{a_{0}}} \left [ \eps a_{1} + \ldots + \eps^{k}(a_{k} + 1) + \ldots \right ] \\
& - \frac{1}{4 \sqrt{a_{0}^{3}}} \left [ \eps a_{1} + \ldots + \eps^{k} (a_{k} + 1) + \ldots \right ]^{2} + \ldots.
\end{aligned}
\end{equation}

\subsection{Outer expansions}\label{sec:OuterExp}
We assume that for $v_{\eps} \in \{\varphi_{\eps}, \bm{u}_{\eps}, p_{\eps}, \lambda_{\eps},  \vartheta_{\eps}, \bm{q}_{\eps}, \pi_{\eps} \}$, the following outer expansions hold:
\begin{align*}
v_{\eps} = v_{0} + \eps v_{1} + \dots.
\end{align*}
Applying Taylor's theorem and \eqref{taylorexpansion}, for the choice $\mathcal{M}(\varphi_{\eps}) = \frac{1}{\sqrt{2} c_{0}} \sqrt{\psi(\varphi_{\eps}) + \delta_{\eps}}$, we obtain following outer expansion
\begin{equation}\label{sqrtPsiOuter} 
\begin{aligned}
& \; \mathcal{M}(\varphi_{\eps}) = \mathcal{M}(\varphi_{0} + \eps \varphi_{1} + \dots) \\
= & \; \frac{1}{\sqrt{2} c_{0}} \left ( \sqrt{ \psi(\varphi_{0}) + \psi'(\varphi_{0}) (\eps \varphi_{1} + l\dots) + \ldots + \psi^{(k)}(\varphi_{0})(\eps \varphi_{1} + \dots)^{k} + \ldots } \right ) \\
= & \; \frac{1}{ \sqrt{2} c_{0}} \left ( \sqrt{\psi(\varphi_{0})} + \frac{\eps \psi'(\varphi_{0}) \varphi_{1} }{2 \sqrt{\psi(\varphi_{0})}}+ \mathcal{O}(\eps^{2}) \right ) =: \mathcal{M}_{0}(\varphi_{0}) + \eps \mathcal{M}_{1}(\varphi_{0}) \varphi_{1} + \text{ h.o.t.}.
\end{aligned}
\end{equation}
We remark that, for the classical smooth double-well potential $\psi(\varphi) = \frac{1}{4}(1-\varphi^{2})^{2}$, one can compute that
\begin{align*}
\lim_{s \searrow -1} \frac{\psi'(s)}{\sqrt{\psi(s)}} = 2, \quad \lim_{s \nearrow 1} \frac{\psi'(s)}{\sqrt{\psi(s)}} = -2,
\end{align*} 
and so $\mathcal{M}_{1}(\pm 1)$ is well-defined for the smooth double-well potential.

We denote $(\cdot)_{O}^{\beta}$ to be the order $\beta$ outer expansions of equation $(\cdot)$. 

To leading order $(\ref{state1})_{O}^{-1}$ gives
\begin{align}\label{state1:leadingorder}
\hat{\alpha}(\varphi_{0}) \bm{u}_{0} = \bm{0}.
\end{align}
By (\ref{assump:asymptotics:alpha}), if $\varphi_{0} \neq 1$, we then obtain $\bm{u}_{0} = \bm{0}$.  
Similarly, to leading order $(\ref{adjoint1})_{O}^{-1}$ gives
\begin{align}\label{adjoint1:leadingorder}
\hat{\alpha}(\varphi_{0}) \bm{q}_{0} = \hat{\alpha}(\varphi_{0}) \bm{u}_{0}.
\end{align}
Thus, if $\varphi_{0} \neq +1$, then $\bm{q}_{0} = \bm{u}_{0} = \bm{0}$.

Meanwhile, $(\ref{state2})_{O}^{0}$, $(\ref{state3})_{O}^{0}$, and $(\ref{adjoint3})_{O}^{0}$ give
\begin{equation*}
\begin{aligned}
\div \bm{u}_{0}  = 0 &\text{ in } \Omega, \\
\bm{u}_{0} = \bm{g}, \quad \bm{q}_{0} = \bm{0} &\text{ on } \pd \Omega.
\end{aligned}
\end{equation*}

To order $-1$, $(\ref{e:PhaseFieldGradientEquStrongRewrite})_{O}^{-1}$ gives
\begin{align}\label{phase:leadingorder}
\hat{\alpha}'(\varphi_{0}) \left ( \frac{1}{2} \abs{\bm{u}_{0}}^{2} - \bm{u}_{0} \cdot \bm{q}_{0} \right ) = -\frac{\gamma}{2c_{0}} \psi'(\varphi_{0}).
\end{align}
If $\varphi_{0} \neq 1$, then from (\ref{state1:leadingorder}), (\ref{adjoint1:leadingorder}), and (\ref{assump:asymptotics:alpha}), we have that
\begin{align}\label{psiroots}
-\psi'(\varphi_{0}) = 0.
\end{align}
Hence, $\varphi_{0}$ must be a piecewise constant function that takes values equal to the roots of $\psi'(\cdot)$.  The stable solutions to (\ref{psiroots}) are $\varphi_{0} = \pm 1$.  In particular, we can define the fluid region and the solid region by
\begin{align*}
E := \{x \in \Omega \mid \varphi_{0}(x) = 1 \},\quad B := \{x \in \Omega \mid \varphi_{0}(x)= -1 \},
\end{align*}
respectively.  Moreover, from (\ref{state1:leadingorder}) and (\ref{adjoint1:leadingorder}) we have
\begin{align}\label{s:OuterExpU0ZeroInOmegas}\bm{u}_{0} = \bm{q}_{0} = \bm{0} \text{ in } B.
\end{align}
Furthermore, as $\varphi_{0} = \pm 1$, we have $\nabla \varphi_{0} = \bm{0}$ in $E$ and $B$, and so, from the definition \eqref{defn:vartheta} that $\vartheta_{0} = 0$.  From $(\ref{adjoint2})_{O}^{0}$  we have
\begin{align}\label{adjoint2:leadingorder}
\div \bm{q}_{0} = 0 &\text{ in } E \cup B.
\end{align}

The next order $(\ref{state1})_{O}^{0}$ gives
\begin{align}\label{state1:firstorder}
\hat{\alpha}'(\varphi_{0}) \varphi_{1} \bm{u}_{0} + \hat{\alpha}(\varphi_{0}) \bm{u}_{1} - \mu \Laplace \bm{u}_{0} + (\bm{u}_{0} \cdot \nabla) \bm{u}_{0} + \nabla p_{0} = \bm{f}.
\end{align}
By (\ref{assump:asymptotics:alpha}), for $\varphi_{0} = 1$, we obtain
\begin{align}\label{NSbulk}
- \mu \Laplace \bm{u}_{0} + (\bm{u}_{0} \cdot \nabla) \bm{u}_{0} + \nabla p_{0} = \bm{f} \text{ in } E.
\end{align}
Similarly, $(\ref{adjoint1})_{O}^{0}$ gives
\begin{align}\label{adjoint1:firstorder}
\notag & \; \hat{\alpha}'(\varphi_{0}) \varphi_{1} (\bm{q}_{0} - \bm{u}_{0}) + \hat{\alpha}(\varphi_{0}) (\bm{q}_{1} - \bm{u}_{1}) - \mu \div ( \nabla \bm{q}_{0} + (\nabla \bm{q}_{0})^{T}) \\
+ & \; (\nabla \bm{u}_{0})^{T} \bm{q}_{0} - (\bm{u}_{0} \cdot \nabla) \bm{q}_{0} + \nabla \pi_{0} = \bm{0}.
\end{align}
For $\varphi_{0} = 1$, we obtain
\begin{align*}
- \mu \Laplace \bm{q}_{0} + (\nabla \bm{u}_{0})^{T} \bm{q}_{0} - (\bm{u}_{0} \cdot \nabla) \bm{q}_{0} + \nabla \pi_{0} = \bm{0} \text{ in } E,
\end{align*}
where we have used (\ref{adjoint2:leadingorder}) to simplify the divergence term.

\subsection{Inner expansions and matching conditions}\label{sec:InnerExp}
Now we consider the interfacial region, i.e. near some free boundary $\Gamma= \pd E \cap\pd B$ which is assumed to be the limiting hypersurface of the zero level sets of $\varphi_{\eps}$. For studying the limiting behaviour in these parts of $\Omega$ we introduce new coordinates.  For this purpose we introduce the signed distance function $d(x)$ to $\Gamma$ and set $z = \frac{d}{\eps}$ as the rescaled distance variable.  Here we use the sign convention $d(x)>0$ if $x \in E$.

Let $\gamma(s)$ denote a parametrization of $\Gamma$ by arc-length $s$, and let $\bm{\nu}$ denote the outward unit normal of $\Gamma$.  Then, in a tubular neighbourhood of $\Gamma$, for sufficiently smooth function $v(x)$, we have
\begin{align*}
v(x) = v(\gamma(s) + \eps z \bm{\nu}(\gamma(s))) =: V(s,z).
\end{align*}
In this new $(s,z)$-coordinate system, the following change of variables apply, see \cite{article:GarckeStinner06}:
\begin{align*}
\nabla_{x} v = \frac{1}{\eps}\pd_{z} V \bm{\nu} + \nabla_{\Gamma} V + \text{ h.o.t.},
\end{align*}
where $\nabla_{\Gamma}f$ denotes the surface gradient of $f$ on $\Gamma$ with components $(\underline{D}_{k}f)_{1 \leq k \leq d}$ and h.o.t. denotes higher order terms with respect to $\eps$.  Moreover, if $\bm{v}$ is a vector-valued function, then we obtain
\begin{align*}
\div_{x} \bm{v} = \frac{1}{\eps}\pd_{z} \bm{V} \cdot \bm{\nu} + \div_{\Gamma} \bm{V} + \text{ h.o.t.}.
\end{align*}
In particular, using the fact that the normal $\bm{\nu}$ is independent of $z$, we have
\begin{align*}
\Laplace v  = \div_{x} (\nabla_{x} v) & = \frac{1}{\eps^{2}} \pd_{zz}V + \frac{1}{\eps}\underbrace{\div_{\Gamma} (\pd_{z}V \bm{\nu})}_{= - \kappa \pd_{z}V} + \text{ h.o.t.},
\end{align*}
where $\kappa = - \div_{\Gamma} \bm{\nu}$ is the mean curvature.

We denote the variables  $\varphi_{\eps}$, $\bm{u}_{\eps}$, $p_{\eps}$, $\bm{q}_{\eps}$, $\pi_{\eps}$ in the new coordinate system by $\Phi_{\eps}$, $\bm{U}_{\eps}$, $P_{\eps}$, $\bm{Q}_{\eps}$, $\Pi_{\eps}$.  We further assume that they have the following inner expansions:
\begin{align*}
V_{\eps}(s,z) = V_{0}(s,z) + \eps V_{1}(s,z) + \ldots,
\end{align*}
for $V_{\eps} \in \{ \Phi_{\eps}, \bm{U}_{\eps}, P_{\eps}, \bm{Q}_{\eps}, \Pi_{\eps} \}$.  We then obtain, 
\begin{align*}
\mathcal{M}(\Phi_{\eps}) = \mathcal{M}_{0}(\Phi_{0}) + \eps \mathcal{M}_{1}(\Phi_{0}) \Phi_{1} + \text{ h.o.t.},
\end{align*}
where $\mathcal{M}_{0}$, $\mathcal{M}_{1}$ are as defined in \eqref{sqrtPsiOuter} if we consider $\mathcal{M}(\varphi) = \frac{1}{c_{0}} \sqrt{\frac{\psi(\varphi) + \delta_{\eps}}{2}}$.

We remark that, for a sufficiently smooth function $\bm{f}$ independent of $\eps$, 
\begin{align*}
\bm{f}(x) & = \bm{f}(\gamma(s) + \eps z \bm{\nu}(s)) = \bm{f}(\gamma(s)) + \eps z \nabla \bm{f}(\gamma(s)) \cdot \bm{\nu} + \text{ h.o.t. } \\
& =: \bm{F}_{0}(s) + \eps \bm{F}_{1}(s,z) + \text{ h.o.t.},
\end{align*}
for $x$ in a neighbourhood of $\Gamma$.  As a consequence, we see that
\begin{align}\label{pdzF0}
\pd_{z} \bm{F}_{0} = \bm{0}.
\end{align}
As the Lagrange multipliers $\lambda_{\eps}$ and $\vartheta_{\eps}$ are constant, we assume that the inner expansions are the same as the outer expansions.  In particular, the leading order expansions of the Lagrange multipliers do not depend on $z$.

The assumption that the zero level set of $\varphi_{\eps}$ converge to $\Gamma$ implies that
\begin{align}\label{e:CondPhi}
\Phi_{0}(0) = 0.
\end{align}
In order to match the inner expansions valid in the interfacial region to the outer expansions of Section \ref{sec:OuterExp} we employ the matching conditions, (for the derivation we refer to \cite[Appendix D]{article:GarckeStinner06}):
\begin{align}
\label{e:MatchingCond1}
\lim_{z \to \pm \infty} V_{0}(s,z) &= v_{0}^{\pm}, \\
\label{e:MatchingCond2}
\lim_{z \to \pm\infty} \pd_{z}V_{0}(s,z) &= 0 ,\\
\label{e:MatchingCond3}
\lim_{z \to \pm \infty} \pd_{z} V_{1}(s,z) &= \nabla v_{0}^{\pm} \cdot \bm{\nu}, \\
\label{e:MatchingCond4}
\lim_{z \to \pm \infty} \pd_{zz} V_{2}(s,z) &= \left(\left(\bm{\nu} \cdot \nabla \right)\left(\bm{\nu} \cdot \nabla \right)u_{0}^{\pm} \right) = \pd_{\bm{\nu}}(\pd_{\bm{\nu}} u_{0}^{\pm}) ,
\end{align}
where $v_{0}^{\pm}:= \lim_{\delta \searrow 0} v_{0}(p \pm \delta \bm{\nu})$ for $p \in \Gamma$.  Then (\ref{e:MatchingCond3}) and (\ref{e:MatchingCond4}) for vector-valued functions read as
\begin{align*}
\lim_{z \to \pm \infty} \pd_{z} \bm{V}_{1}(s,z) = \pd_{\bm{\nu}} \bm{v}_{0}^{\pm}, \quad \lim_{z \to \pm \infty} \pd_{zz} \bm{V}_{2}(s,z) = \bm{\nu} \cdot \nabla (\pd_{\bm{\nu}} \bm{v}_{0}^{\pm}) = \pd_{\bm{\nu}} (\pd_{\bm{\nu}} \bm{v}_{0}^{\pm}).
\end{align*}
As $\div \bm{u}_{\eps} = 0$, we can rewrite
\begin{align*}
\Laplace \bm{u}_{\eps} = \div ( \nabla \bm{u}_{\eps} + (\nabla \bm{u}_{\eps})^{T}).
\end{align*}
For a tensor $\bm{A}$, let $\mathcal{E}(\bm{A}) = \frac{1}{2}(\bm{A} + \bm{A}^{T})$.  Then we can compute
\begin{align*}
\Laplace \bm{u}_{\eps} & = \frac{2}{\eps^{2}} \pd_{z} ( \mathcal{E}(\pd_{z} \bm{U}_{\eps} \otimes \bm{\nu}) \bm{\nu}) + \frac{2}{\eps} \pd_{z}(\mathcal{E}(\surf \bm{U}_{\eps}) \bm{\nu}) + \frac{2}{\eps} \div_\Gamma(\mathcal{E}(\pd_{z} \bm{U}_{\eps} \otimes \bm{\nu}) + \ldots \\
& = \frac{1}{\eps^{2}} \pd_{zz} \bm{U}_{\eps} + \frac{1}{\eps^{2}}\pd_{z}(\pd_{z} \bm{U}_{\eps} \cdot \bm{\nu}) \bm{\nu} + \frac{2}{\eps}  \pd_{z}(\mathcal{E}(\surf \bm{U}_{\eps}) \bm{\nu}) + \frac{2}{\eps} \div_\Gamma(\mathcal{E}(\pd_{z} \bm{U}_{\eps} \otimes \bm{\nu})  + \ldots.
\end{align*}
We note that the same expansion holds for the divergence term in (\ref{adjoint1}).  

Similarly as in Section \ref{sec:OuterExp}, we will denote $(\cdot)_{I}^{\beta}$ to be the order $\beta$ inner expansions of equation $(\cdot)$.

\subsubsection{Inner expansions of the state equations}
To order $-1$, $(\ref{state2})_{I}^{-1}$ gives
\begin{align}\label{state2:innerleadingorder}
\pd_{z}\bm{U}_{0} \cdot \bm{\nu} = \pd_{z}(\bm{U}_{0} \cdot \bm{\nu}) = 0,
\end{align}
while to leading order $(\ref{state1})_{I}^{-2}$ gives
\begin{align}
\label{state1:innerleadingorder}
-\mu \pd_{z} ( \pd_{z} \bm{U}_{0} + (\pd_{z} \bm{U}_{0} \cdot \bm{\nu}) \bm{\nu}) = -\mu \pd_{zz} \bm{U}_{0} = \bm{0},
\end{align}
where we have used (\ref{state2:innerleadingorder}).  Integrating with respect to $z$ from $-\infty$ to $z$ and applying the matching condition \eqref{e:MatchingCond2} leads to
\begin{align}
\label{pdzU0zero}
\pd_{z} \bm{U}_{0}(s,z) = \bm{0},
\end{align}
and so $\bm{U}_{0}$ is independent of $z$. Integrating once more with respect to $z$ from $-\infty$ to $z$ and by the matching condition $\eqref{e:MatchingCond1}$, we hence find that
\begin{align}
\label{U0zero}
\bm{U}_{0}(s,z) \equiv \bm{u}_{0}^{-} = \bm{0},
\end{align}
where we made in particular use of $\eqref{s:OuterExpU0ZeroInOmegas}$. This implies
\begin{align}\label{velocontinuous}
\bm u_{0}^{+} = \bm{u}_{0}^{-} = \bm{0}.
\end{align}

To first order $(\ref{state2})_{I}^{0}$ gives
\begin{align}
\label{state2:innerfirstorder}
\pd_{z}\bm{U}_{1} \cdot \bm{\nu} + \div_\Gamma\bm{U}_{0} = \pd_{z} \bm{U}_{1} \cdot \bm{\nu} = 0,
\end{align}
where we have used (\ref{U0zero}). Using (\ref{U0zero}) and (\ref{state2:innerfirstorder}), to first order $(\ref{state1})_{I}^{-1}$ gives
\begin{align}
\label{state1:innerfirstorder:simple}
- \mu \pd_{zz} \bm{U}_{1} + \pd_{z} P_{0} \bm{\nu} = \bm{0}.
\end{align}


\subsubsection{Phase field equation to leading order}
To leading order $(\ref{e:PhaseFieldGradientEquStrongRewrite})_{I}^{-1}$ gives
\begin{equation}\label{phase:innerleadingorder}
\begin{aligned}
 -\frac{\gamma}{2c_{0}} (\pd_{zz}\Phi_{0} -  \psi'(\Phi_{0})) + \hat{\alpha}'(\Phi_{0})(\tfrac{1}{2} \abs{\bm{U}_{0}}^{2} - \bm{U}_{0} \cdot \bm{Q}_{0}) - \mathcal{M}(\Phi_{0}) \alpha(\Phi_{0}) \bm{U}_{0} \cdot \bm{a} = 0
\end{aligned}
\end{equation}
Using (\ref{U0zero}), the above simplifies to
\begin{align}
\label{ODE}
\pd_{zz}\Phi_{0} - \psi'(\Phi_{0}) = 0.
\end{align}
Along with the matching conditions $\eqref{e:MatchingCond1}$ for $\Phi_{0}$:
\begin{align*}
\Phi_{0}(s, z = \pm \infty) = \pm 1,
\end{align*}
we can choose $\Phi_{0}$ to be independent of $s$ and as the unique monotone solution to (\ref{ODE}) satisfying $\Phi_{0}(z=0) = 0$ (recall $\eqref{e:CondPhi}$).  Moreover, taking the product of (\ref{ODE}) with $\Phi_{0}'(z)$ and integrating with respect to $z$ from $-\infty$ to $z$ leads to the so-called equipartition of energy after matching:
\begin{align}
\label{equipartition}
\frac{1}{2} \abs{\Phi_{0}'(z)}^{2} = \psi(\Phi_{0}(z)) \text{ for } \abs{z} < \infty.
\end{align}
Moreover, a short calculation using \eqref{equipartition}, the monotonicity of $\Phi_{0}$, and a change of variables $s \mapsto \Phi_{0}(z)$ shows that
\begin{align}
\label{const:c0}
c_{0} = \frac{1}{2} \int_{-1}^{1} \sqrt{2 \psi(s)} \ds = \frac{1}{2} \int_{\R} \sqrt{2\psi(\Phi_{0}(z))} \Phi_{0}'(z) \dz = \frac{1}{2} \int_{\R} \abs{\Phi_{0}'(z)}^{2} \dz. 
\end{align}

\subsubsection{Inner expansions of the adjoint equation}
Before we analyse the adjoint equation, we first compute:
\begin{equation}\label{adjoint:RHSterm1:expansion}
\begin{aligned}
\div (\mathcal{M}(\varphi_{\eps}) \nabla \varphi_{\eps}) & = \frac{1}{\eps^{2}} \pd_{z}(\mathcal{M}(\Phi_{\eps}) \pd_{z} \Phi_{\eps}) + \div_\Gamma \left ( \mathcal{M}(\Phi_{\eps}) \left ( \frac{1}{\eps} \pd_{z}\Phi_{\eps} \bm{\nu} + \surf \Phi_{\eps} \right ) \right) \\
& + \text{ h.o.t.},
\end{aligned}
\end{equation}
and for any $1 \leq j \leq d$,
\begin{align*}
& \; (\nabla (\mathcal{M}(\varphi_{\eps}) \nabla \varphi_{\eps}) \bm{a})_{j} = \sum_{i=1}^{d} \pd_{i} (\mathcal{M}(\varphi_{\eps}) \pd_{j} \varphi_{\eps}) a_{i} \\
= & \; \sum_{i=1}^{d} \frac{1}{\eps} \nu_{i} \pd_{z} \left ( \mathcal{M}(\Phi_{\eps}) \left ( \frac{1}{\eps} \pd_{z} \Phi_{\eps} \nu_{j} + \underline{D}_{j} \Phi_{\eps} \right ) \right )a_{i} + \underline{D}_{i} \left ( \mathcal{M}(\Phi_{\eps}) \left ( \frac{1}{\eps} \pd_{z}\Phi_{\eps} \nu_{j} + \underline{D}_{j} \Phi_{\eps} \right ) \right ) a_{i} + \text{ h.o.t.},
\end{align*}
so that
\begin{equation}\label{adjoint:RHSterm2:expansion}
\begin{aligned}
\nabla (\mathcal{M}(\varphi_{\eps}) \nabla \varphi_{\eps}) \bm{a} & = \frac{1}{\eps^{2}} (\bm{\nu} \cdot \bm{a}) \bm{\nu} \pd_{z} (\mathcal{M}(\Phi_{\eps}) \pd_{z} \Phi_{\eps})\\
& + \frac{1}{\eps} \left ( (\bm{\nu} \cdot \bm{a})  \pd_{z} (\mathcal{M}(\Phi_{\eps}) \surf \Phi_{\eps}) + \surf (\mathcal{M}(\Phi_{\eps}) \pd_{z} \Phi_{\eps} \bm{\nu}) \bm{a} \right )  \\
& + \surf (\surf \Phi_{\eps}) \bm{a} + \text{ h.o.t.}.
\end{aligned}
\end{equation}

To leading order $(\ref{adjoint2})_{I}^{-1}$ gives
\begin{align}
\label{adjoint2:innerleadingorder}
\pd_{z} \bm{Q}_{0} \cdot \bm{\nu} = \mathcal{M}_{0}(\Phi_{0}) \Phi_{0}' (\bm{\nu} \cdot \bm{a}),
\end{align}
while to leading order $(\ref{adjoint1})_{I}^{-2}$ gives
\begin{align}
\label{adjoint1:innerleadingorder}
- \mu \pd_{zz} \bm{Q}_{0} - \mu \pd_{z}(\pd_{z} \bm{Q}_{0} \cdot \bm{\nu}) \bm{\nu} = - \mu  \pd_{z} (\mathcal{M}_{0}(\Phi_{0}) \Phi_{0}'  ) ((\bm{\nu} \cdot \bm{a}) \bm{\nu} + \bm{a}),
\end{align}
where we have used \eqref{adjoint:RHSterm1:expansion}, \eqref{adjoint:RHSterm2:expansion} and that $\bm{\nu}$ is independent of $z$ to simplify the right hand side of $(\ref{adjoint1})_{I}^{-2}$.

Integrating \eqref{adjoint1:innerleadingorder} with respect to $z$ from $-\infty$ to $z$ and using the matching condition \eqref{e:MatchingCond2} leads to
\begin{align*}
\pd_{z} \bm{Q}_{0} + (\pd_{z}\bm{Q}_{0} \cdot \bm{\nu}) \bm{\nu} = \mathcal{M}_{0}(\Phi_{0}) \Phi_{0}' ((\bm{\nu} \cdot \bm{a}) \bm{\nu} + \bm{a}),
\end{align*}
and upon adding the product of  \eqref{adjoint2:innerleadingorder} with $\bm{\nu}$ leads to
\begin{align}
\label{pdzQ0}
\pd_{z} \bm{Q}_{0}(s,z) = \mathcal{M}_{0}(\Phi_{0}) \Phi_{0}' \bm{a}.
\end{align}

Integrating \eqref{pdzQ0} with respect to $z$ from $-\infty$ to $z$, using the matching condition \eqref{e:MatchingCond1} and $\bm{q}_{0}^{-} = \bm{0}$ (see \eqref{s:OuterExpU0ZeroInOmegas}), lead to
\begin{align}\label{Q0indeps}
\bm{Q}_{0}(s,z) = \left ( \int_{-\infty}^{z} \mathcal{M}_{0}(\Phi_{0}(z)) \Phi_{0}'(z) \dz \right ) \bm{a}.
\end{align}
In particular, the right hand side is independent of $s$, and so we can deduce that $\bm{Q}_{0}$ is also independent of $s$. Using the matching condition $\eqref{e:MatchingCond1}$, we hence have
\begin{align}
\label{q0Omegaplus}
\bm{q}_{0}^{+} =  \left ( \int_{\R} \mathcal{M}_{0}(\Phi_{0}(z)) \Phi_{0}'(z) \dz \right ) \bm{a}.
\end{align}

For the choice $\mathcal{M}(\varphi) = \frac{1}{2}$, we see that
\begin{align}
\int_{\R} \mathcal{M}_{0}(\Phi_{0}(z)) \Phi_{0}'(z) \dz = \frac{1}{2} \int_{\R} \Phi_{0}'(z) \dz = 1,
\end{align}
while for the choice $\mathcal{M}(\varphi) = \frac{1}{c_{0}} \sqrt{\frac{\psi(\varphi) + \delta_{\eps})}{2}}$, we see that by \eqref{sqrtPsiOuter}, \eqref{equipartition}, and \eqref{const:c0},
\begin{align*}
\int_{\R} \mathcal{M}_{0}(\Phi_{0}(z)) \Phi_{0}'(z) \dz = \frac{1}{c_{0}} \int_{\R} \frac{1}{\sqrt{2}} \sqrt{\psi(\Phi_{0}(z))} \Phi_{0}'(z) \dz = \frac{1}{c_{0}} \int_{\R} \frac{1}{2} \abs{\Phi_{0}'(z)}^{2} \dz = 1.
\end{align*}
Thus, in both cases, we obtain
\begin{align}\label{q0Omegaplus:complete}
\bm{q}_{0}^{+} = \bm{a}.
\end{align}

To the next order, we obtain from $(\ref{adjoint2})_{I}^{0}$
\begin{align}
\label{adjoint2:innerfirstorder}
\pd_{z} \bm{Q}_{1} \cdot \bm{\nu} =  \mathcal{M}_{0}(\Phi_{0}) \pd_{z} \Phi_{1} (\bm{\nu} \cdot \bm{a}) + \mathcal{M}_{1}(\Phi_{0}) \Phi_{1} \Phi_{0}' (\bm{\nu} \cdot \bm{a}),
\end{align}
where we used that $\bm{Q}_{0}$ and $\Phi_{0}$ are functions of $z$ only, and $\vartheta_{0} = 0$ from the outer expansions.  Meanwhile, from \eqref{adjoint:RHSterm1:expansion} and \eqref{adjoint:RHSterm2:expansion}, $(\ref{adjoint1})_{I}^{-1}$ gives
\begin{equation}\label{adjoint1:innerfirstorder}
\begin{aligned}
& \; \hat{\alpha}(\Phi_{0}) \bm{Q}_{0} - \mu \pd_{zz} \bm{Q}_{1} - \mu \pd_{z}(\pd_{z} \bm{Q}_{1} \cdot \bm{\nu}) \bm{\nu} - 2 \mu \pd_{z}(\mathcal{E}(\surf \bm{Q}_{0})) \bm{\nu} - 2 \mu \div_\Gamma( \mathcal{E}( \bm{Q}_{0}' \otimes \bm{\nu})) \\
= & \; -\mu (\bm{a} + (\bm{\nu} \cdot \bm{a}) \bm{\nu}) \pd_{z}(\mathcal{M}_{0}(\Phi_{0}) \pd_{z}\Phi_{1} + \mathcal{M}_{1}(\Phi_{0})\Phi_{1} \Phi_{0}') \\
- & \; \mu \div_{\Gamma} (\mathcal{M}_{0}(\Phi_{0}) \Phi_{0}' \bm{\nu}) \bm{a} - \mu \surf (\mathcal{M}_{0}(\Phi_{0}) \Phi_{0}' \bm{\nu}) \bm{a}.
\end{aligned}
\end{equation}
Moreover, we can simplify, thanks to fact that $\Phi_{0}$ and $\bm{Q}_{0}$ only depend on $z$:
\begin{align*}
2\div_\Gamma(\mathcal{E}( \bm{Q}_{0}' \otimes \bm{\nu})) & = \surf ( \bm{Q}_{0}' ) \bm{\nu} + (\div_{\Gamma} \bm{\nu}) \bm{Q}_{0}' + (\surf \bm{\nu}) \bm{Q}_{0}' + (\div_{\Gamma} \bm{Q}_{0}') \bm{\nu} \\
& =  -\kappa \bm{Q}_{0}' + (\surf \bm{\nu}) \bm{Q}_{0}', \\
\div_{\Gamma} (\mathcal{M}_{0}(\Phi_{0}) \Phi_{0}' \bm{\nu}) & = -\mathcal{M}_{0}(\Phi_{0}) \Phi_{0}' \kappa, \\
\surf(\mathcal{M}_{0}(\Phi_{0}) \Phi_{0}' \bm{\nu}) & = \mathcal{M}_{0}(\Phi_{0}) \Phi_{0}' \surf \bm{\nu}.
\end{align*}

Then, using the relation \eqref{pdzQ0}, we obtain from  \eqref{adjoint1:innerfirstorder}:
\begin{equation}\label{adjoint1:innerfirstorder:sim}
\begin{aligned}
& \; \hat{\alpha}(\Phi_{0}) \bm{Q}_{0} - \mu \pd_{zz} \bm{Q}_{1} - \mu \pd_{z}(\pd_{z} \bm{Q}_{1} \cdot \bm{\nu}) \bm{\nu} + \mu \kappa  \bm{Q}_{0}' - \mu(\surf \bm{\nu}) \bm{Q}_{0}'  \\
= & \;  -\mu (\bm{a} + (\bm{\nu} \cdot \bm{a}) \bm{\nu}) \pd_{z}(\mathcal{M}_{0}(\Phi_{0}) \pd_{z}\Phi_{1} + \mathcal{M}_{1}(\Phi_{0})\Phi_{1} \Phi_{0}') + \mu \bm{Q}_{0}' \kappa - \mu (\surf \bm{\nu}) \bm{Q}_{0}',
\end{aligned}
\end{equation}
and thus, upon cancelling the common terms, we have
\begin{equation}\label{adjoint1:innerfirstorder:sim:2}
\begin{aligned}
& \; \hat{\alpha}(\Phi_{0}) \bm{Q}_{0} - \mu \pd_{zz} \bm{Q}_{1} - \mu \pd_{z}(\pd_{z} \bm{Q}_{1} \cdot \bm{\nu}) \bm{\nu}  \\
= & \; -\mu (\bm{a} + (\bm{\nu} \cdot \bm{a}) \bm{\nu}) \pd_{z}(\mathcal{M}_{0}(\Phi_{0}) \pd_{z}\Phi_{1} + \mathcal{M}_{1}(\Phi_{0})\Phi_{1} \Phi_{0}').
\end{aligned}
\end{equation}

\subsubsection{Phase field equation to first order}
Using (\ref{U0zero}), we obtain from $(\ref{e:PhaseFieldGradientEquStrongRewrite})_{I}^{0}$ to first order:
\begin{equation}\label{phase:innerfirstorder}
\begin{aligned}
&  \frac{\gamma}{2c_{0}} \left ( -\pd_{zz} \Phi_{1} + \kappa  \Phi_{0}' + \psi''(\Phi_{0}) \Phi_{1} \right ) + \lambda_{0}  \\
& - \hat{\alpha}'(\Phi_{0}) \bm{U}_{1} \cdot \bm{Q}_{0} + \mathcal{M}_{0}(\Phi_{0}) \left ( \bm{F}_{0} -\hat{\alpha}(\Phi_{0}) \bm{U}_{1} \right ) \cdot \bm{a} = 0.
\end{aligned}
\end{equation}

Making use of (\ref{pdzQ0}), after taking the product of (\ref{phase:innerfirstorder}) with $\Phi_{0}'$ we have
\begin{equation}\label{phase:innerfirstorder:product}
\begin{aligned}
& \frac{\gamma}{2c_{0}} \left ( -\pd_{zz} \Phi_{1} \Phi_{0}' + \Phi_{1} (\psi'(\Phi_{0}))' + \kappa  \abs{\Phi_{0}'}^{2} \right ) + \lambda_{0} \Phi_{0}' \\
& - (\hat{\alpha}(\Phi_{0}))' \bm{U}_{1} \cdot \bm{Q}_{0} +  \bm{Q}_{0}' \cdot (\bm{F}_{0} - \hat{\alpha}(\Phi_{0}) \bm{U}_{1}) = 0.
\end{aligned}
\end{equation}

We note that by integrating by parts:
\begin{align*}
-\int_{\R} (\hat{\alpha}(\Phi_{0}))' \bm{U}_{1} \cdot \bm{Q}_{0} \dz = \int_{\R} \hat{\alpha}(\Phi_{0}) (\bm{U}_{1} \cdot \bm{Q}_{0}' + \pd_{z} \bm{U}_{1} \cdot \bm{Q}_{0}) \dz - [\hat{\alpha}(\Phi_{0}) \bm{U}_{1} \cdot \bm{Q}_{0}]_{z=-\infty}^{z=+\infty} .
\end{align*}
We use that $\hat{\alpha}(1) = 0$, $\bm{Q}_{0}(z = -\infty) = \bm{q}_{0}^{-} = \bm{0}$ to deduce that the jump term is zero.  Hence,
\begin{align}
\label{phase:innerfirstorder:integraltermU1Q0}
-\int_{\R} (\hat{\alpha}(\Phi_{0}))' \bm{U}_{1} \cdot \bm{Q}_{0} \dz = \int_{\R} \hat{\alpha}(\Phi_{0}) (\bm{U}_{1} \cdot \bm{Q}_{0}' + \pd_{z} \bm{U}_{1} \cdot \bm{Q}_{0}) \dz.
\end{align}

So, from integrating (\ref{phase:innerfirstorder:product}) over $\R$ and using (\ref{phase:innerfirstorder:integraltermU1Q0}) we obtain
\begin{equation}\label{phase:innerfirstorder:integrated}
\begin{aligned}
&\int_{\R} \frac{\gamma}{2c_{0}} \left  ( - \pd_{zz} \Phi_{1} \Phi_{0}' +  \Phi_{1} (\psi'(\Phi_{0}))' + \kappa \abs{\Phi_{0}'}^{2} \right ) + \lambda_{0} \Phi_{0}' \dz \\
& + \int_{\R} \hat{\alpha}(\Phi_{0}) \pd_{z} \bm{U}_{1} \cdot \bm{Q}_{0} + \bm{Q}_{0}' \cdot \bm{F}_{0} \dz = 0.
\end{aligned}
\end{equation}

Considering the first line, we find that, after integrating by parts and applying matching \eqref{e:MatchingCond1}-\eqref{e:MatchingCond2} for $\Phi_{0}$,
\begin{equation}\label{phase:innerfirstorder:line1}
\begin{aligned}
& \; \int_{\R} \frac{\gamma}{2c_{0}}  \left ( -\pd_{zz} \Phi_{1} \Phi_{0}' + \Phi_{1} (\psi'(\Phi_{0}))' + \kappa \abs{\Phi_{0}'}^{2} \right ) + \lambda_{0} \Phi_{0}' \dz \\
= & \; \frac{\gamma}{2c_{0}} \int_{\R} \pd_{z} \Phi_{1} \left( \Phi_{0}'' - \psi'(\Phi_{0}) \right) + \frac{\gamma}{2c_{0}} [\psi'(\Phi_{0}) \Phi_{1} - \Phi_{0}' \pd_{z}\Phi_{1}]_{z=-\infty}^{z=+\infty} \\
+ & \; \kappa \frac{\gamma}{2c_{0}} \underbrace{\int_{\R} \abs{\Phi_{0}'}^{2}\dz}_{=2c_{0}} + \lambda_{0}\underbrace{\int_{\R} \Phi_{0}' \dz}_{=2} = \kappa \gamma + 2 \lambda_{0},
\end{aligned}
\end{equation}
where we made use of \eqref{ODE}, the relation \eqref{const:c0}, and that $\kappa$ is independent of $z$. Thus it remains to identify
\begin{align}\label{phase:innerfirstorder:unknownterm}
\int_{\R}  \bm{F}_{0} \cdot \pd_{z} \bm{Q}_{0} + \hat{\alpha}(\Phi_{0}) \pd_{z} \bm{U}_{1} \cdot \bm{Q}_{0} \dz.
\end{align}

To this end, we take the scalar product of (\ref{adjoint1:innerfirstorder:sim:2}) with $\pd_{z} \bm{U}_{1}$ and use (\ref{state2:innerfirstorder}) to obtain
\begin{equation}\label{adjoint1:innerfirstorder:sim:multipliedpdzU1}
\begin{aligned}
& \; \hat{\alpha}(\Phi_{0}) \bm{Q}_{0} \cdot \pd_{z} \bm{U}_{1} - \mu \pd_{zz} \bm{Q}_{1} \cdot \pd_{z} \bm{U}_{1} \\
= & \; -\mu \pd_{z} \bm{U}_{1} \cdot \bm{a} \pd_{z}(\mathcal{M}_{0}(\Phi_{0}) \pd_{z}\Phi_{1} + \mathcal{M}_{1}(\Phi_{0})\Phi_{1} \Phi_{0}').
\end{aligned}
\end{equation}

Integrating \eqref{adjoint1:innerfirstorder:sim:multipliedpdzU1} over $\R$ with respect to $z$, and applying integration by parts leads to
\begin{equation}\label{adjoint1:innerfirstorder:integrated}
\begin{aligned}
& \; \int_{\R} \hat{\alpha}(\Phi_{0}) \bm{Q}_{0} \cdot \pd_{z} \bm{U}_{1} \dz \\
= & \; \mu \int_{\R} \pd_{zz} \bm{Q}_{1} \cdot \pd_{z} \bm{U}_{1} - \pd_{z} \bm{U}_{1} \cdot \bm{a} \pd_{z}(\mathcal{M}_{0}(\Phi_{0})\pd_{z} \Phi_{1} + \mathcal{M}_{1}(\Phi_{0}) \Phi_{1} \Phi_{0}') \dz \\
= & \; \mu \left [ \pd_{z} \bm{Q}_{1} \cdot \pd_{z} \bm{U}_{1} - \pd_{z} \bm{U}_{1} \cdot \bm{a} (\mathcal{M}_{0}(\Phi_{0}) \pd_{z} \Phi_{1} + \mathcal{M}_{1}(\Phi_{0}) \Phi_{1} \Phi_{0}') \right ]_{z = -\infty}^{z = +\infty} \\
- & \; \mu \int_{\R} \pd_{zz} \bm{U}_{1} \cdot (\pd_{z} \bm{Q}_{1} - \bm{a} (\mathcal{M}_{0}(\Phi_{0}) \pd_{z} \Phi_{1} + \mathcal{M}_{1}(\Phi_{0}) \Phi_{1} \Phi_{0}')) \dz.
\end{aligned}
\end{equation}

Using \eqref{assump:asymptotics:psi}, the matching conditions \eqref{e:MatchingCond1}, \eqref{e:MatchingCond2}, \eqref{e:MatchingCond3} for $\Phi_{0}$, and \eqref{e:MatchingCond3} for $\bm{Q}_{1}$ and $\bm{U}_{1}$, we see that the jump term is
\begin{equation}\label{adjoint1:innerfirstorder:integrated:jumpterm}
\begin{aligned}
\left [ \pd_{z} \bm{Q}_{1} \cdot \pd_{z} \bm{U}_{1} - \pd_{z} \bm{U}_{1} \cdot \bm{a} (\mathcal{M}_{0}(\Phi_{0}) \pd_{z} \Phi_{1} + \mathcal{M}_{1}(\Phi_{0}) \Phi_{1} \Phi_{0}') \right ]_{z = -\infty}^{z = +\infty} = [\pd_{\bm{\nu}} \bm{q}_{0} \cdot \pd_{\bm{\nu}} \bm{u}_{0}]_{-}^{+},
\end{aligned}
\end{equation}
since, in the case $\mathcal{M}(\varphi) = \frac{1}{2}$, we have $\mathcal{M}_{0}(\Phi_{0}) = \frac{1}{2}$ and $\mathcal{M}_{1}(\Phi_{0}) \Phi_{1} = 0$, while for the case $\mathcal{M}(\varphi) = \frac{1}{c_{0}} \sqrt{\frac{\psi(\varphi) + \delta_{\eps}}{2}}$, using \eqref{equipartition} and the matching conditions,  we have
\begin{align*}
& \; \left [ \pd_{z} \bm{U}_{1} \cdot \bm{a} (\mathcal{M}_{0}(\Phi_{0}) \pd_{z} \Phi_{1} + \mathcal{M}_{1}(\Phi_{0}) \Phi_{1} \Phi_{0}') \right ]_{z = -\infty}^{z = +\infty} \\
= & \; \frac{1}{\sqrt{2} c_{0}} \left [ \pd_{z} \bm{U}_{1} \cdot \bm{a} \left ( \sqrt{\psi(\Phi_{0})} \pd_{z}\Phi_{1} + \frac{\psi'(\Phi_{0}) \Phi_{1}}{\sqrt{2}} \frac{\Phi_{0}'}{\sqrt{2 \psi(\Phi_{0})}} \right ) \right ]_{z = -\infty}^{z = +\infty} = 0.
\end{align*}

Meanwhile, using \eqref{adjoint2:innerfirstorder} and \eqref{state1:innerfirstorder:simple}, the integral term is
\begin{equation}\label{adjoint1:innerfirstorder:integrated:integralterm}
\begin{aligned}
& \; \int_{\R} \mu \pd_{zz} \bm{U}_{1} \cdot (\pd_{z} \bm{Q}_{1} - \bm{a} (\mathcal{M}_{0}(\Phi_{0}) \pd_{z} \Phi_{1} + \mathcal{M}_{1}(\Phi_{0}) \Phi_{1} \Phi_{0}')) \dz \\
= & \; \int_{\R} - \pd_{z} P_{0} \bm{\nu} \cdot (\pd_{z} \bm{Q}_{1} - \bm{a} (\mathcal{M}_{0}(\Phi_{0}) \pd_{z} \Phi_{1} + \mathcal{M}_{1}(\Phi_{0}) \Phi_{1} \Phi_{0}')) \dz  = 0.
\end{aligned}
\end{equation}

Together with \eqref{pdzF0}, i.e., $\bm{F}_{0}$ is independent of $z$, we obtain from \eqref{adjoint1:innerfirstorder:integrated:jumpterm}, \eqref{adjoint1:innerfirstorder:integrated:integralterm} that \eqref{phase:innerfirstorder:unknownterm} is
\begin{equation}
\begin{aligned}
\int_{\R} \bm{F}_{0} \cdot \pd_{z} \bm{Q}_{0} + \hat{\alpha}(\Phi_{0}) \pd_{z} \bm{U}_{1} \cdot \bm{Q}_{0} \dz & = \bm{f}_{0} \cdot [\bm{q}_{0}]_{-}^{+} + \mu [\pd_{\bm{\nu}} \bm{q}_{0} \cdot \pd_{\bm{\nu}} \bm{u}_{0}]_{-}^{+} \\
& = \bm{f}_{0} \cdot \bm{a} + \mu \pd_{\bm{\nu}} \bm{q}_{0}^{+} \cdot \pd_{\bm{\nu}} \bm{u}_{0}^{+},
\end{aligned}
\end{equation}
as $\bm{q}_{0}^{-} = \bm{u}_{0}^{-} = \bm{0}$, and $\bm{q}_{0}^{+} = \bm{a}$ from \eqref{q0Omegaplus:complete}.  Thus, we obtain from \eqref{phase:innerfirstorder:integrated} the following solvability condition for $\Phi_{1}$:
\begin{equation*}
\begin{aligned}
2 \lambda_{0} + \kappa \gamma + \bm{f}_{0} \cdot \bm{a} + \mu \pd_{\bm{\nu}} \bm{q}_{0}^{+} \cdot \pd_{\bm{\nu}} \bm{u}_{0}^{+} = \bm{0} \text{ on } \Gamma.
\end{aligned}
\end{equation*}

\subsubsection{Sharp interface limit}
In summary, we obtain the following sharp interface limit:
\begin{subequations}\label{e:StateAndAdjointInLimit}\begin{align}
- \mu \Laplace \bm{u}_{0} + (\bm{u}_{0} \cdot \nabla) \bm{u}_{0} + \nabla p_{0} = \bm{f} & \text{ in } E, \\
- \mu \Laplace \bm{q}_{0} + (\nabla \bm{u}_{0})^{T} \bm{q}_{0} - (\bm{u}_{0} \cdot \nabla) \bm{q}_{0} + \nabla \pi_{0} = \bm{0} & \text{ in } E, \\
\div \bm{u}_{0} = 0, \quad \div \bm{q}_{0} = 0 & \text{ in } E, \\
\bm{u}_{0} = \bm{g}, \quad \bm{q}_{0} = \bm{0} & \text{ on } \pd \Omega \cap E, \\
\bm{u}_{0} = \bm{q}_{0} = \bm{0} & \text{ in } B,\\
\bm{u}_{0} = \bm{0},\quad \bm{q}_{0} = \bm{a} & \text{ on } \Gamma, \label{DirichletFreeBdy}
\end{align}\end{subequations}
together with the following gradient equation:
\begin{equation}\label{e:GradientEquationLimit}\begin{aligned}
\kappa \gamma + 2\lambda_{0} + \mu \pd_{\bm{\nu}} \bm{q}_{0} \cdot \pd_{\bm{\nu}} \bm{u}_{0} + \bm{f} \cdot \bm{a} = 0 \text{ on } \Gamma,
\end{aligned}
\end{equation}
which is consistent with the adjoint system \eqref{IntroAjointEquSharp} and the strong form of \eqref{ShapeDerivHydroDynamForce} from \cite{incoll:Boisgerault}, taking into account the volume constraint (see \eqref{generalh:ProofOptSysEpsVarEq}) and the additional perimeter regularization.

\begin{rem}[Linear scaling for the correction constant $\delta_{\eps}$]\label{rem:deltaepslinearscaling}
Suppose $\delta_{\eps} = \eps$, then we observe from \eqref{taylorexpansion} that
\begin{align*}
\sqrt{\psi(\Phi) + \eps} = \sqrt{\psi(\Phi_{0})} + \eps \frac{\psi'(\Phi_{0}) \Phi_{1} + 1}{2 \sqrt{\psi(\Phi_{0})}} + \mathrm{ h.o.t.},
\end{align*}
i.e., 
\begin{align*}
\mathcal{M}_{0}(\Phi_{0}) = \frac{1}{\sqrt{2} c_{0}} \sqrt{\psi(\Phi_{0})}, \quad \mathcal{M}_{1}(\Phi_{0})\Phi_{1} = \frac{1}{\sqrt{2} c_{0}}\frac{\psi'(\Phi_{0}) \Phi_{1} + 1}{2 \sqrt{\psi(\Phi_{0})}}.
\end{align*}
The presence of this extra factor of $\frac{1}{2 \sqrt{\psi(\Phi_{0})}}$ in $\mathcal{M}_{1}(\Phi_{0}) \Phi_{1}$ alters the jump term of \eqref{adjoint1:innerfirstorder:integrated} to
\begin{equation*}
\begin{aligned}
& \; \left [ \pd_{z} \bm{Q}_{1} \cdot \pd_{z} \bm{U}_{1} - \pd_{z} \bm{U}_{1} \cdot \bm{a} (\mathcal{M}_{0}(\Phi_{0}) \pd_{z} \Phi_{1} + \mathcal{M}_{1}(\Phi_{0}) \Phi_{1} \Phi_{0}') \right ]_{z = -\infty}^{z = +\infty} \\
= & \; [\pd_{\bm{\nu}} \bm{q}_{0} \cdot \pd_{\bm{\nu}} \bm{u}_{0}]_{-}^{+} - \frac{\bm{a}}{2 c_{0}} \cdot \left [ \frac{\Phi_{0}'}{\sqrt{2 \psi(\Phi_{0})}} \pd_{z} \bm{U}_{1} \right ]_{z=-\infty}^{z=+\infty} = \pd_{\bm{\nu}} \bm{q}_{0} \cdot \pd_{\bm{\nu}} \bm{u}_{0} - \frac{\bm{a}}{2 c_{0}} \cdot \pd_{\bm{\nu}} \bm{u}_{0},.
\end{aligned}
\end{equation*}
where we have used \eqref{equipartition}.  Thus, instead of \eqref{e:GradientEquationLimit}, we obtain\begin{align*}
\kappa \gamma + 2\lambda_{0} + \mu \pd_{\bm{\nu}} \bm{q}_{0} \cdot \pd_{\bm{\nu}} \bm{u}_{0} + \frac{\mu}{2 c_{0}} \pd_{\bm{\nu}} \bm{u}_{0} \cdot \bm{a} + \bm{f} \cdot \bm{a} = 0 \text{ on } \Gamma.
\end{align*}
\end{rem}

\section{Numerical computations}\label{sec:Numerics}
In this section we investigate the phase field approach numerically. We minimize the drag and maximize the lift-to-drag ratio of an obstacle in outer flow and apply both  phase field approximations of the corresponding surface functionals.
 
Concerning numerical results in the literature we refer to the minimization of the drag functional in \cite{article:SchmidtSchulz10, incoll:BrandenburgLindemannUlbrichUlbrich09}, where a sharp interface approach is used.  In \cite{article:Kondoh12} the porous medium approach is used, where the authors
argue, that the term $\alpha_{\eps}\bm{u}_\eps$ is a valid approximation for the
hydrodynamic force.

Let us start with defining the free energy $\psi$.  Here we use
\begin{equation}
\begin{aligned}
\tilde{\psi}(y) & = \frac{s}{2} \left({\max}^{2}(0,y - 1) + {\min}^{2}(0,y + 1)  \right ) + \frac{1}{2}(1 - y^{2}), \\
\psi(y) & = \tilde{\psi} \left(\frac{s}{s-1} y \right) + \frac{1}{2(s-1)}.
\end{aligned}
\end{equation}

Note that $\tilde{\psi}$ can be obtained by using a Moreau--Yosida relaxation of the double--obstacle free energy (\ref{doubleobstacle}) with the relaxation (or penalization) parameter $s \gg 1$, and the scaling of the argument and the shifting are chosen such that $\psi$ has its minima at $y = \pm 1$ with
$\psi(\pm 1) = 0$.

We further introduce the convex-concave splitting
\begin{align*}
\psi & = \psi_{+} + \psi_{-}, \\
\psi_{+}(y) & = \frac{s}{2} \left(
 {\max}^{2} \left(0,\frac{s}{s-1}y - 1 \right) +  {\min}^{2} \left(0,\frac{s}{s-1} y + 1 \right) \right),\\
\psi_{-}(y) & = \frac{1}{2} \left( 1 - \left(\frac{s}{s-1}y \right)^{2}\right) + \frac{1}{2(s-1)},
\end{align*}
where $\psi_{+}$ is the convex part of $\psi$ and $\psi_{-}$ is its concave part.

Next we define the interpolation function $\alpha_{\eps}$ as
\begin{align}
\alpha_{\eps}(y) = \frac{\overline{\alpha}}{\eps}
\begin{cases}
0  & \text{ if } y \geq 1,\\
\frac{1}{(1-\theta)(3+\theta)}(y-1)^{2}
& \text{ if } 1> \varphi \geq \theta,\\
\min \left( 1+\frac{2}{3+\theta},1-\frac{2}{3+\theta}(y+1) \right) & \text{ if } \theta > \varphi,
\end{cases}
\end{align}
where $\overline{\alpha}$ is a given constant, and we
choose $\theta = 0.99$.  This function $\alpha_{\eps}(y)$ describes a linear function between $y = -2$ and
$y = \theta$ and has a quadratic extension between
$y = \theta$ and $y = 1$.  We fulfill
Assumption \ref{assump:alpha} with $s_{a} = -2$ and $s_{b} = 1$.  Note that we do not fulfill
the regularity $\alpha_{\eps} \in C^{1,1}(\mathbb{R})$ at $s_{a}$.  But this is not a severe violation since in practice it holds that $-2 < \varphi_{\eps}$ and we
can control the violation of the bound $-1 \leq \varphi_{\eps}$ by choosing an appropriate relaxation parameter $s$.
 
For solving the optimization problem \eqref{IntroObjFcltPhase} we use a
mass conserving $H^{-1}$-gradient flow approach, following \cite{garckehinzeetal}.  For this purpose we introduce an artificial time variable $t$ and solve the following evolution equation for the phase field variable $\varphi_{\eps}(t)$ which is obtained from \eqref{generalh:PhaseFieldOptSys}:
\begin{equation}\label{eq:num:gradflow}
\begin{aligned}
\pd_{t} \varphi_{\eps} & = \Laplace w_{\eps}, \\
w_{\eps} & =  - \gamma \eps \Laplace \varphi_{\eps} + \frac{\gamma}{\eps} \psi'(\varphi_{\eps}) + \alpha'_{\eps}(\varphi_{\eps}) \left ( \frac{1}{2} \abs{\bm{u}_{\eps}}^{2} - \bm{u}_{\eps} \cdot \bm{q}_{\eps} \right ) + J_{\varphi},\\
J_{\varphi} & = \mathcal{M}'(\varphi_{\eps}) h(x, \nabla \bm{u}_{\eps}, p_{\eps}, \nabla \varphi_{\eps}) -\div \left( \mathcal{M}(\varphi_\eps) \der_{4}h(x, \nabla \bm{u}_{\eps}, p_{\eps}, \nabla \varphi_{\eps}) \right),
\end{aligned}
\end{equation}
where $\bm{u}_{\eps}$ is obtained from \eqref{IntroStateEquPhase}, $\bm{q}_{\eps}$ is obtained from \eqref{generalh:adjointsystem} and $J_\varphi$ abbreviates the terms arising from the differentiation of the functional $h$, as shown in Theorem \ref{t:GeneralisationOptimality}.  Note that we include the factor $\frac{1}{2c_0}$ into the parameter $\gamma$.  Using the gradient flow approach allows us to use nonlinear parts of the gradient, for example the derivative of $\psi_{+}$, implicitly in time in a time stepping scheme, which for the chosen free energy is favorable in view of
stability reasons.

After time discretization with variable time step size $\tau^{k+1}$ we at each
time instance solve the following problem:\\

Given $\varphi_{\eps}^{k}$, find $\varphi_{\eps}^{k+1}$, $w_{\eps}^{k+1}$, $\bm{u}_{\eps}$, $p_{\eps}$, $\bm{q}_{\eps}$, and $\pi_{\eps}$ fulfilling
the primal system
\begin{equation}\label{eq:num:TD_primal}
\begin{aligned}
\alpha_{\eps}(\varphi_{\eps}^{k}) \bm{u}_{\eps} - \mu \Laplace \bm{u}_{\eps} + (\bm{u}_{\eps} \cdot \nabla)\bm{u}_{\eps} + \nabla p_{\eps} & = \bm{f} && \text{ in }\Omega, \\
\div \bm{u}_{\eps} & = 0 && \text{ in }\Omega, \\
\bm{u}_{\eps} &= \bm{g} &&\text{ on } \pd \Omega,
\end{aligned}
\end{equation}
the adjoint system
\begin{equation}\label{eq:num:TD_adjoint}
\begin{aligned}
\alpha_{\eps} & (\varphi_{\eps}^{k})\bm{q}_{\eps} - \mu \div \left (\nabla \bm{q}_{\eps} + (\nabla \bm{q}_{\eps})^{T} \right) + (\nabla \bm{u}_{\eps})^{T} \bm{q}_{\eps} - (\bm{u}_{\eps} \cdot \nabla) \bm{q}_{\eps} + \nabla \pi_{\eps} \\
 & =  \alpha_{\eps}(\varphi_{\eps}^{k}) \bm{u}_{\eps} - \div \left (\mathcal{M}(\varphi_{\eps}^{k}) \der_{2} h \right) &&\text{ in }\Omega,\\
\div \bm{q}_{\eps} & = - \mathcal{M}(\varphi_{\eps}^{k})\der_{3} h + \vartheta_{\eps} && \text{ in }\Omega, \\
\bm{q}_{\eps} & = \bm{0} && \text{ on }\pd \Omega,
\end{aligned} 
\end{equation}
and the Cahn--Hilliard system
\begin{equation}\label{eq:num:TD_phase}
\begin{aligned}
\varphi_{\eps}^{k+1} & = \tau^{k+1} \Laplace w_{\eps}^{k+1} + \varphi_{\eps}^{k} && \text{ in }\Omega,\\
w_{\eps}^{k+1} & =  - \gamma \eps \Laplace \varphi_{\eps}^{k+1} + \frac{\gamma}{\eps} \left(\psi'_{+}(\varphi_{\eps}^{k+1}) + \psi'_{-}(\varphi_{\eps}^{k})\right) \\
& + \frac{1}{2} \alpha'_{\eps}(\varphi_{\eps}^{k+1})    \abs{\bm{u}_{\eps}}^{2} - \alpha'_{\eps}(\varphi_{\eps}^{k}) \bm{u}_{\eps} \cdot \bm{q}_{\eps}    + J_{\varphi} && \text{ in }\Omega,\\
J_{\varphi} & = \mathcal{M}'(\varphi_{\eps}^{k}) h(x, \nabla \bm{u}_{\eps}, p_{\eps}, \nabla    \varphi_{\eps}^{k+1})\\
& - \div \left( \mathcal{M}(\varphi_{\eps}^{k}) \der_{4}h(x, \nabla \bm{u}_{\eps}, p_{\eps}, \nabla \varphi_{\eps}^{k+1}) \right),\\
0 & = \gamma\eps \nabla \varphi_{\eps}^{k+1} \cdot \bm{\nu}_{\pd \Omega} 
+ \mathcal M(\varphi_\eps^{k})\bm \nu_{\pd \Omega} \cdot
D_4h &&\text{ on } \pd \Omega,\\ 
0 &=   \nabla w_{\eps}^{k+1}\cdot \bm{\nu}_{\pd \Omega}  &&\text{ on } \pd
\Omega.
\end{aligned}
\end{equation}
As noted above, we evaluate $\psi'_{+}$ at the new time instance for stability reasons.  
 
For the spatial discretization piecewise linear and globally continuous finite
elements are used for the variables $\varphi_{\eps}^{k+1}$, $w_{\eps}^{k+1}$, $p_{\eps}$, and $\pi_{\eps}$, while piecewise quadratic and globally continuous elements are used for $\bm{u}_{\eps}$ and $\bm{q}_{\eps}$.  The meshes are adapted using the jumps of the normal derivative of $\varphi_{\eps}^{k+1}$ and $w_{\eps}^{k+1}$ over edges of the underlying
discretization mesh, see  \cite{article:CarstensenVerfuerth99,
book:Verfuerth_Adaptivity}, together
with a D\"{o}rfler marking \cite{article:Doerfler96}.

\subsection{Minimization of the hydrodynamic force of an obstacle}\label{ssec:num:minHydro}
We investigate the minimization of the drag of an obstacle of fixed area in a channel flow with block inflow profile.
 
The computational domain is $\Omega = (0,1.7) \times (0,0.4)$. The initial phase field $\varphi^{0}$ is defined as a circle of radius $r = 0.05$ with center at $M = (0.5,0.2)$.  The boundary velocity is set to
$\bm g(x,y) = (1,0)^{T}$.  We fix $\delta_{\eps} = 0$,  $s= 1 \times 10^{6}$, and $\bm{f} \equiv \bm{0}$.  We further set 
\begin{align*}
\tau^{k+1} := \xi \min_{T} (h_{T} \norm{\nabla w_{\eps}^{k}}_{L^{2}(T)}^{-1}),
\end{align*}
where the minimization is carried out over all triangles $T$.  Here, the diameter of triangle $T$ is denoted by $h_{T}$, and $\xi$ is a positive scaling parameter typically set to $\xi=5$.  This CFL-like condition prevents the interfacial region from moving too fast for the adaptation process.

We restate the definition of the phase field approximation of the hydrodynamic force in a direction $\bm{a}$ as
\begin{align}\label{eq:num:hydroForce}
F^{\bm{a}} := \int_{\Omega} \mathcal{M}(\varphi_{\eps}) \nabla \varphi_{\eps} \cdot \left(
\mu \left( \nabla \bm{u}_{\eps} + (\nabla \bm{u}_{\eps})^{T} \right) - p_{\eps} \id \right) \cdot \bm{a} \dx. 
\end{align}
When $\bm{a}$ is equal to the direction of the flow, i.e., $\bm{a} = (1,0)^{T}$, we denote the resulting approximation as $F^{D}$, which corresponds to the drag of the obstacle.  Meanwhile, if $\bm{a}$ is perpendicular to the direction of the flow, i.e., $\bm{a} = (0,1)^{T}$, then we denote the resulting approximation as $F^{L}$, which corresponds to the lift of the obstacle.

From \eqref{e:AdjointStrong} and \eqref{e:PhaseFieldGradientEquStrongRewrite}, the terms arising from the derivatives of $h$ in systems \eqref{eq:num:TD_adjoint} and \eqref{eq:num:TD_phase} in the present setting are given as
\begin{align*}
(- \div \left (\mathcal{M}(\varphi_{\eps}^{k}) \der_{2}h \right),\bm{v})& = \mu \int_{\Omega} \mathcal{M}(\varphi_{\eps}^{k}) \nabla \varphi_{\eps}^{k} \cdot \left(\nabla \bm{v} + (\nabla \bm{v})^{T} \right) \bm{a} \dx && \forall \bm{v} \in \bm{H}^{1}_{0}(\Omega), \\
(-\mathcal{M}(\varphi_{\eps}^{k}) \der_{3}h + \vartheta_{\eps}, \eta) &= \int_{\Omega}
 \left( \mathcal{M}(\varphi_{\eps}^{k}) \nabla  \varphi_{\eps}^{k} \cdot \bm{a} - \strokedint_{\Omega}
\mathcal{M}(\varphi_{\eps}^{k}) \nabla \varphi_{\eps}^{k} \cdot \bm{a} \dx \right) \eta \dx && \forall \eta \in L^{2}_{0}(\Omega), \\
(J_{\varphi}, \zeta) &= \int_{\Omega} \mathcal{M}(\varphi_{\eps}^{k}) \left( - \alpha_{\eps}(\varphi_{\eps}^{k}) \bm{u}_{\eps} - (\bm{u}_{\eps} \cdot \nabla ) \bm{u}_{\eps} \right) \cdot \bm{a} \zeta \dx && \forall \zeta \in H^{1}(\Omega).
\end{align*}

Next, we report on the numerical results for the case of minimizing $F^{D}$.  The parameters are chosen as
$\eps = 0.00025$, $\overline{\alpha} = 0.03$, $\mu = 0.001$, and $\gamma=0.01$.  We note that we use path-following with respect to the value of $\mu$, starting from $\mu=0.01$, and also for the value of $\gamma$, starting from $\gamma=0.1$.  In Figure \ref{fig:num:results_drag} we show results obtained with our approach.

\begin{figure} 
  \centering
\includegraphics[height=3cm]{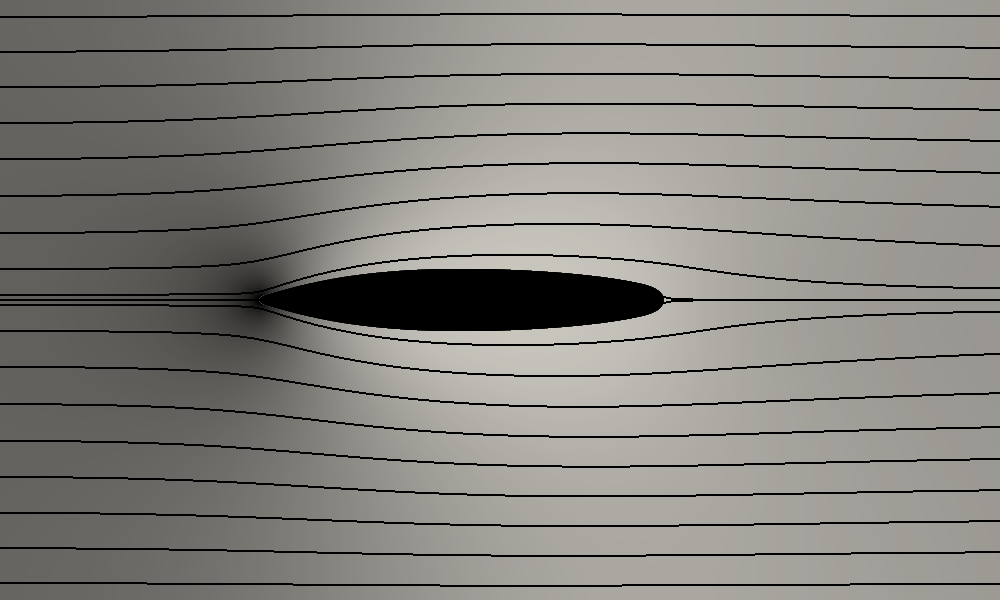}
  \includegraphics[height=3cm]{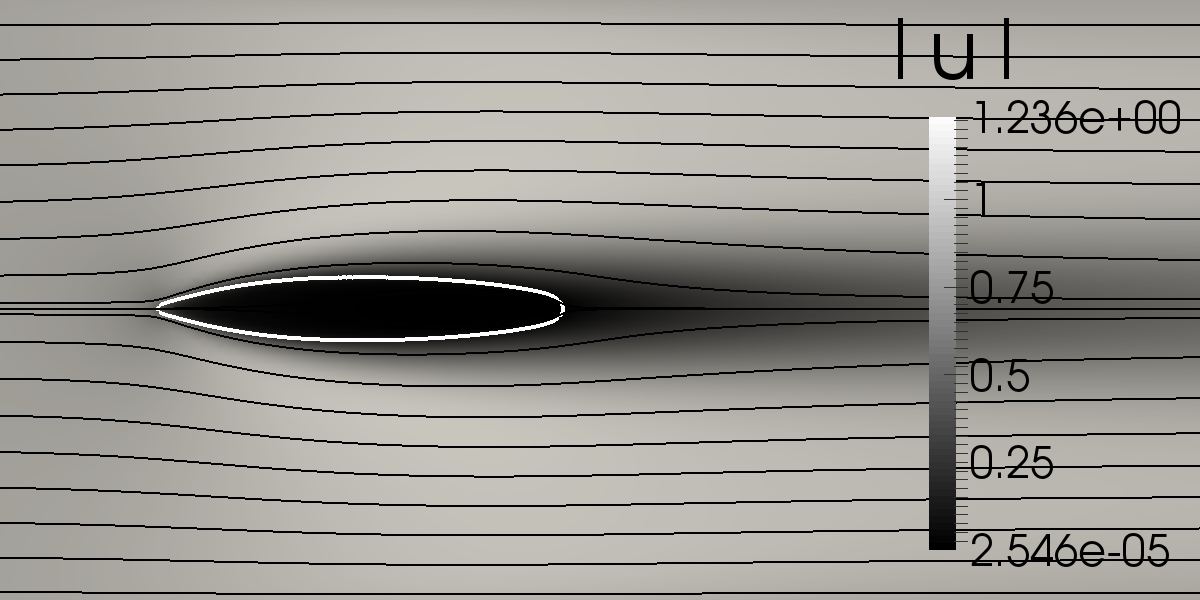}
\caption{Result for minimizing the drag using $\mathcal{M}(\varphi_{\eps}) = \frac{1}{c_{0}}\sqrt{\frac{\psi(\varphi_{\eps})}{2}}$.  In the left plot we show the obstacle (i.e., $\varphi_{\eps} \leq 0$) and streamlines of $\bm{u}_{\eps}$ in black, and the pressure outside of the obstacle in gray.  Darker gray means higher pressure.  On the right we show $\abs{\bm{u}_{\eps}}$ in gray, where darker gray means lower velocity.  The isoline $\varphi_{\eps} \equiv 0$ is shown in white and again streamlines are displayed in black.  The results for $\mathcal  M(\varphi_{\eps}) = \frac{1}{2}$ are visually indistinguishable from these results.  Note that we only show the computational domain in the neighbourhood of the obstacle.}
\label{fig:num:results_drag}
\end{figure}
 
The drag for $\mathcal{M}(\varphi_{\eps}) =
\frac{1}{c_0}\sqrt{\frac{\psi(\varphi_{\eps})}{2}}$
is given by $F^{D} = 3.9454 \times 10^{-2}$ $(3.9492 \times 10^{-2})$, and for $\mathcal{M}(\varphi_{\eps}) = \frac{1}{2}$ we have $F^{D} = 3.9117 \times 10^{-2}$ $(3.9499 \times 10^{-2})$.  In brackets we give the drag obtained by evaluating the surface formulation over the isoline $\varphi_{\eps} \equiv 0$.  We see that both formulations give very similar results.

\subsection{Maximization of the lift-drag ratio of an obstacle}\label{ssec:maxLD}
Based on the results of the previous section we now investigate the maximization of the lift-to-drag ratio given by
\begin{align*}
R := F^{L} / F^{D},
\end{align*}
To this end, we consider
\begin{align*}
  \int_\Omega \mathcal M(\varphi_{\eps} )h(x, \nabla \bm{u}_{\eps}, p_{\eps}, \nabla \varphi_{\eps}) \dx := - \frac{ \int_{\Omega} \mathcal{M}(\varphi_{\eps}) \nabla \varphi_{\eps} \cdot (\mu (\nabla \bm{u}_{\eps} + (\nabla \bm{u}_{\eps})^{T}) - p_{\eps} \id) \bm{a}^{\perp} \dx}{\int_{\Omega} \mathcal{M}(\varphi_{\eps}) \nabla \varphi_{\eps} \cdot (\mu(\nabla \bm{u}_{\eps} + (\nabla \bm{u}_{\eps})^{T}) - p_{\eps} \id) \bm{a} \dx},
\end{align*}
with $\bm{a} = (1,0)^{T}$ and $\bm{a}^{\perp} = (0,1)^{T}$.

The numerical setup is the same as in the previous section and the parameters are chosen as $\eps = 0.0005$, $\overline{\alpha} = 4$, $\mu = 1/15$,
and $\gamma=0.3$.  In this example we fix the y-coordinate of the center of mass of the obstacle
by a Lagrange multiplier approach in order to keep it fixed at the initial position.  We define the center of mass of the obstacle as
\begin{align*}
\mathrm{com} =  \frac{ \int_{\Omega} \frac{1-\varphi_{\eps}}{2} x \dx}{  \int_{\Omega} \frac{1-\varphi_{\eps}}{2}\dx}.
\end{align*}
In Figure \ref{fig:num:results_liftdrag} we show results for this parameter set.

\begin{figure}
\centering
\includegraphics[height=3cm]{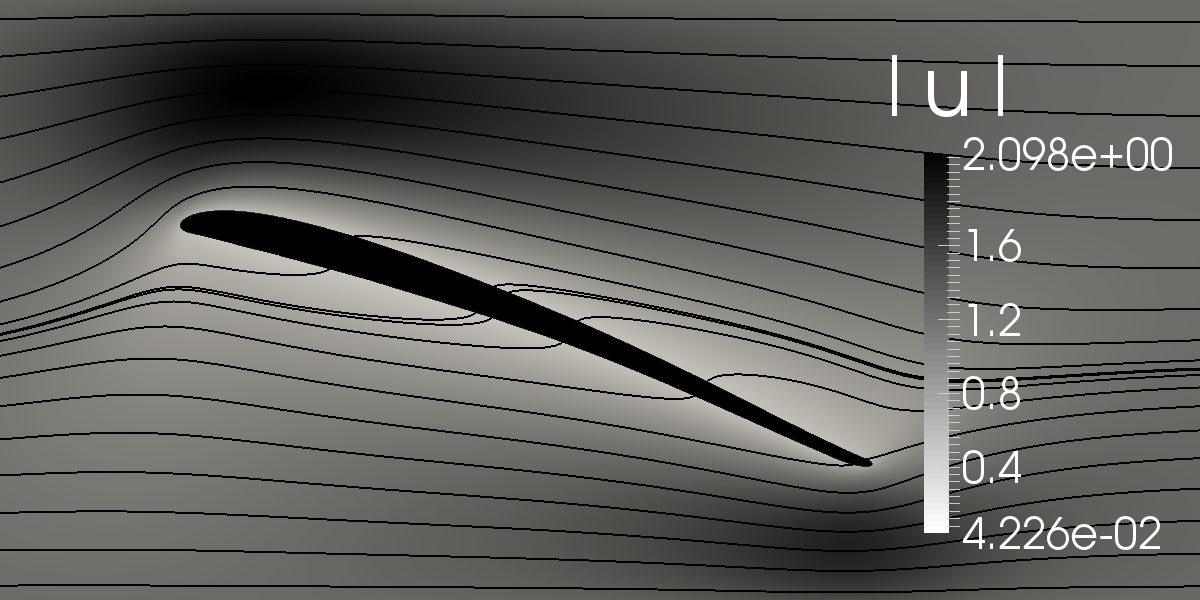}
\includegraphics[height=3cm]{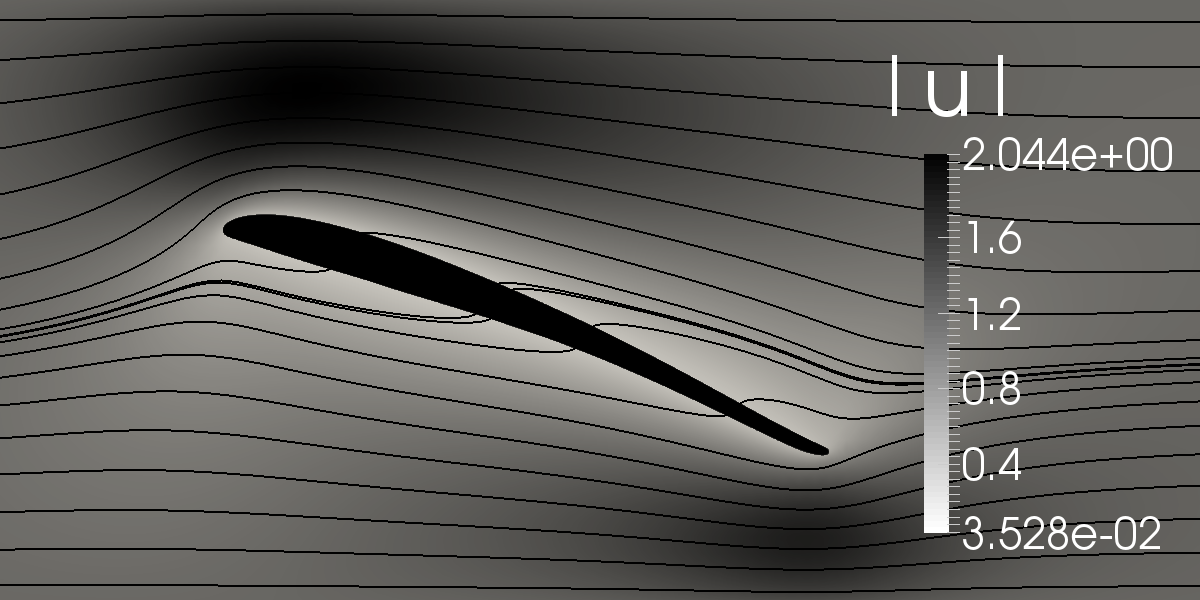}
\caption{Result for maximizing the lift-to-drag ratio using $\mathcal{M}(\varphi_{\eps}) = \frac{1}{c_{0}}\sqrt{\frac{\psi(\varphi_{\eps})}{2}}$ (left) and $\mathcal{M}(\varphi_{\eps}) = \frac{1}{2}$ (right).   The obstacle (i.e., $\varphi_{\eps} \leq 0$) and streamlines are shown in black and the velocity magnitude in gray.  Darker gray means larger velocity.   Note that we only show the computational domain in the neighbourhood of the obstacle.}
\label{fig:num:results_liftdrag}
\end{figure}
  
We observe the expected optimal shape for both formulations, but for $\mathcal{M}(\varphi_{\eps}) 
= \frac{1}{c_{0}}\sqrt{\frac{\psi(\varphi_{\eps})}{2}}$ we obtain a longer and thinner obstacle. 
 
The lift-to-drag ratio for
$\mathcal{M}(\varphi_{\eps}) =
\frac{1}{c_{0}}\sqrt{\frac{\psi(\varphi_{\eps})}{2}}$ is $R = 1.1104$, and for 
$\mathcal{M}(\varphi_{\eps}) = \frac{1}{2}$ it is $R= 0.9885$.  We stress that, here we calculate with a rather small value of $\mu = 1/15$ and that the minimal magnitude of velocity inside the obstacle is $4 \times 10^{-2}$, which is rather large.  However, we think that the results are a promising starting point for further investigations.


\end{document}